\documentclass[11pt]{amsart}
\usepackage{graphicx,amsfonts,amssymb,amsmath,amsthm,url,
  verbatim,amscd}
  \usepackage{color}
\usepackage[usenames,dvipsnames]{xcolor}
\usepackage[normalem]{ulem}
\usepackage{pdfsync}

\usepackage[hyperfootnotes=false, colorlinks, citecolor=RoyalBlue, urlcolor=blue, linkcolor=blue ]{hyperref}

\theoremstyle{plain} %
\newtheorem{theorem}             {Theorem}  [section]
\newtheorem{lemma}      [theorem]{Lemma}
\newtheorem{corollary}  [theorem]{Corollary}
\newtheorem{proposition}[theorem]{Proposition}

\newtheorem{observation}   [theorem]{Observation}
\theoremstyle{definition}
\newtheorem{definition} [theorem]{Definition}
\newtheorem{condition}  [theorem]{Condition}
\newtheorem{example}    [theorem]{Example}

\theoremstyle{remark}
\newtheorem{remark}  [theorem]            {Remark}

\allowdisplaybreaks

\newcommand{\Q}{{\mathbb Q}}
\newcommand{\Z}{{\mathbb Z}}
\newcommand{\N}{{\mathbb N}}
\newcommand{\R}{{\mathbb R}}
\newcommand{\C}{{\mathbb C}}

\newcommand{\p}{\mathfrak p}
\newcommand{\OF}{{\mathfrak o}}
\newcommand{\GL}{{\rm GL}}
\newcommand{\PGL}{{\rm PGL}}
\newcommand{\SL}{{\rm SL}}

\newcommand{\sgn}{{\rm sgn}}
\newcommand{\St}{{\rm St}}

\newcommand{\ad}{{\rm ad}}

\newcommand{\Gal}{{\rm Gal}}

\newcommand{\Hom}{{\rm Hom}}

\newcommand{\mat}[4]{{\setlength{\arraycolsep}{0.5mm}\left[
      \begin{array}{cc}#1&#2\\#3&#4\end{array}\right]}}

\def\Jfcomplete{{J_f^*}}
\def\Jcomplete{{J^*}}
\def\Icomplete{{I\subscriptp^*}}
\def\Iinfcomplete{{I_\infty^*}}
\def\Ipcomplete{{I_p^*}}
\def\Ivcomplete{{I_v^*}}
\def\onehalf{{1/2}}
\def\onehalfbig{{1/2}}
\def\twist{{}}
\def\subscriptp{{}}
\def\fin{\operatorname{fin}}

\def\Stab{\operatorname{Stab}}
\def\SL{\operatorname{SL}}
\def\vol{\operatorname{vol}}
\def\GL{\operatorname{GL}}

\def\res{\operatorname{res}}
\renewcommand{\Re}{\mathrm{Re}}
\renewcommand{\Im}{\mathrm{Im}}
\def\eps{\varepsilon}

\begin{document}

\bibliographystyle{plain}

\title[Rankin--Selberg integrals and quantum unique ergodicity]{Bounds for Rankin--Selberg integrals and quantum unique ergodicity for powerful levels}

\thanks{The first author is supported
  by NSF grant OISE-1064866
  and partially supported by grant SNF-137488. The second author is supported by NSF grant DMS 1100541}

\subjclass[2010]{Primary 11F11; Secondary 11F70, 22E50, 58J51}

\author{Paul D. Nelson}
\address{EPFL, Station 8, CH-1015 Lausanne, Switzerland}
\email{paul.nelson@epfl.ch}

\author{Ameya Pitale}
\address{Department of Mathematics
  \\ University of Oklahoma\\ Norman\\
  OK 73019, USA,}
\email{apitale@math.ou.edu}

\author{Abhishek Saha}
\address{Department of Mathematics\\
  University of Bristol\\
  Bristol BS81TW\\
  UK} \email{abhishek.saha@bris.ac.uk}

\begin{abstract}
  Let $f$ be a classical holomorphic newform
  of level $q$ and even weight $k$.
    We show that the pushforward to the full level modular curve
  of the mass of $f$ equidistributes as $q k \rightarrow
  \infty$. This generalizes known results
  in the case that $q$ is squarefree.
  We obtain a power savings in the rate of equidistribution
  as $q$ becomes sufficiently ``powerful'' (far away from being
  squarefree), and in particular
  in the ``depth aspect'' as $q$ traverses the powers  of a fixed prime.

  We compare the difficulty of such equidistribution problems
  to that of corresponding subconvexity problems
  by deriving explicit extensions of Watson's
  formula
  to certain triple product integrals involving forms of
  non-squarefree
  level.
  By a theorem of Ichino and a lemma of Michel--Venkatesh,
  this amounts to a detailed study of Rankin--Selberg integrals
  $\int |f|^2 E$ attached to newforms $f$ of arbitrary
  level and Eisenstein series $E$ of full level.

  We find that the local factors of such integrals participate in many amusing
  analogies with global $L$-functions.
  For instance,
  we observe that the mass equidistribution conjecture
  with a power savings in the depth aspect
  is equivalent to knowing either a global subconvexity bound
  or what we call a ``local subconvexity bound'';
  a consequence of our local calculations is what
  we call a ``local Lindel\"{o}f hypothesis".

\end{abstract}
\maketitle

\def\eachnotany{~{each}~}

\section{Introduction}
\label{sec:introduction}
\subsection{Main result}
\label{sec:main-result}
Let $f : \mathbb{H} \rightarrow \mathbb{C}$ be a classical holomorphic newform
of weight $k \in 2 \mathbb{N}$ on $\Gamma_0(q)$, $q \in
\mathbb{N}$
(see Section \ref{sec:modular-forms-their}
for definitions).
The pushforward to $Y_0(1) = \SL_2(\mathbb{Z}) \backslash \mathbb{H}$
of the $L^2$-mass of $f$ is the finite measure given by
\[
\mu_f(\phi) = \int _{\Gamma_0(q) \backslash \mathbb{H} }
y ^k \lvert f \rvert ^2 (z) \phi (z) \, \frac{d x \, d y }{ y ^2 }
\]
for each bounded measurable function $\phi$ on $Y_0(1)$.
Its value $\mu_f(1)$ at the constant function $1$ is
(one possible normalization of) the
Petersson norm of $f$.
Let $d \mu(z) = y^{-2} d x \, d y$ denote the standard
hyperbolic volume measure on $Y_0(1)$,
and let
\[
D_f(\phi) :=
\frac{\mu_f(\phi)}{\mu_f(1)}
-
\frac{\mu(\phi)}{\mu(1)}.
\]
The quantity $D_f(\phi)$ compares
the probability measures attached to $\mu_f$ and $\mu$
against a test function $\phi$.

The problem of bounding $D_f(\phi)$
for fixed $\phi$
as the parameters of $f$ vary
is a natural analogue of the Rudnick--Sarnak
quantum unique ergodicity conjecture \cite{MR1266075}. It was raised explicitly in the $q = 1$, $k \rightarrow \infty$ aspect
by Luo--Sarnak \cite{luo-sarnak-mass}
and in the $k=\text{constant}$, $q \rightarrow \infty$ aspect by Kowalski--Michel--VanderKam
\cite{KMV02};
in each case it was conjectured that $D_f(\phi) \rightarrow 0$.
Such a conjecture is reasonable because a theorem of Watson
\cite{watson-2008} and subsequent generalizations
(see Sections \ref{sec:equid-vs-subconvexity} and
\ref{sec:watson})
have shown that it follows in many
cases
from the (unproven) Generalized Lindel\"{o}f Hypothesis,
itself a consequence of the Generalized Riemann Hypothesis.
The first unconditional
result for general (non-dihedral)
$f$ was obtained
by
Holowinsky and Soundararajan \cite{holowinsky-soundararajan-2008},
who showed that
$D_f(\phi) \rightarrow 0$ for fixed $q$ ($=1$) and varying $k
\rightarrow \infty$; we refer to their paper
and \cite{sarnak-progress-que}
for further historical background.
The case of varying squarefree levels
was addressed in \cite{PDN-HQUE-LEVEL},
where it was shown that
$D_f(\phi) \rightarrow 0$ as $qk \rightarrow \infty$
provided that $q$ is squarefree.

Our aim in this paper is to address the remaining case
in which the varying level $q$ need not be squarefree.
We obtain the expected result,
thereby settling
the remaining cases
of the conjecture in \cite{KMV02}:

\begin{theorem}
  \label{thm:main}
  Fix a bounded continuous function $\phi$  on $Y_0(1)$.
  Let $f$ traverse a sequence of holomorphic newforms of weight $k$ on $\Gamma_0(q)$ with $k \in 2 \mathbb{N}$, $q \in \mathbb{N}$. Then
  $D_f(\phi) \rightarrow 0$
  whenever $q k \rightarrow \infty$.
\end{theorem}
Theorem \ref{thm:main} is a consequence of the following more precise
result and a standard approximation argument
(see Section \ref{sec:proof-thm-1} and \cite[Section
1.6]{PDN-HQUE-LEVEL}).
\begin{theorem}
  \label{thm:main-refined}
  Fix a Maass eigencuspform
  or incomplete Eisenstein series $\phi$ on $Y_0(1)$.
  Let $f$ traverse a sequence of holomorphic newforms
  of weight $k$ on $\Gamma_0(q)$ with $k \in 2 \mathbb{N}$, $q
  \in \mathbb{N}$.
  There exist effective positive constants $\delta_1, \delta_2$
  so that\footnote{We use the notation
$A \ll_{x,y,z} B$
to signify that there exists
a positive constant $C$, depending at most upon $x,y,z$,
so that
$|A| \leq C |B|$.}
  \begin{equation}\label{eq:22}
    D_f(\phi)
    \ll_\phi (q/q_0)^{-\delta_1} \log(q k)^{-\delta_2},
  \end{equation}
  where $q_0$ denotes the largest squarefree divisor of $q$.\footnote{If $q$ has the prime factorization $q = \prod_p p^{a_p}$, then $q_0$ has the prime factorization $q_0 = \prod_p p^{\min(a_p, 1)}$.}
\end{theorem}

A potentially surprising aspect
of Theorem \ref{thm:main-refined}
is the unconditional power
savings
in the rate of equidistribution
when $q/q_0$ grows faster than a certain fixed
power of $\log(q_0 k)$,
or in words, when the level is sufficiently \emph{powerful}.
A special case that illustrates the new phenomena is the \emph{depth aspect},
in which $k$ is fixed and $q = p^n$ is the power of a fixed
prime $p$
with $n \rightarrow \infty$.

By contrast, suppose that $q$ is squarefree,
so that $q = q_0$.
Then the logarithmic rate of decay $D_f(\phi) \ll_{\phi}
\log(qk)^{-\delta_2}$
in Theorem \ref{thm:main-refined}
is consistent with that obtained
in \cite{holowinsky-soundararajan-2008,PDN-HQUE-LEVEL},
and the problem of improving
this logarithmic decay
to a power savings $D_f(\phi) \ll_\phi (qk)^{-\delta_3}$
($\delta_3 > 0$)
is equivalent
to the (still open) subconvexity problem for certain fixed $\GL(1)$ or $\GL(2)$
twists
of the adjoint lift of $f$ to $\GL(3)$ (see Section
\ref{sec:equid-vs-subconvexity}).

Explaining this ``surprise'' is a major theme of this paper.
It amounts to a detailed study of certain Rankin--Selberg
zeta integrals $J_f(s)$ arising as proportionality constants
in a formula for $D_f(\phi)$ given by Ichino \cite{MR2449948},
as simplified by a lemma of Michel--Venkatesh \cite[Lemma
3.4.2]{michel-2009}. In classical terms, $J_f(s)$ is proportional uniformly for $\Re(s) \ge \delta > 0$ to the meromorphic continuation of the ratio \begin{equation}\label{eq:29}
    \frac{1}{[\Gamma_0(q):\Gamma_0(1)]}
  \frac{
    \int_{\Gamma_0(q) \backslash \mathbb{H}}
    y^k |f|^2(z) \left( \sum_{\gamma \in \Gamma_\infty \backslash
        \Gamma_0(1)} (\Im \gamma z)^s \right)
    \, \frac{d x \, d y }{y^2}
  }{
    \int_{\Gamma_0(q) \backslash \mathbb{H}}
    y^k |f|^2(z) \left( \sum_{\gamma \in \Gamma_\infty \backslash
        \Gamma_0(q)} (\Im \gamma z)^s \right)
    \, \frac{d x \, d y }{y^2}
  },
\end{equation} defined initially for $\Re(s) > 1$. The quantity $J_f(s)$
factors as a product over the primes
dividing the level:
$$J_f(s) = \prod_{p | q}J_p(s),$$
with each $J_p(s)$ a $p$-adic zeta integral (see~\eqref{RS-integral-defn2}) that differs mildly from a polynomial
function of $p^{\pm s}$
and satisfies a functional equation
under $s \mapsto 1-s$.

We find the analytic properties
of such integrals
to be
unexpectedly rich and to participate
in many amusing analogies.
For instance, we show that the problem of obtaining a positive value of $\delta_1$ in
Theorem
\ref{thm:main-refined} is equivalent to knowing \emph{either} a
``global'' subconvex bound for an $L$-value
\emph{or} what we call a \emph{local subconvex bound}
for $J_f(s)$ (see e.g. Observation \ref{thm:que-versus-local-1}).
The main technical result of this paper is a
proof of
(what we call) the \emph{local Lindel\"{o}f hypothesis}
for $J_f(s)$, which, naturally, saves
nearly a factor of $q^{1/4}$  over
the \emph{local
  convexity bound} on the critical line $\Re(s) = \onehalfbig$
(see Section \ref{sec:local-lind-hypoth}).
We observe numerically that $J_f(s)$ seems to satisfy
a \emph{local Riemann hypothesis}
(see Section \ref{sec:local-riem-hypoth}), the significance of which remains unclear to us.

\begin{remark}
  We comment on the nature of the constants
  $\delta_1,\delta_2$ appearing in Theorem \ref{thm:main-refined}.
  One may choose $\delta_2$ very explicitly
  as in \cite{holowinsky-soundararajan-2008,PDN-HQUE-LEVEL},
  while $\delta_1$ depends upon a bound
  $\theta \in [0,7/64]$ (see~\cite{MR1937203})
  towards the Ramanujan conjecture for Maass forms on
  $\SL_2(\mathbb{Z}) \backslash \mathbb{H}$,
  with
  any improvement over the trivial bound
  $\theta \leq \onehalfbig$
  sufficing
  to yield a positive value of $\delta_1$.
  For example, in the simplest case that
  $q = p^{2 m}$ is an even
  power of a prime (the ``even depth aspect''),
  our method leads to the bound
  \[
  D_f(\phi)
  \ll_{k}
  m^{O(1)}
  (p^m)^{-1/2+ \theta}
  \ll_{k,\eps}
  (p^m)^{-1/2+ \theta + \eps}.
  \]
  Our calculations show that the Ramanujan conjecture for Maass
  forms together with the Lindel\"{o}f
  hypothesis
  for fixed $\GL(1)$ and $\GL(2)$ twists of the adjoint lift of $f$
  would imply the stronger bound
  $D_f(\phi) \ll_{\eps,k} (p^m)^{-1+\eps}$,
  which should be optimal\footnote{That is to say, it should be the optimal bound that holds for \emph{all} $f$ of level $p^{2m}$. Stronger bounds will hold, for instance, for ramified character twists of forms of lower level.} as far as the exponent is concerned.
\end{remark}

Our paper is organized as follows.
The remainder of Section \ref{sec:introduction}
is an extended introduction that explains the main
ideas of our work.
In Section \ref{sec:local-calculations}, we undertake
a detailed study of the local Rankin--Selberg integral
attached to a spherical Eisenstein series
and the $L^2$-mass of a newform of arbitrary level. Our calculations yield an
explicit extension of Watson's formula (see Theorem~\ref{thm:watson-ext}) to certain collections of newforms of
not necessarily squarefree level.
In Section \ref{sec:proof-theor-refthm:m},
we study the Fourier coefficients of highly ramified
newforms
at arbitrary cusps of $\Gamma_0(q)$ (see Section \ref{sec:four-expans-at} for an overview)
and apply a variant of the Holowinsky--Soundararajan method
to deduce Theorem \ref{thm:main-refined}.

The results of Section \ref{sec:local-calculations}
suffice on their own to imply Theorem \ref{thm:main-refined}
when the level $q$ is sufficiently powerful (e.g.,
if $q = p^n$ with $p$ fixed and $n \rightarrow \infty$).
At the other extreme, Theorem \ref{thm:main-refined} is already
known when $q$ is squarefree (see \cite{PDN-HQUE-LEVEL}).
It is the myriad of intermediate possibilities
(e.g., when $q = q_0 p^n$ is the product of a large squarefree integer $q_0$ and a large prime power $p^n$) that justifies
Section \ref{sec:proof-theor-refthm:m}.

\subsection{Equidistribution vs. subconvexity}
\label{sec:equid-vs-subconvexity}
The motivating \emph{quantum unique ergodicity} (QUE) conjecture,
put forth by Rudnick and Sarnak,
predicts that the $L^2$-normalized Laplace eigenfunctions
$\phi$
on a negatively curved compact Riemannian manifold
have equidistributed $L^2$-mass in the large eigenvalue limit.
The \emph{arithmetic QUE} conjecture
concerns the special case that $\phi$
traverses a sequence of
joint Hecke-Laplace eigenfunctions on an arithmetic manifold.
A formula of Watson showed in many cases
that the arithmetic QUE conjecture for surfaces,
in a sufficiently strong quantitative form,
is \emph{equivalent} to a case of the central
\emph{subconvexity problem} in the analytic theory
of $L$-functions.
A principal motivation for this work was to investigate the extent to which
this equivalence survives the passage to variants
of arithmetic QUE not covered by Watson's formula.

In the prototypical case
that $f$ is a Maass eigencuspform on $Y_0(1)$ with
Laplace eigenvalue
$\lambda$,
the definitions of $\mu_{f}$
and $D_{f}$ given in Section \ref{sec:main-result}
still make sense (take $k=0$),
and the equidistribution problem is to improve upon the
trivial bound
\begin{equation}\label{eq:27}
  D_{f}(\phi) \ll_{\phi} 1
\end{equation}
for the \emph{period} $D_f(\phi)$ in the $\lambda \rightarrow \infty$ limit.
Watson's
formula
implies that if $\phi$ is a fixed Maass eigencuspform
on $Y_0(1)$,
then $D_{f}(\phi)$ is closely related
to a central \emph{$L$-value}:
\begin{equation}\label{eq:24}
  \left\lvert D_{f}(\phi) \right\rvert^2 = \lambda^{-1+o(1)}
  L(f  \times f  \times \phi, \onehalf).
\end{equation}

For quite general (finite parts of) $L$-functions
$L(\pi,s)$,
which we always normalize to satisfy a functional equation
under $s \mapsto 1-s$,
there is a commonly accepted notion of a
trivial bound
for the central value $L(\pi,\onehalf)$.
It is called the \emph{convexity bound},
and takes the form $L(\pi,\onehalf) \ll C(\pi)^{1/4+o(1)}$
where $C(\pi) \in \mathbb{R}_{\geq 1}$ is the analytic
conductor attached to $\pi$ by Iwaniec--Sarnak \cite{MR1826269}.
The \emph{subconvexity problem} is to improve
this to $L(\pi,\onehalf) \ll C(\pi)^{1/4-\delta}$ for some positive
constant $\delta$,
while the Grand Lindel\"{o}f Hypothesis --- itself a consequence
of the Grand Riemann Hypothesis --- predicts the
sharper bound $L(\pi,\onehalf) \ll C(\pi)^{o(1)}$.
The subconvexity problem remains open in general for the
triple product $L$-functions considered in this paper.
We refer to \cite{MR1826269,MR1321639,sarnak-progress-que}
for further background.

For the $L$-value appearing in
\eqref{eq:24}, the convexity bound
reads
\begin{equation}\label{eq:25}
  L(f  \times f \times \phi,\onehalf) \ll_{\phi} \lambda^{1 +
    o(1)}.
\end{equation}
Thus under the correspondence between periods
and $L$-values afforded by Watson's formula \eqref{eq:24},
the trivial bound \eqref{eq:27}
for the period
essentially\footnote{That is to say, it coincides up to a bounded multiple of an arbitrarily small power of $\lambda$.} coincides
with the trivial bound \eqref{eq:25} for the $L$-value;
strong bounds for the period imply strong bounds
for the $L$-value,
and vice versa.

This matching between trivial bounds for periods
and trivial bounds for $L$-values holds up
in the weight and squarefree level aspects:
for $f$ a holomorphic newform
of weight $k$ and squarefree level $q$,
a generalization\footnote{Watson's original formula would suffice
  when $q=1$.}
of Watson's formula
due to Ichino
\cite{MR2449948}
that was pinned down precisely in \cite{PDN-HQUE-LEVEL}
asserts
that for each
fixed Maass eigencuspform
or unitary Eisenstein series $\phi$ on $Y_0(1)$,
one has
\begin{equation}\label{eq:32}
\left\lvert D_{f}(\phi) \right\rvert^2
= (q k)^{-1+o(1)}
L(f \times f \times \phi,\onehalf).
\end{equation}
Here the convexity bound reads
$L(f \times f \times \phi, \onehalf) \ll (qk)^{1+o(1)}$.
Thus in the eigenvalue, weight, and squarefree level aspects,
the trivial bounds for periods and $L$-values essentially coincide;
in other words,
the equidistribution and subconvexity problems are essentially
equivalent.

We find
that this equivalence does \emph{not} survive
the passage to non-squarefree levels.
A simple yet somewhat artificial way to see this is to consider
a sequence of twists
$f_p = f_1 \otimes \chi_p$ of a fixed form $f_1$ of level $1$ by
quadratic Dirichlet characters $\chi_p$ of varying prime conductor $p$.
The form $f_p$ has trivial central character
and level $p^2$.
For each $\phi$ as above, one has
\[
L(f_p \times f_p \times \phi,s)
=
L(f_1 \times f_1 \times \phi,s)
\]
for all $s \in \mathbb{C}$.
Thus it does not even make
sense to speak of the ``subconvexity problem'' corresponding to the
equidistribution problem for the measures $\mu_{f_p}$,
as only one $L$-value is involved.
The artificial nature of this example suggests that one could conceivably
still have such an equivalence
by restricting to forms that are \emph{twist-minimal}
(have minimal conductor among their $\GL(1)$ twists),
but this turns out not to be the case;
we find that the equidistribution problem
is (in general) substantially \emph{easier}
than the subconvexity problem (see Section \ref{sec:local-lind-hypoth}).

\subsection{Local Rankin--Selberg integrals}
\label{sec:local-rankin-selberg}
The ideas involved in clarifying the relationship
between the equidistribution and subconvexity problems
discussed in Section \ref{sec:equid-vs-subconvexity}
are exemplified by the following special case.
Let $f$ be a holomorphic newform of fixed weight $k$
and prime power level $q = p ^n$, with
a fixed prime $p$ and varying exponent $n \rightarrow \infty$.
Recall the full-level Eisenstein series
$E_s$,
defined for $\Re(s) > 1$
by the absolutely and uniformly convergent series
\[
E_s(z) =
\sum_{\gamma \in \Gamma_\infty \backslash \Gamma_0(1)}
(\Im \gamma z)^s,
\quad
\Gamma_\infty = \left\{ \pm \left[
    \begin{smallmatrix}
      1&n\\
      &1
    \end{smallmatrix}
  \right] : n \in \Z  \right\}
\]
and in general by meromorphic continuation.
It is known that $s \mapsto E_s$
has no poles in $\Re(s) \geq \onehalf$
except a simple pole at $s=1$ with constant residue.
Those $E_s$ with $\Re(s) = \onehalf$ are called unitary Eisenstein
series,
and furnish the continuous spectrum of $L^2(Y_0(1))$.
We fix $t \in \mathbb{R}$ with $t \neq 0$,
and take $\phi = E_{\onehalf+it}$;
although $\phi$ is not bounded,
it is a natural function against which to test
the measure $\mu_f$.

The period $\mu_f(E_{\onehalf+it})$
is related to the $L$-value
$L(f \times f, \onehalf+it)$,
but not directly.
The ``usual'' integral representation for $L(f \times f, \onehalf+it)$
involves an Eisenstein series for
the group $\Gamma_0(q)$,
so that the integral cleanly unfolds
(initially for $\Re(s) > 1$, in general by analytic continuation):
\begin{align*}
  \int_{\Gamma_0(q) \backslash \mathbb{H}}
  y^k |f|^2(z) \left( \sum_{\gamma \in \Gamma_\infty \backslash
      \Gamma_0(q)} (\Im \gamma z)^s \right)
  \, \frac{d x \, d y }{y^2}
  &=
  \int_{x=0}^{1} \int_{y=0}^{\infty}
  y^{k-1+s} |f|^2(z)
  \, \frac{d x \, d y }{y} \\
  &= \frac{ \Gamma(s+k-1) }{ (4 \pi )^{s+k-1} }
  \sum_{n \in \mathbb{N}}
  \frac{\lambda_f(n)^2}{n^s} \\
  &\approx \frac{ \Gamma(s+k-1) }{ (4 \pi )^{s+k-1} } \frac{L(f \times f, s)}{\zeta (2 s)},
\end{align*}
where $f(z) = \sum_{n=0}^{\infty} \lambda_f(n) n^{(k-1)/2} e^{2
  \pi i n z}$
and $\approx$ denotes equality
up to some very simple Euler factors at $p$
that are bounded from above and below by absolute constants when $\Re(s) = 1/2$
(see Section \ref{sec:proofs}).

On the other hand, the full-level
Eisenstein series $E_s$ is defined relative to
$\Gamma_0(1)$.
Since $f$ is invariant only under
the smaller group $\Gamma_0(q)$, the unfolding
for $\mu_f(E_{\onehalf+it} )$ is not so clean;
instead of giving a simple multiple of the $L$-value ,
it gives its multiple
by a more complicated proportionality factor
$J_f(s)$ satisfying \eqref{eq:29}.
The square of a precise form
of this relation
implies (with $\phi = E_s$ and $s = 1/2+it$)
\begin{equation}\label{eq:30}
  \left\lvert D_f(\phi) \right\rvert^2
  = q^{o(1)}
  \left\lvert
    J_f(s) J_f(1-s)
  \right\rvert
  L(f \times f \times \phi,\onehalf).
\end{equation}
Here the implied constant in $o(1)$ is allowed to depend upon
the weight $k$ and the fixed form $\phi$,
and $L(f \times f \times \phi, \onehalf)
= L(f \times f ,\onehalf + it) L(f \times f ,\onehalf - it)
= |L(f \times f ,\onehalf + it)|^2$.

The content of Ichino's formula \cite{MR2449948}, when
combined with a lemma \cite[Lemma
3.4.2]{michel-2009} of
Michel--Venkatesh,
is that the relation \eqref{eq:30}
continues to hold when $\phi$ is a Maass eigencuspform
provided that $s = s_{\phi,p}$
is chosen so the $p$th Hecke eigenvalue of $\phi$ is $p^{s-\onehalf} +  p^{\onehalf-s}$.
With this normalization, the Ramanujan conjecture asserts $\Re(s) =
1/2$;
it is known unconditionally that $|\Re(s)-1/2| \leq 7/64 < 1/2$
(see \cite{MR1937203}),
so in particular $0 < \Re(s) < 1$.
Thus in all cases, the relative difficulty
of
the equidistribution problem for $\mu_f$
and the subconvexity problem for twists of $f \times f$
(in the $n \rightarrow \infty$ limit)
is governed by the analytic behavior
of $J_f(s)$ in the strip
$\Re(s) \in (1/2 - 7/64, 1/2 + 7/64) \subset (0,1)$.

The quantity $J_f(s)$ is best studied $p$-adically.
Let $W : \PGL_2(\mathbb{Q}_p) \rightarrow \mathbb{C}$ be an $L^2$-normalized Whittaker
newform
for $f$ at $p$;
in classical terms, this function packages all $p$-power-indexed Fourier
coefficients
of $f$ at all cusps of $\Gamma_0(q)$ (see Section \ref{sec:cusps-gamm-four}).
Then the relation \eqref{eq:30} holds
with the definition
\begin{equation}\label{eq:31}
J_f(s)
:=  \int_{k \in \GL_2(\mathbb{Z}_p)} \int_{y \in \mathbb{Q}_p^\times}
\left\lvert
  W \left( \begin{bmatrix}
      y &  \\
      & 1
    \end{bmatrix} k \right)
\right\rvert^2
|y|^{s} \, \frac{d^\times y}{|y|} \, d k.
\end{equation}
We refer to Sections \ref{sec:2-notations}
and \ref{sec:2.2} for precise definitions and normalizations.
When $q = p^1$ is
squarefree, there are explicit formulas for $W$ with which
one may easily show that
\[
J_f(s) =
p^{s-1} \frac{\zeta_p(s) \zeta_p(s+1)}{ \zeta_p(2s)\zeta_p(1)},
\quad \zeta_p(s) := (1-p^{-s})^{-1},
\]
which is consistent with a special case
of the relation \eqref{eq:32}.  When $q = p^n$ with
$n \geq 2$, such as is the case when $f$ is supercuspidal
at $p$, the function $W$ is more difficult to describe
explicitly,
and so it is not immediately clear whether a comparably simple
formula
exists for $J_f(s)$.

\subsection{Local convexity and subconvexity}
\label{sec:local-conv-subc}
In Section~\ref{s:localfe} we prove what we call
a \emph{local convexity bound}
for the local integral $J_f(s)$ as given by \eqref{eq:31}.
The terminology
is justified by the proof,
which we now illustrate. We continue to assume that $f$ is a  newform of prime power level $q = p ^n$, and let $\pi$
be the representation of $\GL_2(\mathbb{Q}_p)$ generated by $f$.
The local $\GL(2) \times \GL(2)$
functional equation (see Proposition \ref{p:localfe}, or
\cite{Ja72})
asserts that the normalized local Rankin--Selberg integral
\begin{equation}\label{eq:37}
\Jfcomplete(s)
:= \frac{\zeta_p(2s)}{L(\pi  \times \pi ,s)}
J_f(s)
\end{equation}
satisfies
\begin{equation}\label{eq:14}
  \Jfcomplete(s) = C^{s-1/2} \Jfcomplete(1-s),
\end{equation}
where $C = C(f \times f)$ is the conductor of
the Rankin--Selberg self-convolution of $f$;
the latter is a power of $p$ that
satisfies
$1 \leq C \leq p^{n+1}$ (see Proposition \ref{propnN}).

Our assumption that $W$ is $L^2$-normalized
implies the trivial bound $\Jfcomplete(s) \ll 1$ for $\Re(s) = 1$,
which we may transfer to the bound $\Jfcomplete(s) \ll C^{-1/2}$ for
$\Re(s)=0$ via the functional equation \eqref{eq:14}.
Interpolating these two bounds by the
Phragmen--Lindel\"{o}f principle,
and using that $\Jfcomplete(s) \asymp J_f(s)$ uniformly for $\Re(s)
\geq \delta > 0$,
we deduce
$J_f(s) \ll C^{-1/2 + \Re(s)/2}$
uniformly for $\Re(s)$ in any compact subset
of $(0,1)$.
If $\Re(s) = 1/2$,
which under the Ramanujan conjecture we
may always assume to be the case in applications,
then the local convexity bound just deduced reads
\begin{equation}\label{eq:8}
  C^{1/2} J_f(s)
  \ll C^{1/4}.
\end{equation}

The proof we have just sketched of \eqref{eq:8}  is
analogous to that of the
(global) convexity bound for $L(f \times
f \times \phi, \onehalf)$,
which augments a trivial bound
in the region of absolute convergence with the functional
equation
and the Phragmen--Lindel\"{o}f principle
(see \cite[Sec 5.2]{MR2061214}).
We refer to a bound that improves upon \eqref{eq:8}
by a positive of power of $q$ as a \emph{local subconvex bound},
and to the problem of producing such a bound as a local
subconvexity problem.

\subsection{QUE versus local and global subconvexity}
\label{sec:que-versus-local}
The upshot of the above considerations is the following.
Preserve the notation and assumptions
of Sections \ref{sec:local-rankin-selberg}
and \ref{sec:local-conv-subc}.
Assume also, for simplicity, that $\Re(s) = 1/2$.
We may
rewrite the formula \eqref{eq:30}
in the suggestive form
\begin{equation}\label{eq:11}
  \left\lvert D_f(\phi)  \right\rvert^2
  = q^{o(1)}
  \left\lvert
    \frac{C^{1/2} J_f(s)}{C^{1/4}}
  \right\rvert^2
  \frac{L(f \times f \times
    \phi, \onehalf)
  }
  {
    C^{1/2}
  }.
\end{equation}
Here the local and global convexity bounds read
\begin{equation}\label{eq:33}
\frac{C^{1/2} J_f(s)}{C^{1/4}} \ll 1
\quad \text{ resp. }
\frac{L(f \times f \times \phi, \onehalf)}{C^{1/2}} \ll C^{o(1)},
\end{equation}
where the implied constants are allowed to depend upon
$k$, $s$ and $\phi$.
 Now, note that the intersection
of the convexity bounds \eqref{eq:33}
is essentially\footnote{That is to say, it is equivalent up to $q^{o(1)}$.}
equivalent, via \eqref{eq:11},
to the trivial
bound $D_f(\phi) \ll 1$
for the QUE problem. For emphasis, we summarize as follows:
\begin{observation}\label{thm:que-versus-local-1}
 Fix a prime $p$, an even integer $k$,
  a complex number $s$,
  and either a Maass eigencuspform
  $\phi$
  with $p$th normalized
  Hecke eigenvalue $p^{s-1/2} + p^{1/2-s}$ or a unitary
  Eisenstein
  series $\phi = E_{s}$
  on $Y_0(1)$.
  Suppose, for simplicity, that $\Re(s) = 1/2$.
  Then the following are equivalent
  (with all implied constants
  allowed to depend upon $p$, $k$, and $\phi$):
  \begin{enumerate}
  \item (Equidistribution in
    the depth aspect with a power savings)
    There exists $\delta > 0$
    so that
    $D_f(\phi) \ll q^{-\delta}$ for all holomorphic
  newforms $f$ of weight $k$
  and prime power level $q = p^n$.

  \item There exists $\delta > 0$
    so that for each holomorphic
  newform $f$ of weight $k$
  and prime power level $q = p^n$, at least one of the following bounds
    hold:
    \begin{enumerate}
    \item  (Global subconvexity without excessive conductor-dropping)\footnote{This estimate is implied by a
        global subconvex bound,
        which saves a small negative power of $C$ rather than of
        $q$,
        together with a condition
        of the form $\log(C) \geq \alpha \log(q)$
        for some fixed $\alpha > 0$.
        Note, for instance, that $C \geq q$ if $f$ is twist-minimal,
        in which case we may drop the phrase
        ``without excessive conductor dropping''.}
      $$
      \frac{
        L(f \times f \times \phi, 1/2)
      }
      {
        C^{1/2}
      }
      \ll
      q^{-\delta},
     $$
    \item (Local subconvexity)
      \begin{equation}\label{e:locsubconvex}
      \frac{C^{1/2} J_f(s)}{C^{1/4}}
      \ll q^{-\delta}.
     \end{equation}
    \end{enumerate}
  \end{enumerate}
\end{observation}

\begin{remark}
  We have stated the above equivalence as an observation
  (rather than as, say, a theorem)
  because one of the main results of this paper
  is that ``local subconvexity'' holds
  in a strong form
  (see Section \ref{sec:local-lind-hypoth}).
\end{remark}

\subsection{Local Lindel\"{o}f hypothesis}
\label{sec:local-lind-hypoth}
One might argue that the more interesting objects
in the identity \eqref{eq:11}
are the global period $D_f(\phi)$ and
the global $L$-value $L(f \times f \times
\phi,\onehalf)$,
rather than the local period $J_f(s)$.
One would like to compare
precisely the difficulty of the QUE problem and the global
subconvexity problem.
In order to do so via \eqref{eq:11},
one must understand the true order of magnitude of $J_f(s)$.
Suppose once again, for simplicity,
that $\Re(s) = 1/2$.
A global heuristic\footnote{The heuristic involved
  a computation by the first-named
  author, without appeal to triple product formulas,
  of an average of $|D_f(\phi)|^2$ over $f$ of level $p^{2 m}$
  (see \cite{PDN-todo}).}
suggested that
one should have $J_f(s) \approx q^{-1/2 + o(1)}$
in a mean-square sense.
This expectation
would be consistent with
the individual bound
\begin{equation}\label{eq:9}
  C^{1/2} J_f(s) \ll
  (C/q)^{1/2} q^{o(1)},
\end{equation}
which we term the \emph{local Lindel\"{of} hypothesis}.

In the special case $q = p^n$ relevant for Observation~\ref{thm:que-versus-local-1}, we remark that $C/q \leq p$
with equality if and only if $n$ is odd (see Proposition \ref{propnN}),
so that one should regard the RHS of \eqref{eq:9} as being
essentially bounded as far as the depth aspect is concerned.
We may rewrite the bound \eqref{eq:9}
in the form
$C^{1/2} J_f(s) \ll C^{1/4} (C/q^2)^{1/4} q^{o(1)}$;
since $C/q^2 \ll_p
q^{-1}$,
we see that \eqref{eq:9} implies \eqref{e:locsubconvex} in a strong sense.
This makes clear the analogy with the (global) Lindel\"{o}f hypothesis,
as described in Section \ref{sec:equid-vs-subconvexity}.

One of the main technical results of this paper is
a proof of the bound \eqref{eq:9}
for all newforms on $\PGL(2)$. The proof goes
by an explicit case-by-case calculation
of $J_f(s)$, and yields the more precise bound
\begin{equation}\label{eq:10}
C^{1/2} J_f(s) \le 10^{3\omega(q)}\tau(q/\sqrt{C}) (C/q)^{1/2},
\end{equation}
where $\tau(n)$ (resp. $\omega(n)$) denotes the number of positive divisors (resp. prime divisors) of $n$. We remark that $q/\sqrt{C}$ is always integral, and equals 1 if and only if $q$ is squarefree. As a byproduct of our explicit calculations,
we obtain a precise generalization of Watson's formula
to certain triple product integrals involving newforms
of non-squarefree level (see~Theorem \ref{thm:watson-ext}).

By the discussion of Section \ref{sec:local-conv-subc},
it follows
that the global convexity bound is remarkably
\emph{stronger} than the trivial bound for the QUE problem,
or in other words,
that the subconvexity problem for $L(f \times f \times
\phi,\onehalf)$ in the depth aspect ($f$ of level $p^n$, $p$
fixed, $n \rightarrow \infty$)
is much
harder than the corresponding equidistribution problem,
in contrast to the essential equivalence of their difficulty
in the eigenvalue, weight and squarefree level aspects.

The above situation is somewhat reminiscent
of how
the problem of establishing
the equidistribution of Heegner points
of discriminant $D$
on $Y_0(1)$ ($D \rightarrow -\infty$) is essentially equivalent
to a subconvexity
problem when $D$ traverses a sequence of fundamental
discriminants (c.f. \cite{MR931205}),
but reduces to any nontrivial bound
for the $p$th Hecke eigenvalue of Maass forms on $Y_0(1)$
when $D = D_0 p^{2n}$ for some fixed fundamental discriminant
$D_0$ and some increasing prime power
$p^n$.

\begin{remark}
  Let $f_1$ and $f_2$ be a pair of
  $L^2$-normalized
  holomorphic newforms,
  of the same fixed weight,
  on $\Gamma_0(p^n)$
  with $n \geq 2$.
  One knows that
  \begin{equation}\label{eq:20}
    C := C(f_1 \times f_2) \leq p^{2 n}.
  \end{equation}
  There is a sense in which $C$ measures the difference between
  the representations of $\PGL_2(\mathbb{Q}_p)$ generated by
  $f_1$ and $f_2$, and that for typical $f_1$ and $f_2$, the
  upper bound in \eqref{eq:20} is attained.  This
  perspective is consistent with the much stronger bound $C \leq
  p^{n+1}$ that holds on the thin diagonal subset $f_1 = f_2$,
  and also with the explicit formulas for $C$ given in
  \cite{MR1606410}.  We expect that the problems of improving
  upon the Cauchy--Schwarz bound $\int \overline{f_1}
  f_2 \phi \ll_\phi 1$ (integral is over $\Gamma_0(q) \backslash
  \mathbb{H}$ with respect to the hyperbolic probability measure) and the
  convexity bound $L(f_1 \times f_2 \times \phi,1/2) \ll
  C^{1/2}$ should have comparable difficulty if and only if the
  upper bound in \eqref{eq:20} is essentially attained.  If
  reasonable,
  this
  expectation suggests a correlation between
  the smallness of $C$ and the
  discrepancy of difficulty between
  the corresponding equidistribution and
  subconvexity problems.
\end{remark}

\subsection{Local Riemann hypothesis}
\label{sec:local-riem-hypoth}
Maintain the assumption that $f$ is a newform of level $p^n$
that generates
a representation $\pi$
of $\PGL_2(\mathbb{Q}_p)$.
Numerical experiments strongly suggest
that
the normalized local Rankin--Selberg integral
$\Jfcomplete(s)$ (see \eqref{eq:37}),
which is an essentially palindromic polynomial\footnote{The local functional equation for $\GL(2) \times \GL(2)$ implies that $\Jfcomplete(s) = P_f(p^s)$ for some $P_f(t)$ in $\C[t, 1/t]$ satisfying $P_f(t) = p^{-N/2} t^N P_f(p/t)$ where the integer $N$ is defined by the equation $C(f \times f) = p^N$.} in $p^{\pm s}$,
has all its zeros on the line $\Re(s) = 1/2$.\footnote{
  More precisely, it seems that $J_f^*(s)$ has its zeros on
  $\Re(s) = 1/2$ unless $\pi$ is a ramified quadratic
  twist of a highly non-tempered spherical representations,
  specifically unless $\pi  = \beta |.|^{s_0} \boxplus \beta
  |.|^{-s_0}$
  with $s_0 \in \mathbb{R}$, $|s_0| > 1/4 + \delta_p$
  (see Section \ref{sec:local-statement-result} for notation);
  here $\delta_p$ is a positive real that satisfies
  $\delta_p \rightarrow 0$
  as $p \rightarrow \infty$.
  By the classical bound $|s_0| \leq 1/4$, the latter
  possibility
  does not occur.}

We suspect that this ``local Riemann Hypothesis'' should follow from
known properties
of the classical polynomials implicit in our formulas
for $\Jfcomplete(s)$ (see Theorem \ref{t:mainlocal}),
but it would be interesting to have a more conceptual
explanation, or a proof that
does not rely upon our brute-force computations.
It seems reasonable to expect that such an alternative explanation
would lead to a different proof of the local Lindel\"{o}f bound
\eqref{eq:9}.

\begin{example}
  Suppose that $\pi$ has
  ``Type 1''
  according to the classification recalled
  in Section \ref{sec:local-statement-result}.
  Let $p^{2 g}$ ($g \geq 1$)
  be the conductor of $\pi$.
  Suppose that $p^{2 g}$ is also the conductor
  of $\pi \times \pi$; equivalently,
  $\pi$ is twist-minimal.
  Then (the calculations leading to) Theorem \ref{t:mainlocal}
  imply that
  $\Jfcomplete(s)$
  differs by a unit in $\mathbb{C}[p^{\pm s}]$
  from $F(p^{-s})$,
  where
  $F$ is the integral polynomial
  \[
  F(t)
  =
  1 +
  \sum_{j=1}^{g-1}
  (p^j-p^{j-1})
  t^{2 j}
  + p^g t^{ 2g}
  \in \mathbb{Z}[t].
  \]
\end{example}

\begin{example}
  Suppose that $\pi$ has
  ``Type 2''
  (see Section \ref{sec:local-statement-result})
  and conductor $p^{2 g + 1}$ ($g \geq 1$).
  Then as above,
  $\Jfcomplete(s)$
  differs by a unit in $\mathbb{C}[p^{\pm s}]$
  from $F(p^{-s})$ with
  \[
  F(t)
  =
  \sum_{j=0}^{g}
  p^{j} t^{2 j}
  -
  \sum_{j=0}^{g-1}
  p^{j} t^{2 j + 1}
  \in \mathbb{Z}[t].
  \]
\end{example}

In either example, $F$ satisfies the formal
properties of the $L$-function of a smooth projective
curve of genus $g$ over $\mathbb{F}_p$;
for example, the roots of $F$ come in complex conjugate pairs,
they have absolute value $p^{-1/2}$,
and $F$ satisfies the functional equation
$F(1/p t) = p^{-g} t^{-2 g} F(t)$.
The geometric significance of this, if any, is unclear.

\subsection{A sketch of the proof}\label{sec:sketch-proof}
The essential inputs to our method for proving~\eqref{eq:10} are the local functional
equations for $\GL(2)$ and $\GL(2) \times \GL(2)$, and some
knowledge of the behavior of
representations of $\GL(2)$  under twisting by $\GL(1)$;
specifically, for $\mu$ on $\GL(1)$
and $\pi$ on $\PGL(2)$, we use that the formula $C(\pi \mu) =
C(\pi)$
holds
whenever $C(\mu)^2 < C(\pi)$. Here and below,
$C(\cdot)$ is the conductor of a representation.

Write $F = \mathbb{Q}_p$,  $|.  |=$ the standard $p$-adic absolute value,
$U = \{x \in F: |x| = 1\} = \mathbb{Z}_p^\times$,
$G = \GL_2(F)$,
$n(x) = \left[
  \begin{smallmatrix}
    1&x\\
    &1
  \end{smallmatrix}
\right]$ for $x \in F$,
$a(y) = \left[
  \begin{smallmatrix}
    y&\\
    &1
  \end{smallmatrix}
\right]$ for $y \in F^\times $,
$N = \{n(x): x \in F\}$,
$K = \GL_2(\mathbb{Z}_p)$,
and $Z = \{\left[
  \begin{smallmatrix}
    z&\\
    &z
  \end{smallmatrix}
\right] : z \in F^\times \}$.
We sketch
a proof of the bound \eqref{eq:10}
in the simplest case
that $f$ is a  newform of prime power level $q = p ^n$, and
 $\pi$, the local representation at $p$ attached to $f$, is a supercuspidal
representation of $G$ with trivial central
character, realized in its Whittaker model
with $L^2$-normalized newform $W$.
Let $f_3 : Z N \backslash G \rightarrow
\mathbb{C}$ be the function given by
$f_3(n(x)a(y)k) = |y|^{s}$ in the Iwasawa decomposition.
We wish to compute the local integral
$J_f(s) = \int_{Z N \backslash G} |W|^2 f_3$.
It is convenient to do so in the Bruhat decomposition,
where our measures are normalized so that
\begin{equation}\label{eq:2was}
  \frac{\zeta_p(1)}{\zeta_p(2)}
  \int_{Z N \backslash G} |W|^2 f_3
  =
  \int_{x \in F}
  \max(1,|x|)^{-2s}
  \int_{y \in F^\times }
  |W|^2(a(y) w n(x))
  |y|^{s-1} \, d^\times y \, d x.
\end{equation}

Because the LHS of  \eqref{eq:2was}
satisfies the $\GL(2) \times \GL(2)$ functional
equation, it suffices
to determine the coefficients
of the \emph{positive} powers of $p^s$ occuring on the RHS.
The left $N$-equivariance
of $W$ implies that no such positive powers
arise from the integral over $|x| \geq
C(\pi)^{1/2}$, an implication which in classical terms
amounts to the
calculation of the widths of the cusps of
$\Gamma_0(q)$ (see Section \ref{sec:cusps-gamm-four}).
In the remaining
range $|x| < C(\pi)^{1/2}$,
we show  that
$W(a(y) w n(x))$ is supported on
the coset $|y| = C(\pi)$
of the unit group $U$ in $F^\times $.
Thus by the invariance of the inner product
on $\pi$, the integral over $F^\times $ in \eqref{eq:2was}
is simply $C(\pi)^{s-1}$.
Integrating over $x$ gives
\begin{equation}\label{eq:16}
\frac{\zeta_p(1)}{\zeta_p(2)}
\int_{Z N \backslash G} |W|^2 f_3
= C(\pi)^{s-1}
\left(
  \int_{\substack{ x \in F \\
      |x| < C(\pi)^{1/2}
    }
    }
  \max(1,|x|)^{-2 s} dx
\right)
+ \sum_{m \in \mathbb{Z}_{\geq 0}} \frac{c_m}{p^{m s}}
\end{equation}
for some coefficients $c_m$.
After determining $c_m$ via the $\GL(2) \times \GL(2)$ functional
equation,
we end up with a formula for $\int |W|^2 f_3$
in terms of $C(\pi)$ and $C(\pi \times \pi)$
that shows, by inspection, that $\int |W|^2 f_3$ satisfies the desired
bounds.\footnote{It would be possible
  to establish this by a slightly softer argument, but we
  believe
  that having precise formulas is of independent interest.}

A key ingredient in the above argument
was the support condition
on $W(a(y) w n(x))$ for $|x| < C(\pi)^{1/2}$.
We derive it via a
Fourier decomposition
over the the character group
of $U$ and invariance properties of $W$.
Indeed, the $\GL(2)$ functional equation
implies
\begin{equation}\label{eq:15}
W(a(y) w n(x))
= \sum_{
  \substack{
    \mu \in \hat{U}  \\
    C(\pi \mu) = |y|
  }
}
\mu(y) \eps(\pi \mu)
G(x,\mu),
\end{equation}
where
$G(x,\mu) = \int_{y \in F^\times } \psi(x y) \mu(y) W(a(y)) = \int_{y \in U} \psi(x y) \mu(y)$ and $\eps(\pi \mu) = \eps(\pi \mu, 1/2)$ is the local $\eps$-factor (see Section \ref{sec:supp-whitt-newf}).
The characters $\mu$ contributing
nontrivially to \eqref{eq:15} all satisfy
$G(x,\mu) \neq 0$, which implies $C(\mu) \leq x$; in that case our assumption
$|x|^2 < C(\pi)$ and our knowledge of the twisting behavior of
$\pi$ implies $C(\pi \mu) = C(\pi)$.
It follows that $W(a(y) w n(x)) = 0$ unless $|y| = C(\pi)$.

\begin{remark}\label{rmk:bulldoze}
It seems worthwhile to note that one may also
compute the RHS of \eqref{eq:2was} in ``bulldozer'' fashion,
as follows.
Suppose for simplicity that $\pi$ is supercuspidal.
We may view the integral over
$y \in F^\times $ as the inner product of the functions
$W(a(y) w n(x))$ and
$W(a(y) w n(x)) |y|^s$, whose Mellin transforms
are (by definition) local
zeta integrals;
applying the Plancherel
theorem on $F^\times $ and the $\GL(2)$ functional equation,
we arrive at the formula
\begin{equation}\label{eq:3}
  \frac{\zeta_p(1)}{\zeta_p(2)}
  \int_{Z N \backslash G} |W|^2 f_3
  = \sum_{\mu \in \hat{U}} C(\pi \mu)^{s-1}
  \int_{x \in F}
  \frac{|G(x,\mu)|^2}{\max(1,|x|)^{2s}} dx.
\end{equation}
This also follows from \eqref{eq:15} by the Plancherel theorem
on $U$.
Substituting into \eqref{eq:3}
the fact that $C(\pi \mu) \leq \max(C(\pi),C(\mu)^2)$
with equality if $C(\mu)^2 \neq C(\pi)$,
evaluating $|G(x,\mu)|$,
and summing some geometric series,
we find that
\begin{equation}\label{eq:5}
  \frac{\zeta_p(1)}{\zeta_p(2)}
  \int_{Z N \backslash G}
  |W|^2 f_3
  = p^{n (s-1)}
  \left\{ 1 +
    \sum_{1 \leq a < n/2}
    \frac{\zeta_p(1)^{-1}}{p^{(2s-1)a}}
  \right\}
  + p^{-r}
  + \zeta_p(1)
  \sum_{C(\mu)^2 = C(\pi)}
  \frac{C(\pi \mu)^{s-1}}{C(\pi)^{s}},
\end{equation}
where $C(\pi) = p^n$ and $r = \lfloor n/2 \rfloor + 1$. This
identity
agrees with \eqref{eq:16}, and shows that the only barrier
to
obtaining immediately an explicit result
is the potentially subtle behavior of the conductors
of twists of $\pi$
by characters of conductor $C(\pi)^{1/2}$
(see also Remark \ref{rmk:direct-iwasawa}).
It suggests another approach to our local calculations
(write $\pi = \pi_0 \mu_0$ with $\pi_0$ twist-minimal
and compute away),
but one that would be more difficult to implement when
$\pi$ is a ramified twist of a principal series or Steinberg representation.

The approach sketched in this remark has the virtue
of applying to arbitrary vectors $W \in \pi$,
leading to formulas generalizing those that we have given in this paper
in the special case that $W$ is the newvector.
\end{remark}

\def\SOMELETTER{{\ell}}
\subsection{Fourier expansions at arbitrary cusps}\label{sec:four-expans-at}
Let  $f$  be a newform on $\Gamma_0(q)$, $q \in \mathbb{N}$.
In order to apply a variant of the Holowinsky--Soundararajan method
in Section \ref{sec:proof-theor-refthm:m}, we require
some knowledge
of the sizes of the normalized Fourier coefficients
$\lambda(\SOMELETTER;\mathfrak{a})$ of $f$ at
an \emph{arbitrary} cusp $\mathfrak{a}$ of $\Gamma_0(q)$.
It is perhaps not widely known that such Fourier coefficients
are \emph{not} multiplicative in general;
this lack of multiplicativity introduces an additional
 complication
in our arguments.
More importantly, we need some knowledge
of the sizes of the coefficients $\lambda(\SOMELETTER;\mathfrak{a})$
when $\SOMELETTER \mid q^\infty$.
For example, the ``Hecke bound''
$\lambda(\SOMELETTER;\mathfrak{a}) \ll \SOMELETTER^{1/2}$
would not suffice for our purposes.

Let $\lambda(\SOMELETTER) = \lambda(\SOMELETTER;\infty)$ denote the $\SOMELETTER$th normalized Fourier coefficient of
$f$ at the cusp $\infty$.
A complete description of the coefficients $\lambda(\SOMELETTER)$
is given by Atkin and Lehner \cite{MR0268123};
for our purposes, it is most significant
to note that $\lambda(p^\alpha) = 0$ for each $\alpha \geq 1$ if $p$
is a prime
for which $p^2 | q$.

If $\mathfrak{a}$ is the image of $\infty$
under an Atkin--Lehner operator (an element of the normalizer of $\Gamma_0(q)$ in $\PGL_2^+(\mathbb{Q})$),
then the coefficients $\lambda(\SOMELETTER)$ and $\lambda (\SOMELETTER;\mathfrak{a})$ are related
in a simple way;
this is always the case when $q$ is squarefree,
in which case the Atkin--Lehner operators act transitively on the
set of cusps.
Similarly, there is a simple relationship between the Fourier coefficients
$\lambda(\SOMELETTER,\mathfrak{a}), \lambda(\SOMELETTER,\mathfrak{a}')$ of $f$ at each pair of cusps $\mathfrak{a}$, $\mathfrak{a}'$
related by an Atkin--Lehner operator (see \cite{2010arXiv1009.0028G}).
However, such considerations do not suffice
to describe $\lambda(\SOMELETTER;\mathfrak{a})$ explicitly when
$\mathfrak{a}$ is not in the Atkin--Lehner orbit of $\infty$.

Our calculations in Section
\ref{sec:local-calculations}
lead to a precise description of $\lambda(\SOMELETTER;\mathfrak{a})$
for arbitrary cusps $\mathfrak{a}$,
at least in a mildly averaged sense.
This may be of independent
interest.
To give some flavor for the results obtained,
suppose that $q = p^n$ with $n \geq 2$.
The nature of the coefficients
$\lambda(\SOMELETTER;\mathfrak{a})$ depends heavily upon the \emph{denominator} $p^k$ of the
cusp $\mathfrak{\mathfrak{a}}$,
as defined
in
Section \ref{sec:cusps-gamm-four};
briefly, $k$ is the unique integer
in $[0,n]$
with the property that $\mathfrak{a}$ is in the
$\Gamma_0(p^n)$-orbit
of some fraction $a/p^k \in \mathbb{R} \subset
\mathbb{P}^1(\mathbb{R})$ with $(a,p) = 1$.
The Atkin--Lehner/Fricke involution
swaps the cusps of denominator $p^k$
and $p^{n-k}$.

Say that $f$ is \emph{$p$-trivial}
at a cusp $\mathfrak{a}$ if
$\lambda(p^\alpha;\mathfrak{a}) = 0$ for all
$\alpha \geq 1$.
For example,
the result of Atkin--Lehner mentioned above
asserts that $f$ is $p$-trivial at $\infty$.
We observe the ``purity'' phenomenon:
$f$ is $p$-trivial at $\mathfrak{a}$
unless $n$ is even and the denominator $p^k$ of $\mathfrak{a}$ satisfies
$k = n/2$ (see Proposition \ref{keypropositionfourier}).
In the latter case, let us call $\mathfrak{a}$ a \emph{middle
cusp}.

In Section \ref{sec:cusps-gamm-four},
we compute for each $\alpha \geq 0$ the mean square of
$\lambda_f(p^\alpha;\mathfrak{a})$
over all middle cusps $\mathfrak{a}$;
an accurate evaluation of this mean square, together with
the aforementioned ``purity'',
turns out to be equivalent
to our local Lindel\"{o}f hypothesis described above (see Remark
\ref{rmk:direct-iwasawa}).
We observe that the ``Deligne bound''
$|\lambda(\SOMELETTER;\mathfrak{a})| \leq \tau(\SOMELETTER)$
can fail
in the strong form
$\lambda(p^\alpha;\mathfrak{a}) \gg p^{\alpha/4}$
for some $\alpha > 0$ when $f$ is not twist-minimal
(see Remark \ref{rmk:failure-of-deligne-bound}).
In general,
$\lambda(\SOMELETTER;\mathfrak{a})$
may be evaluated exactly in terms of $\GL(2)$ Gauss sums
(e.g., combine \eqref{eq:15} and \eqref{lambdajwhittaker2}
when $\pi$ is supercuspidal).
We suppress further discussion of this point for sake of brevity.

\subsection{Further remarks}
Our calculations in Section \ref{sec:local-calculations},
being local,
apply in greater generality than we have used them.
For example, they imply that the pushforward
to $Y_0(1)$ of the $L^2$-mass of a Hecke-Maass newform on $\Gamma_0(p^n)$ of
bounded Laplace eigenvalue
equidistributes as $p^n \rightarrow \infty$
with $n \geq 2$.
They extend also to non-split quaternion algebras,
where Ichino's formula applies
but the Holowinsky--Soundararajan method does not,
due to the absence of Fourier expansions.
For example, one could establish
that Maass or holomorphic newforms
of increasing level
on compact arithmetic surfaces
satisfy an analogue of Theorem \ref{thm:main-refined}
provided that their level is sufficiently powerful
(c.f. the remarks at the end of Section \ref{sec:equid-vs-subconvexity});
in that context, no unconditional result for forms of increasing
squarefree
level is known.
For automorphic forms
of increasing squared-prime level $p^2$ on definite
quaternion algebras,
an analogue of Theorem \ref{thm:main-refined} had
been derived
earlier by the first-named author
(see \cite{PDN-todo})
via a different method
(i.e., without triple product formulas),
but the bounds obtained there are quantitatively weaker
than those
that would follow from the present work.

After completing
an earlier draft of this paper,
we learned of some interesting parallels in the literature
of some of the analogies presented hitherto.
Lemma 2.1 of Soundararajan and Young \cite{2010arXiv1011.5486S}
gives something resembling a ``local Riemann hypothesis''
for a certain Dirichlet series, studied earlier
by Bykovskii and Zagier,
attached to (not necessarily fundamental)
quadratic discriminants.\footnote{We thank
  M. Young
  for bringing this similarity to our attention.}
Section 9 of a paper of Einsiedler, Lindenstrauss, Michel and
Venkatesh \cite{MR2776363}
establishes what they refer to as ``local subconvexity''
for certain local toric periods,
the proof of one aspect of which
resembles that of what we describe
here as ``local convexity''.
It would be interesting
to understand whether
our work can be understood
together with these parallels in a unified
manner.

\subsection{Acknowledgements} We thank Ralf Schmidt for helping us with the representation theory of $\PGL_2(\Q_p)$, and Philippe Michel for pointing us to the crucial lemma \cite[Lemma
3.4.2]{michel-2009}.
We thank Nahid Walji and Matthew Young for helpful comments on
an earlier draft of this paper. Finally, we would like to thank
the referee for
many helpful comments which have improved
the correctness, clarity, and exposition of this paper.

\section{Local calculations}
\label{sec:local-calculations}
\subsection{Notation and preliminaries}\label{sec:2-notations}

\subsubsection{Groups, measures}
Let $p$ be a prime number, and $F = \mathbb{Q}_p$.\footnote{
  Most of this section reads correctly
  in the more general case that
  $p$ is an arbitrary prime power
  and
  $F$ is a non-archimedean local
  field
  of characteristic zero whose
  residue field has cardinality $p$.
  We work in the restricted generality that we need
  for our global applications
  only because we have not checked that the calculations
  in the Type 3 case of the proof of Theorem \ref{t:mainlocal}
  carry through in this more general context when $p = 2$.
}
Let $\mathfrak{o}$ be its ring of integers,
and $\mathfrak{p}$ its maximal ideal.
Fix a generator $\varpi$ of $\mathfrak{p}$.
Let $|.|$ or $|.|_p$ denote the absolute value
on $F$ normalized so that
$|\varpi| = p^{-1}$.

Let $G = \GL_2(F)$ and $K = \GL_2(\mathfrak{o})$.
For each integral ideal $\mathfrak{a}$ of $\mathfrak{o}$,
let $K_0(\mathfrak{a})$
and $K_1(\mathfrak{a})$ denote the usual congruence
subgroups of $K$:
\[
K_0(\mathfrak{a}) = K \cap \begin{bmatrix}
  \mathfrak{o}  & \mathfrak{o}  \\
  \mathfrak{a}  & \mathfrak{o}
\end{bmatrix},
\quad
K_1(\mathfrak{a})
= K \cap \begin{bmatrix}
  1+ \mathfrak{a}  & \mathfrak{o}  \\
  \mathfrak{a}  & \mathfrak{o}
\end{bmatrix}.
\]
In particular, $K_0(\mathfrak{o}) = K_1(\mathfrak{o}) = K$.
Write
\[
w = \begin{bmatrix}
  0 & 1 \\
  -1 & 0
\end{bmatrix},
\quad
a(y) = \begin{bmatrix}
  y &  \\
  & 1
\end{bmatrix},
\quad
n(x) = \begin{bmatrix}
  1 & x \\
  & 1
\end{bmatrix},
\quad z(t)
= \begin{bmatrix}
  t &  \\
  & t
\end{bmatrix}
\]
for $x \in F, y \in F^\times, t \in F^\times$.
Define subgroups
$N =
\{n(x):  x\in F \}$,
$A = \{a(y): y\in F^\times \}$,
$Z =\{ z(t):
t \in F^\times \}$,
and $B = Z N A = G \cap
\left[
  \begin{smallmatrix}
    *&*\\
    &*
  \end{smallmatrix}
\right]$ of $G$.

We normalize Haar measures as
in~\cite[Section 3.1]{michel-2009}:
The measure $dx$ on the additive group $F$ assigns volume 1
to $\OF$, and transports to a measure on $N$.
The measure $d^\times y$ on the multiplicative group $F^\times$ assigns
volume 1 to $\OF^\times$,
and transports to measures on
$A$ and $Z$.
We obtain a left Haar measure $d_Lb$ on $B$ via
$$d_L(z(u)n(x)a(y)) = |y|^{-1}\, d^\times u \, d x \, d^\times
y.$$
Let $dk$ be the probability Haar measure on $K$.
The Iwasawa decomposition
$G = B K$ gives a left Haar measure $dg = d_L b \, d k$ on $G$;
with respect to the Bruhat decomposition $G =  B \sqcup BwN$,
this measure takes the
form \begin{equation}\label{e:bruhatmeasure}dg =
  \frac{\zeta_p(2)}{\zeta_p(1)}|y|^{-1} \, d^\times u \,
  d^\times y \, dx' \, dx \quad \text{ for } g =
  n(x')a(y)z(u)wn(x), \end{equation} where $\zeta_p(s) =
(1-p^{-s})^{-1}$
(see~\cite[(3.1.6)]{michel-2009}).

\subsubsection{Representations, models}

Fix an additive character $\psi : F \rightarrow \mathbb{C}^1$
with conductor $\mathfrak{o}$. For each generic representation $\sigma$
of $G$, let $\mathcal{W}(\sigma\subscriptp, \psi)$
denote the Whittaker model of $\sigma\subscriptp$ with respect
to $\psi$ (see~\cite{MR0401654}).
For two characters $\chi_1$, $\chi_2$ on $F^\times$, let $\chi_1
\boxplus \chi_2$ denote the principal series representation on
$G$
that is unitarily induced from the corresponding representation of  $B$; this consists of smooth functions $f$ on $G$ satisfying $$f\left(\mat{a}{b}{0}{d} g\right) = |a/d|^{\frac12} \chi_1(a) \chi_2(d)  f(g).$$

\subsubsection{Conductors, $L$-functions, $\eps$-factors}\label{sec:cond-l-funct}
For each character\footnote{We adopt the convention that
  a \emph{character} of a topological group is a continuous (but
  not necessarily unitary) homomorphism into $\C^\times$.}
$\sigma$ of $F^\times$,
there exists a minimal integer $a(\sigma)$ such that
$\sigma(1+t)=1$  for all $t \in \p^{a(\sigma)}$.
For each irreducible admissible
representation $\sigma$ of $G$, there exists a minimal integer $a(\sigma)$ such that $\sigma$ has a $K_1(\p^{a(\sigma)})$-fixed vector. In either case, the integer  $p^{a(\sigma)}$ is called the
local analytic conductor\footnote{In the rest of this paper, we
  will often drop the words ``local analytic" for brevity and
  call this simply the ``conductor".} of $\sigma$;
we denote it by
$C(\sigma)$.

For a representation $\sigma$ of $G$ and a character $\chi$
of $F^\times$, write
$\sigma \chi$ for the representation
$\sigma \otimes (\chi \circ \det)$ of $G$.

Let $L(\sigma, s)$ (resp. $\eps(\sigma,\psi,s)$)
denote the $L$-function (resp. $\varepsilon$-factor)
of an irreducible admissible
representation $\sigma$
of $G$ or a character $\sigma$ of $F^\times$. These local factors are defined in~\cite{MR0401654}.
For $\sigma$ an irreducible admissible representation of $G$, let
$L(\ad \sigma,s)$ denote the
adjoint $L$-function
of $\sigma$, or equivalently,
the standard $L$-function of the adjoint lift of $\sigma$ to
an admissible
representation of $\PGL_3(F)$.

If $\sigma_1$, $\sigma_2$ are
two irreducible admissible representations of $G$,
the
local Rankin--Selberg factors $L(\sigma_1 \times \sigma_2, s)$
and $\eps(\sigma_1 \times \sigma_2,\psi,s)$ are defined
in~\cite{Ja72}.
The local analytic conductor $C(\sigma_1 \times \sigma_2)$
is a nonnegative integral power of $p$,
and can
be defined
by the formula
$\eps(\sigma_1 \times \sigma_2, \psi, s)
= C(\sigma_1 \times
\sigma_2)^{1/2-s}
\eps(\sigma_1
\times \sigma_2, \psi, 1/2)$;
we also let $a(\sigma_1 \times \sigma_2)$
denote the nonnegative integer
for which $C(\sigma_1 \times \sigma_2)
= p^{a(\sigma_1 \times \sigma_2)}$.

\subsubsection{Temperedness}
Let $\pi$ be a generic irreducible admissible
unitarizable representation of $G$
with trivial central character.
The quantity
\begin{equation}\label{eq:7}
  \lambda(\pi) = \begin{cases}0& \text{ if } \pi \text{ is
      tempered, }\\ |s_0|  & \text{ if }\pi\subscriptp \cong
    \beta \ | \cdot |^{s_0} \boxplus \beta^{-1}  \ | \cdot
    |^{-s_0}, \quad s_0 \in  \R, \quad \beta \text{ unitary,
    }\end{cases}
\end{equation}
measures the temperedness of $\pi$.
When $\pi$ arises as the local factor of a cuspidal automorphic
representation of $\GL_2(\mathbb{A})$, it is known that
$\lambda(\pi) \leq 7/64$ (see~\cite{MR1937203}).
For our purposes, it suffices to assume that
$\lambda(\pi) < 1/4$.
We record this assumption as follows:
\begin{condition}\label{cond:temperedness}
  $\pi$ is a generic irreducible admissible
  unitarizable representation of $G$
  with trivial central character
  and $\lambda(\pi) < 1/4$.
\end{condition}

\subsubsection{Classification of representations}\label{sec:class-repr}
Let $\pi$
satisfy Condition \ref{cond:temperedness}.
Write $n = a(\pi)$, and suppose that
$n \geq 2$.
We recall
a certain classification of such $\pi$.
The classification is standard,
although our labeling is not
(and we are not aware of a standard labeling).
\begin{itemize}

\item \textbf{Type 1. } These are the supercuspidal
  representations satisfying
  $\pi\subscriptp\cong\pi\subscriptp\twist\eta\subscriptp$,
  where $\eta\subscriptp$ is the unique non-trivial unramified
  quadratic character of $F^\times$. Equivalently,
  $\pi\subscriptp$ is
  the
  dihedral supercuspidal representation $\rho(E/F,\xi)$
  associated to the unramified quadratic extension $E$ of $F$
  and a
  character $\xi$ of $E^\times$
  that is not $\Gal(E/F)$-invariant.

\item  \textbf{Type 2. } These are the supercuspidal representations satisfying $\pi\subscriptp\ncong\pi\subscriptp\twist\eta\subscriptp$, with $\eta$ as above.

\item     \textbf{Type 3. } In this case $\pi\subscriptp$ is a
  ramified quadratic twist of a spherical representation:
  $$\pi\subscriptp \cong \beta \ | \cdot |^{s_0} \boxplus \beta \ | \cdot |^{-s_0}, \quad s_0 \in i\R \cup (-1/4, 1/4), \quad \beta \text{ ramified, } \ \beta^2=1.$$We denote $\beta_{s_0} = (p^{s_0} + p^{-s_0})^2.$ \medskip

\item     \textbf{Type 4. } In this case $\pi\subscriptp$ is a
  ramified principal series that is not of Type 3:
  $$\pi\subscriptp \cong \beta \boxplus \beta^{-1}, \qquad \beta \text { ramified, unitary character of } F^\times, \qquad \beta^2 \text{ ramified}.$$

\item \textbf{Type 5. }In this case $\pi\subscriptp$ is a ramified
  quadratic twist of the Steinberg representation:
  $$\pi\subscriptp \cong
  \beta\St_{\GL(2)}, \qquad \beta \text{ ramified, }
  \beta^2=1.$$

\end{itemize}

\begin{remark}
  If $p$ is odd, then each supercuspidal representation is
  dihedral, i.e., constructed via the Weil representation from a
  quadratic extension $E$ of $F$ and a non-$\Gal(E/F)$-invariant
  character $\xi$ of $E^\times$.
  Such representations are of Type 1 if $E/F$ is unramified and
  of Type 2 if $E/F$
  is
  ramified.
  If $p$ is even,
  there exist non-dihedral supercuspidals; these are also of Type 2.
\end{remark}

\begin{remark}\label{rem:ram}
  For representations
  of Type 3 or 5,
  the ramified quadratic character
  $\beta$ satisfies $a(\beta) =1$ if
  $p$ is odd
  and $a(\beta) \in \{2,3 \}$ if $p$ is even.
\end{remark}

\begin{remark}
  If $\pi$ is of Type 1, 3, 4, or 5,
  then $n$ is even.
  If $\pi$ is of Type 2,
  then $n$ can be either odd or even.
\end{remark}

\subsubsection{Properties of the adjoint conductor}
Let $\pi$ be a generic irreducible admissible unitarizable
representation of $G$ with trivial central character.
Write $n = a(\pi)$ and $N = a(\pi \times \pi)$.
In Section \ref{sec:proofs},
we will establish the
following result concerning the integer $N$ and its
relation to $n$.
We state it here because it will be useful
in interpreting the results to follow.
\begin{proposition}\label{propnN}
  The integer $N$ is even and satisfies $N \le
  n+1$.
  Furthermore, the following conditions on $\pi\subscriptp$ are equivalent:
  \begin{enumerate}
  \item $N=n+1$.
  \item $n$ is odd.
  \item Either
    \begin{enumerate}
    \item $\pi\subscriptp$ is
      the Steinberg representation or an unramified quadratic
      twist thereof
      (in which case $n = 1$), or
    \item $\pi\subscriptp$ is a representation of Type 2 for
      which $n$ is odd.
    \end{enumerate}
  \end{enumerate}
\end{proposition}

\subsubsection{Definition of Ichino integral}\label{sec:defin-ichino-integr}
Let $s$ be a complex parameter,
and $\pi_3 = |.|^{s-1/2} \boxplus |.|^{1/2-s}$ the corresponding
principal series representation of $G$.
It is well-known that
$\pi_3$ is irreducible and unitarizable
if and only if
$\Re(s) = 1/2$ or $s \in (0,1)$;
suppose that this is the case.
Fix a non-zero $K$-invariant vector $x_3 \in \pi_3$,
which is then unique up to a scalar.
We recall, for later use, the following formula for the normalized
Hecke eigenvalues
of $x_3$:
\begin{equation}\label{eq:12}
  \lambda_{s,m} =
  \sum_
  {
    \substack
    {
      i, j \in \mathbb{Z}_{\geq 0} \\
      i + j = m
    }
  }
  \alpha^i \beta^j
  = \begin{cases}
    \frac{\alpha^{m+1} - \beta^{m+1}}{\alpha - \beta }
    & m \geq 0 \\
    0 &  m < 0
  \end{cases}
  \quad \text{ with }
  \alpha = p^{s-1/2}, \beta = p^{1/2-s}.
\end{equation}

Let $\pi$ be a representation of $G$
satisfying Condition \ref{cond:temperedness}.
Let $x \in \pi$ be a newvector,
i.e., a nonzero vector on the
unique line of $K_0(\mathfrak{p}^{a(\pi)})$-invariant
vectors in $\pi$.
Fix arbitrary $G$-invariant inner products
$\langle , \rangle$
on $\pi$ and $\pi_3$.
It follows from \cite[Lemma 2.1]{MR2449948}
that what we will call
the \emph{local Ichino integral}
\begin{equation}\label{trip-prod-defn}
  I(s)
  =
  I(s;\pi)
  =
  \int\limits_{Z \backslash G}
  \left(
    \frac
    {
      \langle g  x,x \rangle
    }{
      \langle x,x \rangle
    }
  \right)^2
  \frac{
    \langle g  x_3, x_3 \rangle
  }
  {
    \langle x_3, x_3 \rangle
  }
  \, d g
\end{equation}
converges absolutely
provided that either  $\Re(s) = 1/2$ or $s \in (2 \lambda(\pi), 1
- 2 \lambda(\pi))$;
we will see later that it extends to
a meromorphic function
of $s \in \mathbb{C}$.
Note that $I(s)$ depends only upon $\pi$, $s$,
and our normalization of measures,
and not upon the precise choice of $x$, $x_3$,
or the inner products $\langle , \rangle$
on $\pi,\pi_3$.
While not immediately obvious, it can be shown
(using~\eqref{mvgeneral} below, for instance) that the right
hand side of~\eqref{trip-prod-defn}  is nonnegative.

It will be convenient
to work with the normalized quantity
\[
\Icomplete(s) = I^*(s;\pi)=
  \left(
    \frac{ L (\pi  \times \pi  \times \pi_3, \onehalf ) \zeta _p
      (2)^2}{
      L (\ad \pi_3, 1 )
      L (\ad \pi , 1 ) ^2 }
  \right) ^{-1}
  I(s).
\]
We note that
$L(\pi \times \pi \times \pi_3,1/2)
= L(\pi \times \pi, s) L(\pi \times \pi, 1-s)$
and $L(\ad \pi_3, 1)^{-1}
= (1 - p^{2s-2}) (1 - p^{-1}) (1 - p^{-2s}) $.
\subsection{Statement of results}
\label{sec:local-statement-result}
Let $\pi$ be a representation of $G$
satisfying Condition \ref{cond:temperedness},
and let $s \in \mathbb{C}$.
Our main local result is an explicit formula
for the normalized local Ichino integral
$I^* = I^*(s) = I^*(s;\pi)$;
as a consequence, we deduce optimal bounds for the latter.
The proofs will occupy the remaining subsections of Section
\ref{sec:local-calculations}.
We will use the notation
\[
n = a(\pi),
\quad
N = a(\pi \times \pi),
\quad
n' = n - \frac{N}{2}.
\]
Proposition~\ref{propnN} implies
that $n'$ is an integer satisfying $\frac{N}2-1 = n' = \frac{n-1}{2}$ if $n$ is odd and
$\frac{N}2 \le n' \le n$ if $n$ is even.

When $n \in \{0,1\}$, the value of $\Icomplete$ is
already known
(see
\cite[Theorem 1.2]{MR2585578}
and
\cite[Lemma 4.2]{PDN-HQUE-LEVEL}):
\begin{theorem}
  \label{t:oldeasy}
  Suppose that $n=0$ or $n=1$. Then $\Icomplete = p^{-n}.$
\end{theorem}

We turn to the case $n \geq 2$.
Our formulas will depend upon the classification
of $\pi$
recalled in Section \ref{sec:class-repr}
and
the notation $\lambda_{s,m}$ introduced in \eqref{eq:12}.

\begin{theorem}\label{t:mainlocal}
  Suppose that $n \ge 2$. Then $\Icomplete = p^{-n} \cdot L(\ad  \pi\subscriptp ,1)^2 \cdot Q_{\pi, p}(s)^2$ with $$Q_{\pi, p}(s) = \begin{cases} \lambda_{s,n'} - p^{-1}\lambda_{s,n'-2} & \text{ for Type 1,}\\ \lambda_{s,n'} - p^{-\onehalfbig}\lambda_{s,n'-1}  & \text{ for Type 2,} \\    \lambda_{s,n'} - 2p^{-\onehalfbig}\lambda_{s,n'-1} + p^{-1}\lambda_{s,n'-2}  & \text{ for Type 4,}\\ \lambda_{s,n'} - p^{-\onehalfbig} (1 + p^{-1})\lambda_{s,n'-1} +  p^{-2}\lambda_{s,n'-2}  & \text{ for Type 5.}

  \end{cases}$$

  In the remaining case
  that $\pi$ is of Type 3, we have $N=0$, $n=n' = 2a(\beta) \in \{2,4,6\}$  and
  $$Q_{\pi, p}(s) = \begin{cases} \lambda_{s,2} -  p^{-\onehalfbig}\beta_{s_0}\lambda_{s,1} + p^{-1}(2\beta_{s_0} -2  - p^{-1}), &  p \text{ odd},
    \\ \lambda_{s,n} - p^{-\onehalfbig}\beta_{s_0}\lambda_{s,n-1} + 2p^{-1}(\beta_{s_0} -1)\lambda_{s,n-2} -  p^{-3/2}\beta_{s_0}\lambda_{s,n-3} + p^{-2} \lambda_{s,n-4}, & p \text{ even}. \end{cases}$$

\end{theorem}

\begin{corollary}[Local Lindel\"{o}f
  hypothesis]\label{cor:something-about-implied-bound-on-Ip}
  Let $\theta = |\Re(s-1/2)|$.
  Then
  $\Icomplete < 10^5 p^{-n} \tau(p^{n'})^2 p^{2 \theta n'}$.
\end{corollary}
\begin{proof}
  The case $n \in \{0,1\}$
  follows easily from Theorem~\ref{t:oldeasy},
  so suppose that $n \ge 2$. The value of $L(\ad  \pi\subscriptp ,1)$ can be read off from the formula for $L(\pi\subscriptp \times \pi\subscriptp,s)^{-1}\zeta_p(2s) = (1+p^{-s})^{-1} L(\ad  \pi\subscriptp ,s)^{-1}$ given in Table \ref{table1} below. We see that $L(\ad  \pi\subscriptp ,1) \le 30 <10^{\frac32}$ in every case.  The formulas for $Q_{\pi, p}(s)$ provided above imply the bound $|Q_{\pi, p}(s)| \le 10\tau(p^{n'}) p^{ \theta n'}$. The result now follows from Theorem \ref{t:mainlocal},
  noting that $n \geq 2$ implies $n' \geq 1$.
\end{proof}

\subsection{An identity of local integrals}\label{sec:2.2}
In this section we apply a lemma of Michel--Venkatesh
to establish the meromorphic
continuation
of the local Ichino integral $I(s) = I(s;\pi)$ defined in
Section \ref{sec:defin-ichino-integr}
and to reduce its study
to that of a Rankin--Selberg integral
involving the Whittaker newform of $\pi$.

Let $\pi$ be a generic
irreducible admissible unitarizable
representation of $G$ with trivial central character,
realized in its $\psi$-Whittaker model:
$\pi = \mathcal{W}(\pi,\psi)$.
By \cite[Lemma 2.19.1]{MR0401654},
the formula
\begin{equation}
  \langle W_1, W_2 \rangle = \int_{F^\times}
  W_1(a(y)) \overline{W_2(a(y))} \,  d^\times y
  \quad (W_1, W_2 \in \pi)
\end{equation}
defines
a $G$-invariant hermitian pairing on $\pi$.
\begin{definition}\label{defn:normalized-W-newform}
  The \emph{normalized Whittaker newform}
  $W \in \pi$
  is
  the unique vector invariant under $K_0(\p^{a(\pi)})$
  that satisfies $\langle W, W\rangle = 1$ and $W(1)>0$.
\end{definition}
\begin{remark}
  One can check that $W(1) = 1$ whenever $a(\pi) \geq 2$.
\end{remark}


Let $s \in \mathbb{C}$ be a complex parameter.
We realize $\pi_3 = |.|^{s-1/2} \boxplus |.|^{1/2-s}$
in its induced model,
and let $f_s \in \pi_3$ denote the unique $K$-invariant vector
that satisfies $f_s(1) = 1$.
Define the \emph{local Rankin--Selberg integral}
\begin{equation}\label{RS-integral-defn}
  J\subscriptp(s) =  \int\limits_{N Z \backslash G}
  W(g) W(a(-1)g) f_s(g)\, d g,
\end{equation}
where
$W \in \pi$ is the normalized Whittaker newform.
It is well-known that the RHS converges
absolutely in some nonempty vertical strip
and extends to a meromorphic function of $s$ on the
complex plane (see \cite{Ja72}).
Using the identity $W(a(-1)g) = \overline{W(g)}$ and the Iwasawa
decomposition,
we can rewrite this definition as
\begin{equation}\label{RS-integral-defn2}
  J\subscriptp(s) =  \int\limits_{ k \in K}  \int\limits_{y \in F^\times} |W|^2(a(y)k) |y|^{s-1} \, d^\times y \,  dk.
\end{equation} or alternatively, using the Bruhat decomposition (see~\eqref{e:bruhatmeasure}), as
\begin{equation}\label{RS-integral-defn3}
  J\subscriptp(s) =
  \frac{\zeta_p(2)}{\zeta_p(1)}\int_{x \in F}
  \max(1,|x|)^{-2s}
  \int_{y \in F^\times }
  |W|^2(a(y) w n(x))
  |y|^{s-1}\, d^\times y \, dx.
\end{equation}

The following important result is a
consequence
of Lemma
3.4.2 in \cite{michel-2009}.

\begin{proposition}\label{lemmamvext}
  Suppose that $\pi$ satisfies Condition \ref{cond:temperedness}.
  The integral $I(s)$,
  defined initially
  for $\Re(s) = 1/2$ or $s \in (2 \lambda(\pi), 1 - 2
  \lambda(\pi))$,
  extends to a meromorphic function of $s$ on the entire complex
  plane.
  We have an identity of meromorphic functions
  \begin{equation}\label{mvgeneral}I\subscriptp(s)
    = (1-p^{-1})^{-1}J\subscriptp(s)J\subscriptp(1-s).\end{equation}
\end{proposition}

\begin{proof}
  Denote by $\mathcal{D}= \{s \in \mathbb{C} :
  \Re(s) = 1/2\} \cup (2 \lambda(\pi), 1 - 2 \lambda(\pi))$
  the cross on which $I(s)$ was defined,
  and let $s \in \mathcal{D}$.
  Then $\pi_3$ is irreducible and unitarizable.
  We normalize the (unique up to scaling)
  $G$-invariant hermitian pairing $\langle , \rangle$
  on $\pi_3$
  so that $\langle f_s, f_s \rangle = 1$.\footnote{In the tempered case $\Re(s) = 1/2$,
    we have explicitly
    \[
    \langle f, f' \rangle
    = \int _{k \in K} f(k) \overline{f'(k)} \, d k
    \quad (f, f' \in \pi_3).
    \]
    When $s \in (0,1)$,
    the formula for the pairing is slightly more complicated.}
  With this normalization, the definition \eqref{trip-prod-defn}
  reads
  \begin{equation}\label{eq:13}
    I(s)
    =
    \int _{Z \backslash G}
    \langle g W, W \rangle^2
    \langle g f_s, f_s \rangle \, d g.
  \end{equation}
  This integral converges absolutely and locally uniformly
  on $\mathcal{D}$.

  We observe that $\langle g f_s, f_s \rangle$
  extends to an entire function of $s$,
  and in fact a polynomial function of $p^{\pm s}$;
  explicitly,
  $$\langle k_1 a(\varpi^m) k_2 f_s, f_s \rangle
  = p^{-m/2} (1+ p^{-1})^{-1} (\lambda_{s,m} - p^{-1}\lambda_{s,m-2})$$
  with $\lambda_{s,m}$ as in \eqref{eq:12}
  for all $k_1, k_2 \in K$, $m \ge 1$.
  Moreover, we have the majorization
  $|\langle g f_{s}, f_{s} \rangle|
  \leq \langle g f_{\sigma}, f_{\sigma}  \rangle \in
  \mathbb{R}_{\geq 0}$
  with $\sigma = \Re(s)$.
  Consequently, the integral  \eqref{eq:13}
  converges normally
  and defines a holomorphic function
  on the strip
  $\mathcal{D} ' = \{s \in \mathbb{C} : \Re(s) \in (2
  \lambda(\pi),
  1 - 2 \lambda(\pi)) \}$.

  The relation \eqref{mvgeneral} on the
  line $\Re(s) = 1/2$
  follows from Lemma
  3.4.2 in \cite{michel-2009}
  upon noting
  that
  $J\subscriptp(s) = \overline{J\subscriptp(1-s)}$
  whenever $\Re(s) = \onehalfbig$.
  Since both sides
  of \eqref{mvgeneral}
  vary analytically with $s$ on the strip $\mathcal{D}'$, we obtain at once the
  meromorphic
  continuation of $J(s)$ to the complex plane and the general
  case
  of the identity \eqref{eq:13}.
\end{proof}

Proposition \ref{lemmamvext}
is significant for our purposes
because it reduces the evaluation of
the integral $I\subscriptp(s)$,
which appears in Ichino's formula,
to that of the simpler integral $J\subscriptp(s)$.

\subsection{The local functional equation}\label{s:localfe}
Let $\pi = \mathcal{W}(\pi,\psi)$
be a generic irreducible admissible
unitarizable representation of $G$ with trivial central character,
let $W \in \pi$ be the normalized Whittaker newform,
and let $J(s)$ be the local Rankin--Selberg integral.
The main difficulty in computing
$J\subscriptp(s)$, and hence $I\subscriptp(s)$,
is that $W(g)$ has no simple
formula when
$a(\pi) \ge 2$.
In Section \ref{sec:proofs}, we will split the
integral~\eqref{RS-integral-defn3} defining $J(s)$ into several
pieces. Initially, we will be able to evaluate at least half of
these pieces. The key tool that will enable us to compute the
remaining pieces is the local functional equation for $\GL(2)
\times \GL(2)$, which we now recall
in a specialized form.
It is convenient to define
the \emph{normalized local Rankin--Selberg integral}
$$\Jcomplete(s) = \frac{J\subscriptp(s)
  \zeta_p(2s)}{L(\pi\subscriptp \times \pi\subscriptp, s)},$$
and to introduce the shorthand
$C\subscriptp = C(\pi \times\pi)$.

\begin{proposition}[Local functional equation for $\GL(2) \times
  \GL(2)$]\label{p:localfe}
  $\Jcomplete(s)$ extends to a polynomial
  function of $p^{\pm s}$ that satisfies the functional equation $\Jcomplete(s) = C\subscriptp^{s - \onehalfbig} \Jcomplete(1-s).$
\end{proposition}

\begin{proof}This follows from (1.1.5) of~\cite{MR533066} by taking the Schwartz function $\Phi$ to be the characteristic function of $\OF \times \OF$. We have used here that the epsilon factor $\eps(s, \pi\subscriptp \times \pi\subscriptp, \psi)$ equals $C\subscriptp^{ \onehalfbig-s}$. This follows from the fact that the local root number of $\pi\subscriptp\times \pi\subscriptp$ is equal to +1; see the proof of Prop. 2.1 of~\cite{prasadram}.
\end{proof}

Suppose that $\pi$ satisfies Condition \ref{cond:temperedness}. From Proposition~\ref{lemmamvext} and the definitions of $I^*(s)$ and $J^*(s)$, we readily derive the formula \begin{equation}\label{deftilip}\Icomplete(s)
  =   (1+p^{-1})^2 L(\ad  \pi\subscriptp,
  1)^2\Jcomplete(s)\Jcomplete(1-s).\end{equation}   By Proposition~\ref{p:localfe}, $\Jcomplete(s)$ is an entire function of $s$. It follows that $\Icomplete(s)$ is also entire as a function of $s$. By contrast, $I(s)$ may have poles.
Using soft analytic techniques, we deduce
from Proposition~\ref{p:localfe} the \emph{local convexity
  bound}
described in the introduction.

\begin{corollary}[Local convexity bound]\label{c:loconv} For $0
  \le \Re(s) \le 1$,
  we have $\Jcomplete(s) \ll C\subscriptp^{-1/2 + \Re(s)/2}$ and $\Icomplete(s) \ll C\subscriptp^{-1/2}$ with absolute implied constants.
\end{corollary}
\begin{proof}
  By~\eqref{deftilip}, it suffices to prove the first
  part of the statement. Using~\eqref{RS-integral-defn2} and the
  fact that $W(g)$ is $L^2$-normalized, we get the trivial bound
  $\Jcomplete(s) \ll 1$ for $\Re(s) = 1$.
  We transfer this to the bound $\Jcomplete(s) \ll C\subscriptp^{-1/2}$ for
  $\Re(s)=0$ via Proposition~\ref{p:localfe}.
  We interpolate these two bounds by the
  Phragmen--Lindel\"{o}f theorem,
  which in this context is nothing more than the maximum modulus
  principle,
  to deduce that
  $\Jcomplete(s)
  \ll C\subscriptp^{-1/2 + \Re(s)/2}$
  for all $s$ with $0 \leq \Re(s) \leq 1$.

\end{proof}

\subsection{Properties of Whittaker functions}
\label{sec:supp-whitt-newf}
Let $\pi = \mathcal{W}(\pi,\psi)$ be a generic irreducible
admissible unitarizable
representation of $G$ with trivial central character,
and $W \in \pi$  its normalized Whittaker newform.
The purpose of this section is to establish the key properties
of $W$ that will be used in our proof of Theorem
\ref{t:mainlocal}.

\begin{lemma}[Invariance of inner product on Whittaker model]
  \label{lem:invariance-inner-product-W-model}
  For each $g_1, g_2 \in G$, one has
  $\int _{y \in F^\times } |W|^2( a(y) g_1 ) \, d^\times y
  = \int _{y \in F^\times } |W|^2( a(y) g_2 )
  \, d^\times y = 1$.
\end{lemma}
\begin{proof}
  Since  $G$ acts on $\mathcal{W}(\pi,\psi)$
  by right translation,
  the first identity amounts to the fact that
  integration along $A$
  defines a $G$-invariant hermitian pairing
  on $\mathcal{W}(\pi,\psi)$ (see the beginning of Section
  \ref{sec:2.2}).
  The second identity follows from our assumption that $W$ is $L^2$-normalized.
\end{proof}

\begin{lemma}[Support
  condition\footnote{
    The reader looking to understand how our arguments would apply to
    slightly more general vectors
    $W' \in \pi$ might complain that this condition is very
    particular to the newform.
    We refer to Remark \ref{rmk:bulldoze} for a sketch of an alternative,
    more robust argument that does not make use of this condition.}]\label{keywhittakersupportlemma}
  Write $n = a(\pi)$, and suppose that $n \geq 2$.
  If $|x|^2 < \max(p^n,|y|)$ and $W(a(y) w n(x)) \neq  0$,
  then $|y|
  = p^n$.
\end{lemma}
Before embarking on the proof of this lemma,
we must introduce some notation and recall
the local $\GL(2)$ functional equation.
Let $\mu$ be a character of the unit group
$\mathfrak{o}^\times$.
We extend $\mu$ to a (unitary) character of $F^\times$
(non-canonically) by setting $\mu(\varpi) = 1$,
and henceforth denote this extension also by $\mu$.
We may write the standard $\eps$-factor
for $\pi \mu$
in the form
$\eps(\pi \mu,s,\psi)
= \eps(\pi \mu) C(\pi \mu)^{1/2-s}$
for some $\eps(\pi \mu)  = \eps(\pi \mu, \psi, 1/2) \in \mathbb{C}^1$,
where $C(\pi \mu) = p^{a(\pi \mu)}$
is as in Section \ref{sec:cond-l-funct};
for notational simplicity,
we suppress the dependence of
$\eps(\pi \mu)$
on our fixed choice of uniformizer $\varpi$ and unramified
additive character
$\psi$.

With this notation, the local $\GL(2)$ functional equation
(see \cite{MR0401654})  asserts that for each vector $W' \in
\mathcal{W}(\pi\subscriptp,\psi)$, each character $\mu$ of
$\mathfrak{o}^\times$, and each
complex number $s \in \mathbb{C}$,
the local zeta
integral
\[
Z(W',\mu,s) = \int_{F^\times} W'(a(y)) \mu(y) |y|^s \, d^\times y
\]
satisfies
\begin{equation}\label{gl2-loc-func-eqn}
  \frac{Z(W',\mu^{-1},s)}{L(\pi\subscriptp \twist \mu^{-1},\onehalf+s)} = \eps(\pi\subscriptp \twist \mu) C(\pi\subscriptp \twist \mu)^s \frac{Z(w  W',\mu,-s)}{L(\pi\subscriptp \twist \mu,\onehalf-s)}.
\end{equation}

\begin{proof}[Proof of Lemma \ref{keywhittakersupportlemma}]
  Suppose that $|x|^2 < \max(p^n,|y|)$.
  If $|x|^2 \geq p^n$,
  then $|y| > |x|^2 \geq p^n$,
  hence
  \[
  \max \left(
    \left\lvert \frac{x}{y} \right\rvert,
    \left\lvert \frac{x^2}{y} \right\rvert,
    p^n
    \left\lvert \frac{1}{y} \right\rvert
  \right) < 1.
  \]
  It follows that for each
  unit $u \in \mathfrak{o}^\times$,
  the matrix
  \[
  (a(y)wn(x))^{-1} n(u\varpi^{-1}) (a(y)wn(x))
  =
  \begin{bmatrix}
    1 + \frac{x}{y} u \varpi^{-1} &   \frac{x^2}{y}  u \varpi^{-1} \\
    -\frac{1}{y} u \varpi^{-1} & 1 - \frac{x}{y} u \varpi^{-1}
  \end{bmatrix}
  \]
  belongs to $K_0(\p^n)$.  Therefore $W(a(y)wn(x)) = \psi(
  u\varpi^{-1})W(a(y)wn(x))$ for all $u \in \OF^\times$.
  Since $\psi$ has conductor $\OF$, we see that
  $W(a(y)wn(x)) = 0$.

  It remains to consider the case that $|x|^2 < p^n$.
  Let $W' = w n(x)  W$.
  We wish to show that $W'(a(y)) = 0$ unless $|y| = p^n$.
  By Fourier inversion on the unit group $\mathfrak{o}^\times$,
  it is equivalent
  to show
  that for each character $\mu$ of $\mathfrak{o}^\times$,
  the zeta integral $Z(W',\mu^{-1},s)$
  is a constant multiple
  of $p^{ns}$, where the constant
  is allowed to depend upon $\mu$ but not upon $s$.

  It is a standard fact (see \cite{MR0447121,Sch02}) that
  the map $F^\times \ni y \mapsto W(a(y))$, and hence
  also the map
  \begin{equation}\label{restriction-of-wWprime-to-diag}
    F^\times  \ni y \mapsto (w  W')(a(y)) = (n(x)  W)(a(y))
    = W(a(y) n(x)) = \psi(x y) W(a(y)),
  \end{equation}
  is supported on $\mathfrak{o}^\times$,
  so that $c_0(\mu) := Z(w  W', \mu,-s )$ is
  independent of $s$; it is here that we have used
  the assumption $n \geq 2$.
  Therefore the functional equation
  \eqref{gl2-loc-func-eqn} reads
  \[
  Z(W',\mu^{-1},s) =
  c_0(\mu) \eps(\pi\subscriptp \twist \mu) C(\pi\subscriptp \twist \mu)^s
  \frac{L(\pi\subscriptp \twist \mu^{-1},\onehalf+s)}{L(\pi\subscriptp \twist \mu,\onehalf-s)},
  \]
  and we reduce to showing that
  $c_0(\mu) \neq 0$ implies
  that
  $C(\pi\subscriptp \twist \mu) = p^n$ and
  $L(\pi\subscriptp \twist \mu,s) = L(\pi\subscriptp \twist \mu^{-1},s) = 1$.

  The right-$a(\mathfrak{o}^\times)$-invariance of $W$
  implies that \eqref{restriction-of-wWprime-to-diag} is invariant under
  $\mathfrak{o}^\times \cap (1 + x^{-1} \mathfrak{o})$,
  hence $c_0(\mu) = 0$ unless
  $C(\mu) \leq |x|$, in which case $C(\mu)^{2} \leq |x|^2 < p^n$
  and $C(\pi\subscriptp \twist \mu) =
  p^n$.
  If $\pi\subscriptp$ is of Type $1$ or Type $2$, we deduce
  immediately that
  $L(\pi\subscriptp \twist \mu,s) = L(\pi\subscriptp \twist \mu^{-1},s) = 1$;
  in the other cases this holds by inspection.
\end{proof}

\begin{remark} A slight modification of the above argument
  implies that under the hypotheses of Lemma
  \ref{keywhittakersupportlemma},
  we have
  \begin{equation}\label{eq:15new}
    W(a(y) w n(x))
    = \sum_{
      \substack{
        \mu \in \widehat{\OF^\times}  \\
        C(\pi\subscriptp \twist \mu) = |y|
      }
    }
    \mu(y) \eps(\pi\subscriptp \twist \mu)
    G(x,\mu),
  \end{equation}
  where $G(x,\mu) = \int_{u \in U} \psi(x u) \mu(u)$
  is a Gauss-Ramanujan sum. Note that the
  characters $\mu$ contributing nontrivially to
  \eqref{eq:15new} are those for which $G(x,\mu) \neq 0$,
  which implies that $C(\mu) \leq |x|$.

\end{remark}

\subsection{The proofs}\label{sec:proofs}
Our aim in this section is to prove Theorem~\ref{t:mainlocal};
along the way,
we will also establish Proposition~\ref{propnN}.
Let $\pi$ satisfy Condition \ref{cond:temperedness}.
Recall the notation
$n = a(\pi)$ and $N = a(\pi \times \pi)$.
Suppose that $n \geq 2$.
By \eqref{mvgeneral},
the calculation of $I^*(s)$
reduces to that of $\Jcomplete(s)$.
Let
$T_m$ be the coefficient of $p^{m s}$ therein:
\begin{equation}\label{deftm}\Jcomplete(s) = \sum_{m \in \Z}T_m
  p^{ms}. \end{equation}
Recalling that $J^*(s)$ is a polynomial
in $p^{\pm s}$ and applying its functional equation
$J^*(s) = (p^N)^{s-1/2} J^*(1-s)$
(Proposition~\ref{p:localfe}),
we see that $T_m = 0$  for
almost all $m$
and
\begin{equation}\label{keyequation}
  T_{-m+N} = p^{m-\frac{N}{2}}T_{m}.
\end{equation}
Setting $s=1$ in~\eqref{deftm} and using the identity $J\subscriptp(1)=1$, we obtain
\begin{equation}\label{sumtm} \sum_m T_m p^m = \frac{\zeta_p(2)}{L(\pi\subscriptp \times \pi\subscriptp,1)}.
\end{equation}

Closely related to $T_m$ are the quantities $R_m$ defined
by \begin{equation}\label{defrm}J\subscriptp(s) = \sum_{m \in
    \Z}R_m p^{ms}. \end{equation} A linear  relation between the
sequences $T_m$ and $R_m$ follows immediately from the
definition $$\Jcomplete(s) = \frac{J\subscriptp(s)
  \zeta_p(2s)}{L(\pi\subscriptp \times \pi\subscriptp,s)}.$$ For
convenience,
we explicate this relation case-by-case in Table \ref{table1}.

\begin{table}
  \centering
  \caption{Relation between $T_m$ and $R_m$}\label{table1}
  \begin{tabular}{ccc}
    \text{Representation}&$L(\pi\subscriptp \times \pi\subscriptp,s)^{-1}\zeta_p(2s)$&$T_m$ \text{ in terms of } $R_m$ \\
    \hline
    \noalign{\medskip} $\mathrm{Type }  \ 1$& $1$    &   $R_m$ \\
    \noalign{\medskip} $\mathrm{Type } \  2$& $\frac1{1+p^{-s}}$   &   $\sum_{r=0}^\infty (-1)^r R_{m+r}$ \\
    \noalign{\medskip} $\mathrm{Type } \  3$& $\frac{(1-p^{-s})(1-p^{2s_0-s})(1-p^{-2s_0-s})}{1+p^{-s}} $  &    $R_m - \beta_{s_0}R_{m+1} - R_{m+2} + 2\beta_{s_0}\sum_{r=2}^\infty (-1)^r R_{m+r}$ \\
    \noalign{\medskip}  $\mathrm{Type } \  4$&  $ \frac{1-p^{-s}}{1+p^{-s}}$   & $R_m + 2 \sum_{r=1}^\infty (-1)^r R_{m+r}$ \\ \noalign{\medskip}  $\mathrm{Type } \  5$& $\frac{1-p^{-s-1}}{1+p^{-s}}$   & $R_m + (1+p^{-1}) \sum_{r=1}^\infty (-1)^r R_{m+r}$ \\ \noalign{\medskip}
  \end{tabular}

\end{table}

Let us now explain our strategy for computing $\Jcomplete(s)$. In view of~\eqref{deftm},~\eqref{keyequation}  and~\eqref{sumtm}, it suffices to compute $T_m$ for positive $m$. Using Table \ref{table1}, we reduce further to computing $R_m$ for positive $m$.

The definition \eqref{RS-integral-defn3}
of $J(s)$ implies that
\[
R_m
=
\frac{\zeta_p(2)}{\zeta_p(1)}
\mathop{
  \int _{x \in F}
  \int _{y \in F^\times }
} _{
  |y|/\max(1,|x|)^2 = p^m
}
|W|^2(a(y) w n(x))
|y|^{-1} \, d^\times y \, d x.
\]
Let $x, y$ be as in the integrand above,
and suppose that $W(a(y) w n(x)) \neq 0$.
Since $m > 0$,
we have $|x|^2 < |y|$.
By the support condition on $W$ (Lemma \ref{keywhittakersupportlemma}),
we deduce that $|y| = p^n$.
Therefore
\[
R_m
=
\frac{\zeta_p(2)}{\zeta_p(1)}
p^{-n}
\mathop{
  \int _{x \in F}
} _{
  \max(1,|x|)^2 = p^{n-m}
}
\int _{y \in F^\times }
|W|^2(a(y) w n(x))
\, d^\times y \, d x.
\]
By the invariance
of the inner product on $\mathcal{W}(\pi,\psi)$
and our assumption that $W$ is $L^2$-normalized
(Lemma \ref{lem:invariance-inner-product-W-model}),
we deduce that the integral over $y$ is identically $1$,
and in particular, independent of $x$.
We summarize thusly:



\begin{proposition}\label{propRm}
  Let $m$ be a positive integer.
  Then
  \[
  R_m = \frac{\zeta_p(2)}{\zeta_p(1)}
  p^{-n}
  \vol (
  \{x \in F : \max(1,|x|)^2 = p^{n - m}\},
  d x
  ).
  \]
  Explicitly,
  \begin{enumerate} \item
        $R_m = 0$
        if either
        \begin{itemize}
        \item $m>n$, or
        \item $1 \le m < n$ and $m-n$ is odd.
        \end{itemize}
      \item $R_n = \frac{p^{-n}}{1 + p^{-1}}$
      \item  $R_m =  p^{\frac{-n-m}{2}}\frac{1 - p^{-1}}{ 1 + p^{-1}}$  if $1 \le m < n$ and $m-n$ is  even.

      \end{enumerate}

    \end{proposition}

    In order to prove Theorem~\ref{t:mainlocal}, it suffices to evaluate $T_m$ for each
    $m$ and then use~\eqref{deftilip}
    and~\eqref{deftm}. From~\eqref{keyequation} and Proposition~\ref{propRm} we
    know that $T_m = 0$ if $m > n$ or $m < N-n$. For the remaining
    values of $m$, the evaluation of $T_m$ follows by collecting together
    ~\eqref{keyequation},~\eqref{sumtm} Proposition~\ref{propRm}, and the relations in Table~\ref{table1}. We record the results.

    \medskip

    \emph{Type 1. } In this case $\pi\subscriptp$ is a dihedral supercuspidal representation $\rho(E/F,\xi)$, associated to the unramified quadratic extension $E$ of $F$ and to a non-Galois-invariant character $\xi$ of $E^\times$. A standard computation~\cite{Sch02} shows that $n = 2 a(\xi)$ and $N =2 a(\xi^2)$. This shows that $n$ and $N$ are even and $N\le n$. As for $T_m$, we have $T_m  =  0$  if  $m >n$ or $m <N-n$ or $N-n \le m \le n$ and $m-n$ is odd;  $T_m = \frac{p^{-n}}{1+p^{-1}}$ if $m = n$; $T_m = \frac{p^{ - \frac{N}{2}}}{1+p^{-1}}$ if $m =  N-n$ and $T_m = p^{\frac{-n-m}{2}}\frac{1 - p^{-1}}{ 1 + p^{-1}}$ in the remaining cases.

\medskip

\emph{Type 2. } In this case we will prove that\begin{equation}\label{formulatype2}T_m = (-1)^{m+n}\ \frac{p^{ \lfloor \frac{-m- n}{2} \rfloor}}{1+p^{-1}}\end{equation} unless we have $m >n$ or $m <N-n$, in which case $T_m$ equals 0.
Indeed, from Proposition~\ref{propRm} and Table~\ref{table1}, we
see that~\eqref{formulatype2} holds for $m$ positive. Now,
if $N$ were odd, then we would
be able to use~\eqref{keyequation} to find $T_m$ for all $m$;
however, the resulting formula would contradict~\eqref{sumtm}. We conclude that $N$ is even. Now using~\eqref{keyequation} and~\eqref{sumtm} we see that $T_m$ is given by~\eqref{formulatype2} for all $m$ in the range $N-n \le m \le n$
and is 0 otherwise.

Next we show that $N=n+1$ whenever $n$ is odd. Indeed, if not, then we must have either $N\ge n + 2$ or $N \le n$. In the first case,~\eqref{formulatype2} implies that $T_1= T_0=0$, and hence (by the relation $R_0 = T_1 + T_0$) that $R_0 =0$. This is a contradiction since~\eqref{RS-integral-defn2} shows immediately that $R_0 \ge  \int\limits_K  |W(k)|^2  dk >0$  since $W(1) > 0$. In the second case,~\eqref{formulatype2} implies that $T_0 = (-1)^{n}\ \frac{p^{\lfloor - \frac{ n}{2} \rfloor}}{1+p^{-1}}$ and $T_1 = (-1)^{n+1}\ \frac{p^{ \lfloor \frac{ -n-1}{2} \rfloor}}{1+p^{-1}}.$  As $n$ is odd, we have $T_0 = - T_1$ and because $T_0 = R_0 - T_1$, this implies that $R_0=0$, once again leading to the same contradiction.

\medskip

\emph{Type 3. } In this case, we must have $n=2 a(\beta)$, $N=0$. We have $T_m =0$ if $m > n$ or $m < N-n = -n$. First assume that $p$ is odd; so $a(\beta) =1$.  We have $T_2 =\frac{p^{-2}}{1+p^{-1}}$ and $T_1 =- \beta_{s_0}\frac{p^{-2}}{1+p^{-1}}$. From~\eqref{keyequation}, it follows that $T_{-1} =- \beta_{s_0}\frac{p^{-1}}{1+p^{-1}}$ and $T_{-2} =\frac{1}{1+p^{-1}}$. It is left to calculate $T_0$. For that we use the fact that $\sum T_m p^m  = \frac{\zeta_p(2)}{L(\pi\subscriptp \times \pi\subscriptp,1)}.$ This gives us
$T_0 = p^{-1} \frac{1-2p^{-1} - p^{-2} + 2 \beta_{s_0}p^{-1}}{1+p^{-1}}.$ The case $p =2$ is similar, except that now $a(\beta) \in \{2,3 \}$. We compute $T_{n-2}$, $T_{n-3}$ from Table~\ref{table1} (since $n \ge 4$). We omit the details.

\medskip

\emph{Type 4. } In this case $n = 2a(\beta)$ and $N = 2 a(\beta^2)$. So $n$ and $N$ are even and $N \le n$. As always, we have  $T_m  =  0$  if  $m >n$ or $m <N-n$.  For the remaining cases, we compute $T_m  =   \frac{p^{-n}}{1+p^{-1}} $ if $m = n$; $T_m= \frac{p^{ - \frac{N}{2}}}{1+p^{-1}}$   if $ m =  N-n $; $T_m = p^{\frac{-n-m}{2}}$  if  $0 < n-m<2n-N$ and $m-n$ is even; $T_m = - \frac{2p^{\frac{-n-m-1}{2}}}{1+p^{-1}} $ if $0< n-m<2n-N$ and $m-n$  is odd.

\medskip

\emph{Type 5. }  In this case $N=2$ and $n = 2 a(\beta)$. As always, we have  $T_m  =  0$  if  $m >n$ or $m <N-n$. Moreover $T_m =  \frac{p^{-n}}{1+p^{-1}}$ if $m = n$; $T_m = \frac{p^{ - \frac{N}{2}}}{1+p^{-1}}$  if  $m =  N-n$; $T_m = p^{\frac{-n-m}{2}}\frac{1 + p^{-2}}{ 1 + p^{-1}}$  if  $0 < n-m<2n-N$  and  $m-n$ is even; $T_m = - p^{\frac{-n-m-1}{2}}$  if $0< n-m<2n-N$  and $m-n$  is odd .

\medskip

By substituting the above formulas into~\eqref{deftm}, we get an explicit formula for $J^*(s)$. This immediately proves Theorem~\ref{t:mainlocal} using the relation~\eqref{deftilip}. We note here the precise relation between $Q_{\pi, p}(s)$ and $J^*(s)$, $$ p^{\frac{Ns}{2}}Q_{\pi, p}(s) = (1+p^{-1}) p^{\frac{N}4 +\frac{n}2} J^*(s).$$
Note that along the way we have also proved Proposition~\ref{propnN}.

Finally, one can easily derive explicit formulas for $R_m$ for
all $m$ from those for $T_m$ calculated above and the relations
written down in Table~\ref{table1}. For example, for Type 1
representations, we have
\begin{enumerate} \item
  $R_m = 0$
  if either
  \begin{itemize}
  \item $m>n$, or $m<N-n$, or
  \item $N-n \le m < n$ and $m-n$ is odd.
  \end{itemize}
\item $R_n = \frac{p^{-n}}{1 + p^{-1}}$.
\item $R_{N-n} = \frac{p^{-N/2}}{1 + p^{-1}}$.
\item  $R_m =  p^{\frac{-n-m}{2}}\frac{1 - p^{-1}}{ 1 + p^{-1}}$  if $N-n \le m < n$ and $m-n$ is  even.
\end{enumerate}
The values of $R_m$ for $m \le 0$ are related to the Fourier
coefficients at various cusps of a newform corresponding to $\pi$
(see
Section~\ref{sec:cusps-gamm-four}).

\section{Proof of Theorem \ref{thm:main-refined}}\label{sec:proof-theor-refthm:m}

\subsection{Background and notations}\label{sec:modular-forms-their} In this subsection we collect some notation that will be used frequently in this section. For complete definitions and proofs, we refer the reader to
Serre \cite{MR0344216}, Shimura \cite{MR0314766},
Iwaniec \cite{Iw97,MR1942691}
and Atkin--Lehner \cite{MR0268123}. We note that some of this
(boilerplate) subsection is borrowed from~\cite{PDN-HQUE-LEVEL}.

\subsubsection*{General notations}
For an integer $n$ and a prime $p$, we let $n_p$ denote the
largest divisor of $n$ that is a power of $p$, and let $n_{\diamond}$
denote the largest integer such that $n_\diamond^2$ divides
$n$.
In words, $n_p$ is the ``$p$-part'' of $n$
(the maximal $p$-power divisor),
while $n_\diamond^2$ is the ``square part'' of $n$
(the maximal square divisor).
Note that $n_p = |n|_p^{-1}$ where $|n|_p$ denotes the
$p$-adic absolute value. We let $n_0$ denote the largest
squarefree divisor of $n$.
One could also write
$n_p = (n,p^\infty)$
and
$n_0 = (n, \prod_p p)$.
We have $n_\diamond = 1$ if and only if
$n_0 = n$ if and only if $n$ is squarefree,
but there is in general no simple relation
between $n_\diamond$ and $n_0$.

Given a finite collection of rational
numbers $\{\dotsc,a_i,\dotsc\}$,
the greatest common divisor \linebreak
$(\dotsc,a_i, \dotsc)$
(resp. least common multiple $[\dotsc,a_i,\dotsc]$)
is the unique nonnegative generator
of the (principal) $\mathbb{Z}$-submodules $\sum \mathbb{Z} a_i$
(resp. $\cap \mathbb{Z} a_i$) of $\mathbb{Q}$.
In particular, if $a$ and $b$ are two positive rational numbers
with prime factorizations $a = \prod p^{a_p}$,  $b =
\prod p^{b_p}$, then we have $(a,b) = \prod p^{\min(a_p,
  b_p)}$
and $[a,b] = \prod p^{\max(a_p, b_p)}$.
We write $a|b$ to denote that the ratio $b/a$ is an integer.

For each complex number $z$, we write $e(z) := e^{2 \pi
  i z}$. For each positive integer $n$, we let $\varphi(n)$ denote the Euler
phi function
$\varphi(n) = \# (\mathbb{Z}/n)^\times = \# \{a \in \mathbb{Z} : 1 \leq a
\leq n, (a,n) = 1\}$.
We let $\tau(n)$ denote the number of positive divisors of $n$ and $\omega(n)$ the number of prime divisors of $n$.

\subsubsection*{The upper-half plane}
We shall make use of notation
for the upper half-plane
$\mathbb{H} = \{z \in \mathbb{C} : \Im(z) > 0\}$, the modular
group $\Gamma = \SL(2,\mathbb{Z}) \circlearrowright \mathbb{H}$
acting by fractional linear transformations,
its congruence subgroup $\Gamma_0(q)$
consisting of those elements with lower-left entry
divisible by $q$,
the modular curve $Y_0(q) = \Gamma_0(q) \backslash \mathbb{H}$,
the Poincar\'{e} measure $d \mu = y^{-2} \, d x \, d y$,
and the stabilizer  $\Gamma_\infty = \{\pm \left(
  \begin{smallmatrix}
    1&n\\
    &1
  \end{smallmatrix}
\right) : n \in \mathbb{Z} \}$ in $\Gamma$ of $\infty
\in \mathbb{P}^1(\mathbb{R})$.
We denote a typical element of $\mathbb{H}$ as
$z = x + i y $ with $x, y \in \mathbb{R} $.

\subsubsection*{Holomorphic
  newforms}\label{sec:holomorphic-newforms}
Let $k$ be a positive even integer,
and let $\alpha$ be an element of $\GL(2,\mathbb{R})$ with
positive determinant; the element $\alpha$ acts on
$\mathbb{H}$ by fractional linear transformations
in the usual way.
Given a function $f :  \mathbb{H} \rightarrow \mathbb{C}$,
we denote by $f|_k \alpha$ the function
$z \mapsto \det(\alpha)^{k/2} j(\alpha,z)^{-k} f(\alpha
z)$,
where $j \left( \left[
    \begin{smallmatrix}
      a&b\\
      c&d
    \end{smallmatrix}
  \right], z  \right) = c z + d$.

 A \emph{holomorphic cusp form} on $\Gamma_0(q)$ of weight $k$ is
a holomorphic function $f : \mathbb{H} \rightarrow \mathbb{C} $
that satisfies $f|_k \gamma = f$ for all $\gamma \in
\Gamma_0(q)$ and vanishes at the cusps of $\Gamma_0(q)$.  A
\emph{holomorphic newform} is a cusp form that is an eigenform
of the algebra of Hecke operators and orthogonal with respect to
the Petersson inner product to the oldforms
(see \cite{MR0268123}).
We say
that a holomorphic newform $f$ is a \emph{normalized
  holomorphic newform} if moreover $\lambda_f(1) = 1$ in the
Fourier expansion
\begin{equation}\label{eq:f-fourier}
  y^{k/2} f(z) = \sum_{n \in \mathbb{N}}
  \frac{\lambda_f(n)}{\sqrt{n}}
  \kappa_f(n y) e(n x),
\end{equation}
where $\kappa_f(y) = y^{k/2} e^{- 2 \pi y}$; in that case the Fourier coefficients $\lambda_f(n)$
are real, multiplicative, and satisfy
\cite{deligne-l-adic,deligne-weil-1} the Deligne bound
$|\lambda_f(n)| \leq \tau(n)$.

Recall, from Section \ref{sec:main-result},
the definitions of the measures
$\mu$ and $\mu_f$ on $Y_0(1)$,
given by
\[
\mu(\phi)
= \int_{\Gamma \backslash \mathbb{H}}
\phi(z) \, \frac{d x \, d y }{y^2 },
\quad \mu_f(\phi) =
\int_{\Gamma_0(q) \backslash \mathbb{H}}
\phi(z) |f|^2(z) y^k \, \frac{d x \, d y}{y^2}
\]
for all bounded measurable functions $\phi$ on $Y_0(1)$.

\subsubsection*{Maass forms}
A \emph{Maass cusp form} (of level $1$, on $\Gamma_0(1)$, on
$Y_0(1)$, $\dotsc$)
is a $\Gamma$-invariant
eigenfunction of the hyperbolic Laplacian $\Delta := y^{-2}
(\partial_x^2 + \partial_y^2)$ on $\mathbb{H}$ that decays
rapidly at the cusp of $\Gamma$.  By the ``$\lambda_1 \geq
1/4$'' theorem (see \cite[Corollary 11.5]{Iw97}) there exists a real number $r
\in \mathbb{R}$ such that $(\Delta + 1/4 + r^2) \phi =
0$;
our arguments use only that $r \in \mathbb{R} \cup
i(-1/2,1/2)$,
which follows from the nonnegativity of $\Delta$.

A \emph{Maass
  eigencuspform} is a Maass cusp form that is an eigenfunction
of the  Hecke operators at all finite places and of the involution
$T_{-1} : \phi \mapsto [z \mapsto \phi(-\bar{z})]$; these operators commute with one another as well as with $\Delta$.
A Maass eigencuspform $\phi$ has
a Fourier expansion
\begin{equation}\label{eq:phi-fourier}
  \phi (z) = \sum _{n \in \mathbb{Z} _{\neq 0}}
  \frac{\lambda_\phi(n)}{\sqrt{|n|}}
  \kappa_{\phi}(n y) e( n x)
\end{equation}
where $\kappa_{\phi}(y) = 2 |y|^{1/2} K_{i r}(2 \pi |y|)
\sgn(y)^{\frac{1-\delta}{2}}$ with
$K_{ir}$ the standard $K$-Bessel function, $\sgn(y) = 1$ or $-1$
according
as $y$ is positive or negative,
and $\delta \in \{\pm 1\}$ the $T_{-1}$-eigenvalue of $\phi$;
note that the argument $n y$ of $\kappa_\phi(n y)$
in \eqref{eq:phi-fourier}
may be negative even if $y$ is positive.
A \emph{normalized Maass
  eigencuspform} further satisfies $\lambda_\phi(1) = 1$; in
that case the coefficients $\lambda_\phi(n)$ are real and
multiplicative.

Because $f(-\bar{z}) = \overline{f(z)}$ for each normalized holomorphic
newform $f$, we have $\mu_f(\phi) = 0$ whenever
$T_{-1} \phi = \delta \phi$ with $\delta = -1$.  Thus we shall assume throughout the rest of this
paper
that $\delta = 1$, i.e., that $\phi$ is an \emph{even} Maass form.

\subsubsection*{Eisenstein series}
Let $s \in \mathbb{C}$, $z \in \mathbb{H}$.
The \emph{real-analytic Eisenstein series}
$E(s,z) = \sum_{\Gamma_\infty \backslash \Gamma}
\Im(\gamma z)^s$
converges normally for $\Re(s) > 1$ and continues
meromorphically to the half-plane $\Re(s) \geq 1/2$ where
the map $s \mapsto E(s,z)$ is
holomorphic with the exception of a unique simple pole at $s =
1$ of constant residue $\res_{s=1} E(s,z) = \mu(1)^{-1}$.
The Eisenstein
series satisfies the invariance
$E(s,\gamma z)
= E(s,z)$ for all $\gamma \in \Gamma$.
When $\Re(s) = 1/2$ we call $E(s,z)$ a
\emph{unitary Eisenstein series}.
We write $E_s$ for the function $E_s(z) = E(s,z)$.

To each $\Psi \in C_c^\infty(\mathbb{R}_+^*)$,
we attach the \emph{incomplete Eisenstein series}
$E(\Psi,z) = \sum_{\gamma \in \Gamma _\infty \backslash \Gamma}
\Psi( \Im(\gamma z))$,
which descends to a compactly supported
function on $Y_0(1)$.
One can express $E(\Psi,z)$
as a weighted contour integral of $E(s,z)$
via Mellin inversion.

\subsection{An extension of Watson's formula}\label{sec:watson}
The general analytic properties
of triple product $L$-functions
on $\GL(2)$
follow from an integral representation
introduced by Garrett
\cite{MR881269} and further developed
by Piatetski-Shapiro--Rallis \cite{MR911357}.

Harris--Kudla \cite{harris-kudla-1991}
established a general ``triple product formula''
relating the (magnitude squared of the) integral of the product
of three automorphic forms (on quaternion algebras)
to the central value of their triple product $L$-function,
with proportionality constants given by somewhat
complicated local zeta integrals.
Gross and Kudla
\cite{MR1145805}
and Watson \cite{watson-2008}
evaluated sufficiently many
of the Harris--Kudla zeta integrals
to obtain a completely explicit triple product formula
for each triple of newforms having the
\emph{same squarefree level}.

Ichino \cite{MR2449948}
obtained a more general triple product formula
of the type considered by Harris--Kudla,
but in which the proportionality
constants are given by simpler integrals over the group
$\PGL_2(\mathbb{Q}_p)$.
Sufficiently many of those simpler integrals
were computed in \cite[Theorem 1.2]{MR2585578}
and
\cite[Lemma 4.2]{PDN-HQUE-LEVEL}
to derive an explicit triple product formula
for each triple of newforms of
(not necessarily the same) squarefree level
(see \cite[Remark 4.2]{PDN-HQUE-LEVEL}).

Our local calculations in Section~\ref{sec:local-calculations}
give
an explicit triple product formula for certain triples of
newforms of not necessarily squarefree level.
We state only the identity that we shall need.

\subsubsection*{Conventions regarding $L$-functions}
Let $\pi = \otimes \pi_v$
be one of the symbols
$\phi$, $f$,
$\ad \phi, \ad f$, or $f \times f \times \phi$;
here $v$ traverses the set of places of $\mathbb{Q}$.
One can attach a local factor
$L_v(\pi,s) = L(\pi_v,s)$ for each $v$.
We write $L(\pi , s) = \prod_p L_p(\pi ,s)$
for the finite part of the corresponding
global $L$-function
and $\Lambda(\pi, s)
= L_\infty(\pi, s) L(\pi,s)
= \prod_v L_v(\pi, s)$
for its completion.
The functional equation relates
$L(\pi,s)$
and $L(\pi,1-s)$.

For the convenience of the reader,
we collect here some references
for the definitions of $L(\pi,s)$ with $\pi$ as above.
Watson \cite[Section 3.1]{watson-2008}
is a good reference
for squarefree levels.
In general,
the standard $L$-functions attached to $\pi = f$ and $\pi =
\phi$
may be found
in a number of sources
(see for instance \cite{MR0379375,MR0401654,MR1431508}).
Since $\phi$ has trivial central character and is everywhere unramified,
we may write $L_v(\phi,s) = \zeta_v(s+ s_0) \zeta_v(s-s_0)$
for some $s_0 \in \mathbb{\mathbb{C}}$,
where $\zeta_\infty(s) = \pi^{-s/2} \Gamma(s/2)$
and $\zeta_p(s) = (1-p^{-s})^{-1}$.
Then $L_v(f \times f \times \phi,s)
= L_v(f \times f, s + s_0) L_v(f \times f, s -
s_0)$.
It is known that
$L_v(f \times f,s)$ factors as $L_v(\ad f,s) \zeta_v(s)$.
Finally, the local factors $L_v(\ad f, s)$ may be found in
\cite{MR546600}.

\begin{theorem}\label{thm:watson-ext}
  Let $\phi$ be a Maass eigencuspform of level $1$.
  Let $f$ be a holomorphic newform on $\Gamma_0(q)$, $q \in \mathbb{N}$.
  Then
  \begin{align*}
    &\frac{
      \left\lvert \int_{\Gamma_0(q) \backslash \mathbb{H}}
        \phi(z) |f|^2(z) y^k \, \frac{d x \, d y}{y^2}
      \right\rvert^2
    }{
      \left(\int_{\Gamma \backslash \mathbb{H}}
        |\phi|^2(z)  \, \frac{d x \, d y}{y^2} \right)
      \left( \int_{\Gamma_0(q) \backslash \mathbb{H}}
        |f|^2(z) y^k \, \frac{d x \, d y}{y^2}
      \right)^2
    }
    \\ &= \ \  \frac{1}{8 q}
    \frac{\Lambda(\phi \times f \times f,\tfrac{1}{2})}{
      \Lambda(\ad \phi,1) \Lambda(\ad f,1)^2} \prod_{p| q_\diamond}  \left(L_p(\ad f\subscriptp ,1) \cdot Q_{f, p}(s_{\phi,p})\right)^2,
  \end{align*}
  with
  $s = s_{\phi,p} \in \mathbb{C}$ chosen so that
  the $p$th normalized Hecke eigenvalue
  of $\phi$ is $p^{s-1/2} + p^{1/2-s}$
  and
  the local factors $Q_{f, p}(s_{\phi,p})$ as in
  Theorem~\ref{t:mainlocal}.
\end{theorem}
\begin{proof} Ichino's generalization of Watson's formula~\cite{MR2449948} reads
  \begin{equation}
    \frac{
      \left\lvert \int_{\Gamma_0(q) \backslash \mathbb{H}}
        \phi(z) |f|^2(z) y^k \, \frac{d x \, d y}{y^2}
      \right\rvert^2
    }{
      \left(\int_{\Gamma \backslash \mathbb{H}}
        |\phi|^2(z)  \, \frac{d x \, d y}{y^2} \right)
      \left( \int_{\Gamma_0(q) \backslash \mathbb{H}}
        |f|^2(z) y^k \, \frac{d x \, d y}{y^2}
      \right)^2
    }
    = \frac{1}{8} \frac{\Lambda ( f \times f \times
      \phi,\onehalf )}{
      \Lambda(\ad \phi, 1) \Lambda ( \ad f, 1)^2}
    \prod \Ivcomplete,
  \end{equation}
  where $\Ipcomplete$ was defined and explicitly calculated in Section
  \ref{sec:local-calculations}
  and $\Iinfcomplete \in \{0,1,2\}$ (see
  \cite{watson-2008}).
  In our case, $\Iinfcomplete = 1$.
  The result now follows
  from Theorems~\ref{t:oldeasy} and~\ref{t:mainlocal}.
\end{proof}
\begin{remark}
  \label{rmk:triple-product-formula-also-for-eis}
  A conclusion analogous to that of Theorem \ref{thm:watson-ext}
  holds also when $\phi = E_s$ is an Eisenstein
  series, in which case the computation follows
  more directly from the Rankin--Selberg method
  and the calculations of Section \ref{sec:local-calculations}.
  See also \cite[Section 4.4]{michel-2009}.
\end{remark}

\subsection{Bound for $D_f(\phi)$ in terms of $L$-functions}
\label{sec:bound-d_fphi-terms}

We briefly recall the setup for Theorem
\ref{thm:main-refined}.
Let $f$ be a holomorphic newform of weight $k \in 2 \mathbb{N}$
on $\Gamma_0(q)$.
We assume without loss of generality that $f$ is a normalized newform.
Fix a Maass eigencuspform  or incomplete Eisenstein series
$\phi$ on $Y_0(1) = \Gamma_0(1) \backslash \mathbb{H}$.
We wish to prove the bound asserted by
Theorem
\ref{thm:main-refined},
i.e., that
$$
D_f(\phi)
:= \frac{\mu_f(\phi)}{\mu_f(1)}
- \frac{\mu(\phi)}{\mu(1)}
\ll_\phi
(q/q_0)^{-\delta_1} \log(q k)^{-\delta_2}
$$
for some $\delta_1, \delta_2 > 0$,
with $q_0$ the largest squarefree divisor of $q$.
For simplicity, we treat in detail
only the case that $\phi$ is a Maass eigencuspform,
since the changes required to treat incomplete Eisenstein series
are exactly as in \cite{PDN-HQUE-LEVEL}.\footnote{However, one obtains different numerical values
  for $\delta_1, \delta_2$ when $\phi$ is an incomplete
  Eisenstein series;
  see the statement of Theorem \ref{thm:main2}.}

We collect first an upper bound for
$D_f(\phi)$ obtained by combining
the extension of Watson's formula (Theorem~\ref{thm:watson-ext})
with Soundararajan's weak subconvex bounds \cite{soundararajan-2008}.

\begin{proposition}
  \label{prop:watson-ext-upper-bound}
  For each holomorphic newform $f$ on $\Gamma_0(q)$
  and each Maass eigencuspform $\phi$ (of level $1$), we have
  \[
  \left\lvert
    D_f(\phi)
  \right\rvert^2 \ll_\phi
  \frac{1}{q}
  \frac{
    \Lambda (f \times f \times
    \phi,\onehalf )
  }{
    \Lambda ( \ad f, 1)^2
  } \
  10^{5\omega(q/\sqrt{C})}\tau(q/\sqrt{C})^2
  (q/\sqrt{C})^{2 \theta}.
  \] where $\theta \in [0,7/64]$ (see \cite{MR1937203}) is a bound
  towards the Ramanujan conjecture for Maass forms on
  $\SL_2(\mathbb{Z}) \backslash \mathbb{H}$.
\end{proposition}
\begin{proof}
  Let $C$ be the (finite) conductor of $f \times f$.
  Then $C$ is a perfect square,
  and $\sqrt{C}$ divides $q$.
  The result now follows from Theorem~\ref{thm:watson-ext} and the bounds of Corollary \ref{cor:something-about-implied-bound-on-Ip}.
\end{proof}

The analytic conductor of $f \times f \times \phi$
is $\asymp C^2 k^4$, so the arguments of Soundararajan \cite{soundararajan-2008}
imply that
\[
L( f \times f \times \phi , \onehalf)
\ll \frac{\sqrt{C} k }{\log(C k)^{1-\eps}}.
\]
By Stirling's formula as in \cite[Proof of Cor
1]{soundararajan-2008},
we deduce:
\begin{proposition}
  \label{prop:bound-a-la-sound}
  \begin{equation}\label{eq:yay-bound1}
    \left\lvert D_f(\phi) \right\rvert ^2
    \ll_\phi
    \frac{1}{L(\ad f,1)^2}
    \frac{10^{5\omega(q/\sqrt{C})}}{ \log(C k)^{1-\eps}}
    \frac{
      \tau(q/\sqrt{C})^2
    }{
      (q/\sqrt{C})^{1 - 2 \theta}
    }.
  \end{equation}
\end{proposition}
Note that $q/\sqrt{C} \in \mathbb{N}$ (cf. Prop \ref{propnN}).
Furthermore, when $q$ is squarefree,
we have $C = q^2$,
so that the third factor on the RHS of \eqref{eq:yay-bound1}
is absent.
\begin{remark}
  The same bound holds when $\phi$ is a unitary Eisenstein series,
  and with uniform implied constants.
  By Mellin inversion, the bound holds also when $\phi$ is an incomplete
  Eisenstein series (c.f. \cite[Proof of Cor
  1]{soundararajan-2008}
  or \cite[Proof of Prop 5.3]{PDN-HQUE-LEVEL}).
\end{remark}

\subsection{Cusps of $\Gamma_0(q)$ and Fourier expansions}
\label{sec:cusps-gamm-four}
We collect some (to the best of our knowledge, non-standard) information
concerning the Fourier expansions of newforms
at arbitrary
cusps of $\Gamma_0(q)$ (\S\ref{sec:fourier-expansions-new-stuff}).
To illuminate that discussion,
we take some time to recall in detail
certain comparatively standard facts concerning the cusps of
$\Gamma_0(q)$ themselves (\S\ref{sec:background-cusps}).

\subsubsection{Background on cusps}\label{sec:background-cusps}
The group $G := \PGL_2^+(\mathbb{R})$
acts on the upper half-plane $\mathbb{H}$ and its boundary
$\mathbb{P}^1(\mathbb{R})$
by fractional linear transformations.
For each lattice (i.e., discrete subgroup of finite covolume)
$\Delta < G := \PGL_2^+(\mathbb{R})$,
let $\mathcal{P}(\Delta)$
denote the set of boundary points
$\mathfrak{a} \in
\mathbb{P}^1(\mathbb{R})$
stabilized by a nonscalar element of $\Delta$;
one might call $\mathcal{P}(\Delta)$ the set of \emph{parabolic
  vertices}
of $\Delta$.
Equivalently, for each
$\mathfrak{a} \in \mathbb{P}^1(\mathbb{R})$,
let
$U_\mathfrak{a}$
denote the unipotent radical
of the parabolic subgroup $P_\mathfrak{a} =
\Stab_G(\mathfrak{a})$.
Then $\mathcal{P}(\Delta) = \{\mathfrak{a} \in
\mathbb{P}^1(\mathbb{R}):
\vol(U_\mathfrak{a}/U_\mathfrak{a} \cap \Delta) < \infty\}$.

The group $\Delta$ acts on $\mathcal{P}(\Delta)$,
and the orbit space $\mathcal{C}(\Delta) := \Delta \backslash
\mathcal{P}(\Delta)$
is called the set of \emph{cusps}
of $\Delta$.
One may take as representatives for $\mathcal{C}(\Delta)$
the set of parabolic vertices
of a given fundamental
polygon for $\Delta \backslash \mathbb{H}$.
Intrinsically,
$\mathcal{C}(\Delta)$ is in bijection with the set of
$\Delta$-conjugacy classes of parabolic subgroups
$P < G$ whose unipotent radical $U$ satisfies $\vol(U/U\cap \Delta) < \infty$.

Recall that
$\Gamma = \Gamma_0(1) = \SL_2(\mathbb{Z})$,
and set henceforth
$\Gamma' = \Gamma_0(q)$.
Then $\mathcal{P}(\Gamma) = \mathcal{P}(\Gamma ') = \mathbb{P}^1(\mathbb{Q})$.
The action of $\Gamma$ on $\mathcal{P}(\Gamma)$ is transitive,
and the stabilizer in $\Gamma$ (as well as in $\Gamma'$) of $\infty \in \mathcal{P}(\Gamma)$
is $\Gamma_\infty = \{\pm \left(
  \begin{smallmatrix}
    1&n\\
    &1
  \end{smallmatrix}
\right): n \in \mathbb{Z}\}$.
Thus we have
the left $\Gamma$-set
$\mathcal{P}(\Gamma) = \Gamma / \Gamma_\infty$,  the left $\Gamma'$-set $\mathcal{P}(\Gamma')= \Gamma / \Gamma_\infty$
and their orbit spaces
$\mathcal{C}(\Gamma) = \Gamma \backslash \Gamma /
\Gamma_\infty = \{1\}$,
$\mathcal{C}(\Gamma') = \Gamma' \backslash
\Gamma / \Gamma_\infty$.

For an arbitrary ring $R$, the group $\Gamma$
has a natural right action on
the set $\mathbb{P}^1(R)$,
realized as row vectors:
$[x:y] \cdot \left(
  \begin{smallmatrix}
    a&b\\
    c&d
  \end{smallmatrix}
\right) = [a x + c y : b x + d y]$.
The congruence subgroup $\Gamma_0(q)$ is
then the stabilizer in $\Gamma$ of $[0:1] \in
\mathbb{P}^1(\mathbb{Z}/q)$ .
The group $\Gamma$ acts transitively on
$\mathbb{P}^1(\mathbb{Z}) =
\mathbb{P}^1(\mathbb{Q})$,
hence on $\mathbb{P}^1(\mathbb{Z}/q)$,
and so we may identify $\Gamma ' \backslash \Gamma =
\mathbb{P}^1(\mathbb{Z}/q)$ as right $\Gamma$-sets.
Under this identification,
$\alpha = \left(
  \begin{smallmatrix}
    a&b\\
    c&d
  \end{smallmatrix}
\right) \in \Gamma$
corresponds to $[c:d] \in \mathbb{P}^1(\mathbb{Z}/q)$.
Two row vectors $[c:d]$ and $[c':d']$
with
$(c,d) = (c',d') = 1$
represent the same element of $\mathbb{P}^1(\mathbb{\Z}/q)$
if and only if there exists
$\lambda \in (\mathbb{Z}/q)^{\times} $ for which
$c' = \lambda c$ and $d' = \lambda d$.
Thus $\mathbb{P}^1(\mathbb{Z}/q)$
may be identified with the set of diagonal $(\mathbb{Z}/q)^{\times} $-orbits
on the set of ordered pairs $[c:d]$
of relatively prime residue classes $c,d \in \mathbb{Z}/q$.
In each such orbit there is a pair $[c:d]$ for which
$c$ divides $q$;\footnote{We say that the residue class $c \in (\Z/q)$ divides $q$ if its unique representative $c' \in   [1,q]$ divides $q$.}
if $[c,d]$ is one such pair,
then all such pairs
arise as $[c:\lambda d]$
for some $\lambda \in (\mathbb{Z}/q)^{\times} $
that satisfies $\lambda c \equiv c \pmod{q}$,
or equivalently $\lambda \equiv 1 \pmod{q/c}$.
Thus as $c$ traverses the set of positive divisors
of $q$ and $d$ traverses
$\{d \in \mathbb{Z}/(q/c) : (d,c,q/c) = 1\}$,
the vector $[c:d]$ traverses
$\mathbb{P}^1(\mathbb{Z}/q)$.\footnote{
  Note that
  $d \mapsto (d,c,q/c)$ is a well-defined
  function on $\mathbb{Z}/(q/c)$.
}

The element $\left(
  \begin{smallmatrix}
    1&n\\
    &1
  \end{smallmatrix}
\right)$ of $\Gamma_\infty$
sends $[c:d] \in \mathbb{P}^1(\mathbb{Z}/q)$
to $[c:d+nc]$.
The orbit of $[c:d]$
in $\mathbb{P}^1(\mathbb{Z}/q)$ may then be identified
with the set of all
$[c:d']$ where $d' \in \mathbb{Z}/(q/c)$ and $d' \equiv d
\pmod{c}$.
In summary, each section of the map $\Gamma \ni \left(
  \begin{smallmatrix}
    a&b\\
    c&d
  \end{smallmatrix}
\right) \mapsto [c:d] \in \mathbb{P}^1(\mathbb{Z}/q)$
gives rise to a commutative diagram
\begin{equation*}
  \begin{CD}
    \Gamma ' \backslash \Gamma  @>>> \Gamma ' \backslash \Gamma / \Gamma_\infty = \mathcal{C}(\Gamma ') \\
    @|  @| \\
    \mathbb{P}^1(\mathbb{Z}/q) @>>> \mathbb{P}^1(\mathbb{Z}/q) / \Gamma_\infty \\
    @|  @| \\
    \{[c:d] : c | q, d \in \mathbb{Z}/(q/c), (d,c,q/c) = 1\}
    @>>>
    \{[c:d] : c | q, d \in (\mathbb{Z}/(c,q/c))^{\times} \}. \\
  \end{CD}
\end{equation*}
When $c | q$ and $d \in (\mathbb{Z}/(c,q/c))^{\times} $,
we henceforth write $\mathfrak{a}_{d/c} \in \mathcal{C}(\Gamma ')$ for the corresponding
cusp. It corresponds to $a/c \in \mathbb{P}^1(\Q)$ where $a$ is an integer with $(a,c) = 1$ and $ad \equiv 1 \pmod{(c, q/c)}.$

Thinking of $[c:d]$ as the ``fraction'' $d/c$,
we define the \emph{denominator} of the cusp $\mathfrak{a}_{d/c}$
to be $c$, which is by assumption a positive divisor of $q$.

The \emph{width} of a cusp $\mathfrak{a} \in \mathcal{C}(\Gamma ')$
is the index
$w_\mathfrak{a} = [\Stab_{\Gamma } (\mathfrak{a}) : \Stab_{\Gamma '} (\mathfrak{a}) ]$
of
its $\Gamma'$-stabilizer in its
$\Gamma$-stabilizer.\footnote{For more general subgroups than
  $\Gamma_0(q)$, one should replace ``$\Gamma '$-stabilizer''
  with ``$\Gamma ' \cdot \{\pm 1\}$-stabilizer''.}
Equivalently, if we take as a fundamental domain for $\Gamma '
\backslash \mathbb{H}$
a union of translates of fundamental domains for $\Gamma$,
then the width of $\mathfrak{a}$ is the number of such
translates that touch $\mathfrak{a}$
(regarded as a $\Gamma '$-orbit of parabolic vertices);
in other words, it is the cardinality of the fiber
above $\mathfrak{a}$
under the projection $\Gamma ' \backslash \Gamma \rightarrow \mathcal{C}(\Gamma')$.
Let us write $\pi$ for the bottom horizontal arrow in the above
diagram.  Then the width of $\mathfrak{a}_{d/c}$ is
\[\# \pi^{-1}(\mathfrak{a}_{d/c})
=
\frac{
  (q/c) (c,q/c)^{-1} \varphi((c,q/c))
}{
  \varphi((c,q/c))
}
= \frac{q/c}{(c,q/c)}
= \frac{q}{(c^2,q)}
= [q/c^2,1].
\]

We now write simply $\mathcal{C} = \mathcal{C}(\Gamma ')$
for the set of cusps of $\Gamma '$,
which we enumerate as $\mathcal{C} = \{ \mathfrak{a}_j \}_j$.
Write $c_j$ for the denominator of $\mathfrak{a}_j$,
and $w_j = [q/c_j^2,1]$ for its width.
For each positive divisor $c$ of $q$,
let
\[
\mathcal{C} [c] := \{\mathfrak{a}_j \in \mathcal{C} :
c_j = c\}
\]
denote the set of cusps of denominator $c$.
It follows from the above diagram
that $\# \mathcal{C} [c] = \varphi((c,q/c))$.

Choose an element $\tau_j \in \Gamma$
representing the double coset $\mathfrak{a}_j \in \Gamma '
\backslash \Gamma / \Gamma_\infty$.
If $\mathfrak{a}_j = \mathfrak{a}_{d/c}$, then we may take
$\tau_j = \left(
  \begin{smallmatrix}
    *&*\\
    c&d'
  \end{smallmatrix}
\right)$ for any integer $d'$ for which $(d',c) = 1$ and $d ' \equiv d \pmod{(c,q/c)}$.
The $\tau_j$ so-obtained form a set of representatives for
$\Gamma ' \backslash \Gamma / \Gamma_\infty$.
Intrinsically,
the width
of
$\mathfrak{a}_j$
is given by
$w_j = [ \Gamma_\infty : \Gamma_\infty \cap
\tau_j^{-1} \Gamma ' \tau_j]$.
The \emph{scaling matrix} of $\mathfrak{a}_j$ is
\begin{equation}\label{eq:2}
  \sigma_j = \tau_j
  \begin{bmatrix}
    w_j &  \\
     & 1
  \end{bmatrix}
\end{equation}
which has the property
$B \cap \sigma_j^{-1} \Gamma' \sigma_j = \Gamma_\infty$
with $B = \left\{ \left(
    \begin{smallmatrix}
      *&*\\
      &*
    \end{smallmatrix}
  \right) \right\} < G$.
To put it another way, for each $z \in \mathbb{H}$,
let us write $z_j = x_j + i y_j$ for the change of variable
$z_j := \sigma_j^{-1} z$
and $\Gamma '_{j} = \Stab_{\Gamma '}(\mathfrak{a}_j)$.
Then
each element $\gamma \in \Gamma '$
satisfying $(\gamma z)_j = z_j + 1$
generates $\Gamma_j'$.
In other words,
$z \mapsto z_j$ is a proper isometry
of $\mathbb{H}$ under which $z_j \mapsto z_j + 1$
corresponds to the action on $z$ by a generator
for $\Gamma '_j$.

\subsubsection{Fourier expansions}\label{sec:fourier-expansions-new-stuff}
We now turn to explicating the Fourier expansion
of $|f|^2$ at the cusp $\mathfrak{a}_j$,
or equivalently
that of $|f|^2(z)$ regarded as a function of
the variable $z_j$.
Recall the weight $k$ slash operation:
for $\alpha \in \GL_2^+(\mathbb{R})$,
set $f|_k \alpha(z) = \det(\alpha)^{k/2}
j(\alpha,z)^{-k} f(\alpha z)$,
where $j \left( \left[
    \begin{smallmatrix}
      a&b\\
      c&d
    \end{smallmatrix}
  \right], z  \right) = c z + d$.
We then have
$|f|^2(z) y^k = |f|^2(\sigma_j z_j) \Im(\sigma_j z_j)^k
= \left\lvert f |_k \sigma_j \right\rvert^2(z_j)
y_j^k$,
and may write
\begin{equation}\label{eq:f-slashed-fourier-exp}
  f|_k \sigma_j(z_j)
  = y_j^{-k/2} \sum_{n \in \mathbb{N}}
  \frac{\lambda_j(n)}{\sqrt{n}}
  \kappa(n y_j) e(n x_j)
\end{equation}
for $\kappa(y) = y^{k/2} e^{-2 \pi y}$
($y \in \mathbb{R}^\times_+$)
and some coefficients $\lambda_j(n) \in \mathbb{C}$. In the special case,
$\mathfrak{a}_j = \infty$, we note that $\lambda_j(n) =
\lambda(n)$.
In general, the notation $\lambda_j(n)$ is slightly
misleading because $\lambda_j(n)$ depends not only on
the cusp $\mathfrak{a}_j$, but also on the choice of scaling
matrix $\tau_j$.
However, if $\lambda_j'(n)$ denotes the coefficient
obtained by a different choice $\tau_j'$,
then one has $\lambda_j'(n) = e(b n/w_j) \lambda_j(n)$
for some integer $b$.

The coefficients $\lambda_j(n)$
seem easiest to describe
by working adelically. For background on adeles and adelization of automorphic forms, we refer the reader to~\cite{MR0379375}. We recall the following notation from
Section~\ref{sec:local-calculations}:
\[
w =  \mat{0}{1}{-1}{0},
\quad
a(y) = \mat{y}{}{}{1},
\quad
n(x) = \mat{1}{x}{}{1},
\quad \text{ and }
z(y) = \mat{y}{}{}{y}.
\]
Let $\hat{\mathbb{Z}} = \varprojlim \mathbb{Z}/n = \prod
\mathbb{Z}_p$,
$\hat{\mathbb{Q}} = \hat{\mathbb{Z}} \otimes_\mathbb{Z}
\mathbb{Q} = \prod '\mathbb{Q}_p$
and $\mathbb{A} = \mathbb{R} \times \hat{\mathbb{Q}}$.
To $f$ one attaches a function
$F : \GL_2(\mathbb{A}) \rightarrow \mathbb{C}$
in the following standard way.
By strong approximation,
every element of $\GL_2(\mathbb{A})$
may be expressed in the form
$\gamma g_\infty \kappa_0$
for some $\gamma \in \GL_2(\mathbb{Q})$,
$g_\infty \in \GL_2(\mathbb{R})^+$
and $\kappa_0 \in K_0(q) = \{\left[
  \begin{smallmatrix}
    a&b\\
    c&d
  \end{smallmatrix}
\right] \in \GL_2(\hat{\mathbb{Z}}) : q \mid c \}$.
Then
$F(\gamma g_\infty \kappa_0)
= f|_kg_\infty(i)$.
Recall that $\sigma_j \in \GL_2(\mathbb{Q})^+$.
Let $i_\infty : \GL_2(\mathbb{Q}) \hookrightarrow \GL_2(\mathbb{R})
\hookrightarrow \GL_2(\mathbb{A})$
and $i_{\fin} :  \GL_2(\mathbb{Q}) \hookrightarrow \GL_2(\hat{\mathbb{Q}})
\hookrightarrow \GL_2(\mathbb{A})$
be the natural inclusions.
If $g_z \in \GL_2(\mathbb{R})^+$ is chosen so that
$g_z i = z$,
then
$f|_k \sigma_j(z) =
f|_k \sigma_j g_z(i)
= F(\iota_\infty(\sigma_j) g_z)
= F(g_z \iota_{\fin}(\sigma_j^{-1}))$
by the left-$G(\mathbb{Q})$-invariance of $F$.
For $g \in \GL_2(\mathbb{A})$,
one has a Fourier expansion
\[
F(g)
= \sum_{n \in \mathbb{Q}_{\neq 0}}
W(a(n) g),
\]
where $W$ is a global Whittaker newform
corresponding to $f$; it is given explicitly by
$W(g) = \int_{x \in \mathbb{A} / \mathbb{Q}} F(n(x) g)
\psi(-x) \, d x$
where the integral is taken with respect to an invariant
probability
measure.
It satisfies $W(n(x) g)
= \psi(x) W(g)$
for all $x \in \mathbb{A}$,
where $0 \neq \psi = \prod \psi_v \in
\Hom(\mathbb{A}/\mathbb{Q},\mathbb{C}^1)$
is the additive character
for which $\psi_\infty(x) = e^{2 \pi i
  x}$.
The function $W$ factors as $\prod W_v$
over the places of $\mathbb{Q}$. We may pin down this
factorization
uniquely by requiring that $W_\infty(a(y)) = \kappa(y)$
and $W_p(1) = 1$ for all primes $p$.
Writing $z = x + i y$,
we may and shall assume that
$g_z = n(x) a(y)$.
Then
\[
y^{k/2} f|_k \sigma_j(z) =
F(g_z \iota _{\fin} (\sigma_j^{-1}))
= \sum_{n \in \mathbb{Q}_{\neq 0}}
\kappa(n y) e(n x)
\prod_p
W_p(a(n) \sigma_j^{-1}).\]
Here we identify $\sigma_j$ with its image
under the natural inclusion $G(\mathbb{Q}) \hookrightarrow
G(\mathbb{Q}_p)$.
If $p \nmid q$,
then $W_p$
is unramified
at $p$
and $\sigma_j \in \GL_2(\mathbb{Z}_p)$,
since $\sigma_j$ differs
from $\tau_j \in \SL_2(\mathbb{Z})$
by a diagonal matrix with integral entries dividing $q$ (and hence with determinant coprime to $p$);
thus $W_p(a(n) \sigma_j^{-1}) = W_p(a(n))$.
If we also have $p \nmid n$, then $a(n) \in
\GL_2(\mathbb{Z}_p)$,
and so $W_p(a(n))=1$.
Therefore
the expansion \eqref{eq:f-slashed-fourier-exp} holds with
\begin{equation} \label{lambdajwhittaker}
  \lambda_j(n) = \sqrt{n} \prod_{p| [n,q]}
  W_p(a(n) \sigma_j^{-1}) = \sqrt{n} \prod_{p| q}
  W_p(a(n) \sigma_j^{-1}) \prod_{p| \frac{n}{(n,q^\infty)}} W_p(a(n)) .
\end{equation}

Let us spell out \eqref{lambdajwhittaker} a bit more precisely.
Write $\tau_j = \left[
  \begin{smallmatrix}
    a&b\\
    c&d
  \end{smallmatrix}
\right]$, so that $\mathfrak{a} =
\mathfrak{a}_{d/c}$
in the notation introduced above.
The Bruhat decomposition of $\tau_j^{-1}$ reads
\[
\tau_j^{-1}
=
\begin{bmatrix}
  d & -b \\
  -c & a
\end{bmatrix}
=
\begin{bmatrix}
  -c &  \\
  & -c
\end{bmatrix}
n(-d/c) a(1/c^2) w n(-a/c),
\]
so that for $y \in \mathbb{Q}_p^{\times} $,
we have
\begin{align*}
  W_p(a(y) \sigma_j^{-1})
  &=
  W_p (a(y/[q/c^2,1]) n(-d/c) a(1/c^2) w n(-a/c)) \\
  &=
  W_p (n(-y d/[q/c,c]) a(y/[q,c^2]) w n(-a/c)) \\
  &=
  \psi _p \left( \frac{- d y}{[q/c,c]} \right)
  W_p ( a(y/[q,c^2]) w n(-a/c)).
\end{align*}
Note also that  $$\prod_{p| \frac{n}{(n,q^\infty)}} W_p(a(n))  = \prod_{p| \frac{n}{(n,q^\infty)}} W_p\left(a\left(\frac{n}{(n,q^\infty)}\right)\right) =  \lambda \left( \frac{n}{(n,q^\infty)} \right). $$ Recall here that $\lambda(m)$ is our notation for
the coefficient $\lambda_j(m)$ at the distinguished cusp $\mathfrak{a}_j = \infty$. From the above calculations, we deduce that
\begin{equation} \label{lambdajwhittaker2}
  \lambda_j(n) = \sqrt{n} \cdot
  e \left( \frac{d n}{[q/c,c]} \right)
  \lambda \left( \frac{n}{(n,q^\infty)} \right)
  \prod_{p| q}
  W_p ( a(n/[q,c^2]) w n(-a/c)).
\end{equation}

One can check that $\lambda_j$ is not multiplicative in general;
for example, it can happen that $\lambda_j(1) \neq 1$,
or even that $\lambda_j(mn) \lambda_j(1) \neq
\lambda_j(m)\lambda_j(n)$
for pairs of coprime integers $m,n$.
To circumvent
this lack of multiplicativity,
we work with the root-mean-square
of $\lambda_j$ taken over all cusps of a given denominator.
For each positive divisor $c$ of $q$, define
\begin{equation}\label{lambdaedef}\lambda_{[c]}(n) = \left(
    \frac1{\# \mathcal{C}[c]} \sum_{ \mathfrak{a}_j \in
      \mathcal{C}[c]} |\lambda_{j}(n)| ^2
  \right)^{1/2}.\end{equation}
An explicit formula
in terms of $\GL(2)$ Gauss sums
for the RHS of \eqref{lambdajwhittaker2},
and hence for $\lambda_j(n)$,
may be derived following the method of Section \ref{sec:supp-whitt-newf}.
For our purposes, it suffices
(by Cauchy--Schwarz; see Section \ref{sec:outline-proof})
to evaluate the simpler averages
$\lambda_{[c]}(n)$.
It turns out that these averages
are multiplicative in a certain non-conventional sense:

\begin{definition}
  Let us call an arithmetic function $f: \N \rightarrow \C$
  \emph{factorizable} if it can be written as a product
  $f = \prod_pf_p$
  over the primes,
  where the $f_p :\N \rightarrow \C$ satisfy
  \begin{enumerate}
  \item $f_p(n) = f_p(n_p)$ for all $n \in \N$ and all $p$,\footnote{Recall that $n_p$ denote the
      largest divisor of $n$ that is a power of $p$.}
    and
  \item $f_p(1) = 1$ for all but finitely many $p$.
  \end{enumerate}
\end{definition}
\begin{remark}
  Every multiplicative\footnote{Recall that an arithmetic
    function $f$ is multiplicative if  $f(mn) = f(m)f(n)$
    whenever $(m,n)=1$.} function is factorizable
  (take $f_p(n) = f(n_p)$),
  and every factorizable function $f$  satisfies
  \begin{equation}
    \label{weakmult}f(mn)f(1) = f(m)f(n) \text{
      whenever } (m, n)=1,
  \end{equation}
  but neither of these implications is reversible.
  A factorizable function $f$ is multiplicative
  if and only if $f(1) =1$.
  Many non-factorizable functions $f$ satisfy
  \eqref{weakmult},
  but
  if $f(1) \neq 0$, then $f$ is factorizable
  if and only if it satisfies \eqref{weakmult},
  in which case $n \mapsto f(n)/f(1)$ is multiplicative.
\end{remark}


\begin{lemma}
  \label{lem:factorize-fourier-coeff}
  Let $c$ be a positive divisor of $q$.
  The function $n \mapsto \lambda _{[c]}(n)$ is factorizable:
  $$\lambda _{[c]}(n)
  = \prod_{p} \lambda_{[c],p}(n)$$
  for each $n \in \mathbb{N}$,
  where
  $\lambda_{[c],p} : \mathbb{N} \rightarrow
  \mathbb{R}_{\geq 0}$
  is defined by
  \begin{equation}\label{lambdaewhittaker}
    \lambda_{[c],p}(n) = \begin{cases}
      |\lambda(n_p)|
      = |n|_p^{-1/2} | W_p|(a(n))
      & p \nmid q, \\
      |n|_p^{-1/2} \left(
        \int _{u \in \mathbb{Z}_p ^\times }
        |W _p|^2
        \left(
          a \left( \frac{u n}{[q,c^2]} \right)
          w n (1/c)
        \right)
        \, d^\times u
      \right)^{1/2} & p \mid q.
    \end{cases}
  \end{equation}
\end{lemma}
\begin{proof}
  For each $\mathfrak{a}_j \in \mathcal{C}[c]$,
  let us write
  $\tau_j = \left(
    \begin{smallmatrix}
      a_j &*\\
      c& *
    \end{smallmatrix}
  \right)$.
  Recall that as $\mathfrak{a}_j$ traverses $\mathcal{C}[c]$,
  the lower-right entry of $\tau_j$ traverses
  $(\mathbb{Z}/(c,q/c))^\times$,
  hence so does the upper-left entry $a_j$.
  The formula \eqref{lambdajwhittaker2}
  and the definition \eqref{lambdaedef} imply that
  \begin{equation}\label{eq:4}
    \lambda _{[c]}(n) =  n^{1/2}\left|\lambda \left( \frac{n}{(n,q^\infty)} \right)\right| \left(
      \frac1{\# \mathcal{C}[c]} \sum_{ \mathfrak{a}_j \in
        \mathcal{C}[c]} \prod_{p|q} |W_p|^2 ( a(n/[q,c^2]) w n(-a_j/c))
    \right)^{1/2}.
  \end{equation}
  We treat the three factors on the RHS successively;
  in doing so, we shall repeatedly
  invoke
  the right-$a(\mathbb{Z}_p^\times)$-invariance of $W_p$
  for each prime $p$.
  The first factor may be written $n^{1/2} = \prod_p |n|_p^{-1/2}$.
  The second is
  $\left|\lambda \left( \frac{n}{(n,q^\infty)} \right)\right| =
  \prod_{p \nmid q} |W_p|(a(n))$.
  For the third,
  note that the average
  over $\mathcal{C}[c]$, or equivalently,
  over $a_j \in (\mathbb{Z}/(c,q/c))^\times$,
  lifts to an Eulerian integral over $\prod_{p \mid q}
  \mathbb{Z}_p^\times$:
  \[
  \frac1{\# \mathcal{C}[c]} \sum_{ \mathfrak{a}_j \in
    \mathcal{C}[c]} \prod_{p|q} |W_p|^2 ( a(n/[q,c^2]) w
  n(-a_j/c))
  =
  \prod_{p|q}
  \int _{u \in \mathbb{Z}_p^\times }
  | W_p|^2 ( a(n/[q,c^2]) w
  n(-u/c)) \, d^\times u.
  \]
  The identity
  $w n(-u/c)
  \equiv a(-1/u) w n(1/c) \pmod{Z(\mathbb{Q}_p)}$
  and substitution $u \mapsto -1/u$
  allows us to rewrite the above as
  \[
  \prod_{p|q}
  \int _{u \in \mathbb{Z}_p^\times }
  | W_p|^2 ( a \left( \frac{u n}{[q, c^2]} \right) w
  n(1/c)) \, d^\times u.
  \]
  Collecting the identities
  obtained  for each
  of the three factors in \eqref{eq:4}, we deduce
  \[
  \lambda_{[c]}(n)
  = \prod_p \left(
  |n|_p^{-1/2}
  \times
  \begin{cases}
    |W_p|(a(n)) & p \nmid q \\
    \left(

      \int _{u \in \mathbb{Z}_p^\times }
      | W_p |^2( a \left( \frac{u n}{[q, c^2]} \right) w
      n(1/c)) \, d^\times u
    \right)^{1/2}
    & p \mid q \\
  \end{cases}\right).
  \]
    This establishes the claimed formula $\lambda _{[c]}(n)
    = \prod_{p} \lambda_{[c],p}(n)$.
    It is clear from the definition
    that $\lambda_{[c],p}(n) = \lambda_{[c],p}(n_p)$
    for all $p$
    and that $\lambda_{[c],p}(1) = 1$ for all $p$ not dividing $q$.
  \end{proof}

  \begin{remark}
    \label{rmk:lambda-c-factorization}
    It follows from the right-$a(\mathbb{Z}_p^\times)$-invariance
    of $W$
    that
    $\lambda_{[c],p}(n)
    = \lambda_{[c_p],p}(n_p)$.
  \end{remark}

  \begin{lemma}
    \label{lem:lambda-atkin-lehner-duh}
    For each prime $p$, each $c|q$ and each $n \in \mathbb{N}$,
    we have $\lambda_{[c],p}(n) = \lambda_{[q/c],p}(n)$.
  \end{lemma}
  \begin{proof}
    Let $w_q = \left[\begin{smallmatrix}0&1\\ - q
        &0 \end{smallmatrix}\right]$.
    Then $w_q$ acts as the Atkin--Lehner operator on the newvector $W_p$,
    and so $W_p(gw_q) = \pm W_p(g)$ for all $g \in \GL_2(\Q_p)$.
    Since
    \[
    a \left( \frac{y}{[q,c^2]} \right)
    w n(1/c) w_q
    =
    z \left( \frac{q}{c} \right)
    n(-u)
    a \left( \frac{- y}{[q,(q/c)^2]} \right)
    w n(1/(q/c))
    a(-1)
    \]
    for each $y \in \mathbb{Q}_p^\times$,
    the lemma follows
    from the left-$Z(\mathbb{Q}_p) N(\mathbb{Q}_p)$-equivariance
    and right-$A(\mathbb{Z}_p)$-invariance
    of $W_p$.
  \end{proof}

  \begin{remark}
    \label{rmk:dfdfsadfs}
    When $q$ is a prime power,
    the classical content of the proof of the above lemma
    is that
    for $a d \equiv 1 \pod{q}$,
    the operator
    $z \mapsto -1/(q z)$
    takes $a/c$ to $-a^{-1} / (q c^{-1}) \equiv -d/(q c^{-1})
    \pmod{\mathbb{Z}}$.
  \end{remark}

  We are now in a position to compute
  $\lambda_{[c],p}$ exactly. We do this in Proposition~\ref{keypropositionfourier}. The  quantities $R_{m,p}$ that appear in the statement below are the coefficients ``$R_m$'' that were defined in~\eqref{defrm} and later  computed exactly\footnote{We wrote down exact formulas only for $T_m$ but similar ones for $ R_m$ can be easily worked out using Table~\ref{table1}.} for all representations of $\PGL_2(\Q_p)$ with conductor at least $p^2$.

  \begin{proposition}\label{keypropositionfourier} Let $c$ be a
    positive divisor of $q$,
    $p$ a prime divisor of $q$,
    and $n$ a natural number.
    Write $n = u p^k$ with $(u,p) = 1$ and $k \geq 0$.
    \begin{enumerate}
    \item $\lambda_{[c],p}(n) = \lambda_{[c],p}(p^k)$.

    \item If $p^2$ does not divide $q$, then
      $\lambda_{[c],p}(p^k) = p^{-k/2}$.
    \item If $p^2$ divides $q$ and $c_p^2 \neq q_p$,
      then
      $\lambda_{[c],p}(p^k) = 1$ if $k = 0$ and vanishes otherwise.
    \item If $p^2$ divides $q$  and $c_p^2 = q_p$,
      then
      $$ \lambda_{[c],p}(p^k) ^2= \begin{cases}\left(\frac{1 + p^{-1}}{1 - p^{-1} }\right) q_p^\frac12 R_{-k, p} & \text{ if } k>0\\  \left(\frac{1 + p^{-1}}{1 - p^{-1} }\right) \left(q_p^\frac12 R_{0, p} - \frac1{p+1}\right) & \text{ if } k=0  . \end{cases}$$

    \end{enumerate}

  \end{proposition}

  By the formulas for $R_{-k,p}$ from
  Section~\ref{sec:local-calculations},
  we deduce immediately:
  \begin{corollary}\label{corollaryfourierbound}
    For each prime $p$ for which $p ^2 $ divides $q$,
    each positive divisor $c$ of $q$, and each nonnegative integer $k$,
    we have $$\lambda_{[c],p}(p^k) \ll p^{k/4}$$ with an absolute implied constant.
  \end{corollary}

  \begin{remark}\label{rmk:failure-of-deligne-bound}
    In general, one cannot hope to improve upon the above
    inequality in the range  $0 \le k \le n-N$ where the integer
    $N$ is such that $p^N = C_p$, the $p$-part of the conductor of
    $f \times f$. This is clear from the formulas for $R_m$ from
    Section~\ref{sec:local-calculations}. In particular, the
    ``Deligne bound"
    $|\lambda_j (p^k)| \leq \tau(p^k)$ does not hold in general.
  \end{remark}

  \begin{proof}[Proof of Proposition~\ref{keypropositionfourier}]
    Part (1) follows immediately from the definition of
    $\lambda_{[c], p}$.
    Part (2) follows from standard formulas for the local
    Whittaker function attached to a Steinberg representation
    (see~\cite[Lemma 2.1]{MR1145805}).

    We now turn to (3) and (4).
    \emph{To simplify notation,
      we restrict henceforth to the
      case that $q$ and $c$ are powers of $p$};
    the general case then follows by the observation of
    Remark \ref{rmk:lambda-c-factorization}.
    The proofs of (3) and (4)
    will each make use of the following consequence
    of
    the support condition on $W_p$ established in
    Lemma \ref{keywhittakersupportlemma}
    and the $\GL_2(\mathbb{Q}_p)$-invariance
    of the Whittaker inner product:
    for each
    $v \in \mathbb{Z}$
    and $x \in \mathbb{Q}_p$ with $|x|^2 < q$,
    we have
    \begin{equation}\label{eq:1}
      \int _{u \in \mathbb{Z}_p^\times }
      |W_p|^2\left(a ({u p^{v}}/{q}) w n(x)\right)
      \, d^\times u
      = \delta_{v}
      := \begin{cases}
        1 & v = 0, \\
        0 & v \neq 0.
      \end{cases}
    \end{equation}

    For part (3), suppose that
    $p^2 | q$ and $c^2 \neq q$.
    By the ``functional equation''
    $\lambda_{[c],q}(p^k)= \lambda_{[q/c],p}(p^k)$
    of Lemma \ref{lem:lambda-atkin-lehner-duh},
    we may assume without loss of generality
    that $c^2$ (properly) divides $q$.
    Then
    $[q,c^2] = q$,
    so
    $$
    \frac{
      \lambda_{[c], p}(p^k)^2
    }
    {
      p^{k}
    }
    =  \int_{\substack{u \in \Z_p^\times  }}
    |W_p|^2\left(a ( {u p^k}/{q} ) w n(1/c)\right)
    d^\times u. $$
    Since $|(1/c)|_p^2 = c^2 < q$,
    the identity \eqref{eq:1}
    implies $\lambda_{[c],p}(p^k)^2 = p^k \delta_k
    = \delta_k$, as desired.

    It remains to consider part (4), in which $c^2 = q$.
    By definition (see \eqref{defrm}),
    $R_{-k,p}$ is the coefficient of $p^{-ks}$ in $J_p(s)$
    (see \eqref{RS-integral-defn3}).
    By writing the $p$-adic integral in \eqref{RS-integral-defn3}
    as a sum
    and invoking the right invariance
    of $W_p$,
    we obtain
    \begin{align*}
      \frac{\zeta_p(1)}{\zeta_p(2)}
      R_{-k,p}
      &=
      \mathop{
        \int _{x \in \mathbb{Q}_p} \int _{y \in
          \mathbb{Q}_p^\times }
      }_{|y|/\max(1,|x|)^2 = p^{-k}}
      |y|^{-1} |W_p|^2 \left(a(y) w n(x)\right) \, d x \, d ^\times y \\
      &=
      \sum_{t=0}^\infty
      v_t p^{k-2t} \int_{u \in \Z_p^\times}
      |W_p|^2\left(a(p^{k-2t})wn(u p^{-t})\right) d^\times u,
    \end{align*}
    where
    $v_0=
    \vol(\mathbb{Z}_p, d x) = 1$
    and $v_t = \text{vol}(p^{-t}
    \Z_p^\times, dx)
    = p^t (1-p^{-1})$ for $t \geq 1$;
    the measures here are normalized as in
    Section~\ref{sec:2-notations}.
    Set $q = p^n$.
    By the right-$a(\mathbb{Z}_p)^\times$-invariance
    of $W_p$, the inner integral may be written as
    \begin{equation}\label{eq:6}
      \int_{u \in \Z_p^\times}
      |W_p|^2\left(a(u p^{k-2t+n}/q)wn(p^{-t})\right) d^\times u.
    \end{equation}
    We consider separately several cases:
    \begin{itemize}
    \item     If $t < n/2$, then $|p^{-t}|_p^2 < q$
      and $k-2t+n > 0$, so \eqref{eq:1} implies that \eqref{eq:6}
      vanishes.
    \item
      If $t > n$, then
      the identity
      $w n(x) \equiv  n(-x^{-1}) a(x^{-2}) n_-(x^{-1})
      \pmod{Z(\mathbb{Q}_p)}$,
      where $n_-(x^{-1}) = \left(
        \begin{smallmatrix}
          1&\\
          x^{-1}&1
        \end{smallmatrix}
      \right)$,
      shows that
      \begin{align*}
        W_p(a(u p^{k-2t+n}/q)wn(p^{-t}))
        &=
        W_p(n( -x^{-1}u p^{k+n}/q) a(u p^{k+n}/q) n_-(p^{t})) \\
        &=
        W_p(a(u p^{k+n}/q))
        = \delta_{k}.
      \end{align*}
      Thus the integral \eqref{eq:6} is $\delta_{k}$.
    \item
      For $n/2 \leq t \leq n$,
      the definition \eqref{lambdaewhittaker}
      specializes to
      \[
      \lambda_{[p^t],p}(p^k)^2
      = p^{k}
      \int _{u \in \mathbb{Z}_p^\times }
      |W_p|^2 \left( a(u p^{k-2t}) w n(p^{-t}) \right) \, d^\times u.
      \]
      This shows that \eqref{eq:6} equals
      $p^{-k} \lambda_{[p^t],p}(p^k)^2$.
      If the lower inequality is strict, i.e., if $t > n/2$,
      then the proof given above of part (3) of the present
      proposition
      shows moreover that $\lambda_{[p^t],p}(p^k)^2 = \delta_k$.
    \item Combining the previous two cases, we see
      for $t > n/2$
      that \eqref{eq:6}
      equals $\delta_k$.
    \end{itemize}
    Collecting together the above calculations,
    we deduce that
    \[
    \frac{\zeta_p(1)}{\zeta_p(2)}R_{-k,p}
    =
    v_{n/2}
    p^{-n}
    \lambda_{[p^{n/2}],p}(p^k)^2
    +
    \delta_k
    \sum_{t > n/2}
    v_t
    p^{-2t}.
    \]
    Rearranging, recalling that
    that $v_t = \frac{p^t}{\zeta_p(1)}$ (for $t \geq 1$),
    and summing some geometric series,
    we arrive at
    \begin{align*}
      \lambda_{[p^{n/2}],p}(p^k)^2
      &=
      \frac{p^n}{
        v_{n/2}
      }
      \frac{\zeta_p(1)}{\zeta_p(2)}R_{-k,p}
      -
      \frac{
        p^n
      }{
        v_{n/2}
      }
      \delta_k
      \sum_{t > n/2}
      v_t
      p^{-2t} \\
      &=
      p^{n/2}
      \frac{\zeta_p(1)^2}{\zeta_p(2)}      R_{-k,p}
      - \delta_k
      p^{n/2}
      \sum_{t>n/2} p^{-t} \\
      &=
      p^{n/2}
      \frac{1+p^{-1}}{1-p^{-1}}
      R_{-k,p}
      - \delta_k
      \frac{p^{-1}}{1-p^{-1}},
    \end{align*}
    which is equivalent to the claimed formula.

  \end{proof}

  \begin{remark}\label{rmk:brunault}
    It is
    instructive
    to apply Proposition~\ref{keypropositionfourier}
    when $f$ is associated to an elliptic  curve
    $E_{/  \Q}$ of conductor $q$.
    In that case, we have  $k=2$ and $\lambda(n) \sqrt{n} \in
    \Z$.
    Since $\mathrm{Aut}(\C)$ acts transitively on the set of cusps of given denominator,
    Proposition~\ref{keypropositionfourier} provides a characterization of the cusps
    at which the differential form $f(z)dz$ vanishes,
    complementing some recent work of Brunault~\cite{brunault}.
    With further work, one may derive from Proposition~\ref{keypropositionfourier} an exact formula
    for the ramification index
    at
    a given cusp
    of the modular parametrization $X_0(q) \rightarrow E$.
    The resulting formula turns out to depend only on the reduction modulo certain powers of $2$ and $3$ of the coefficients of the minimal Weierstrass equation for $E$.
  \end{remark}

  \begin{remark}
    \label{rmk:direct-iwasawa}
    One may extend $\lambda_{[c],p}$ to a function
    on $\mathbb{Q}_p^\times$
    via the formula in its original definition \eqref{lambdaedef},
    and then $\lambda_{[c]} : \mathbb{N} \rightarrow
    \mathbb{R}_{\geq 0}$
    to a function $\lambda_{[c]} : \hat{\mathbb{Q} } ^\times
    \rightarrow \mathbb{R}_{\geq 0}$
    via
    $(y_p)_{p} \mapsto \prod \lambda_{[c],p}(y_p)$.
    Then by directly evaluating $J_p(s)$ in the Iwasawa decomposition,
    one obtains
    \[
    J_f(s)
    = \prod_{p|q} J_p(s)
    = \frac{1}{[\Gamma:\Gamma ']}
    \int _{y \in \prod_{p|q} \mathbb{Q}_p^\times }
    |y|_{\mathbb{A}}^s
    \sum _{c | q} [q/c^2,1]^s \varphi((q/c,c)) \lambda_{[c]}(y)^2
    \, d ^\times y.
    \]
    Suppose now that $q = p^{2 m}$ is a prime power with even
    exponent.
    Then the support condition (by Proposition \ref{keypropositionfourier}) that $\lambda_{[c],p}(p^k) = 0$
    unless $k = 0$ or $c = p^m$
    implies
    \[
    J_f(s)
    =
    \frac{p^{2 m (s-1)}}{1+1/p}
    \sum_{0 \leq t \leq m - 1}
    \frac{\varphi(p^t)}{p^{2 t s}}
    +
    \frac{p ^{- m - 1}}{1+1/p}
    + p^{-m} \frac{1 - 1/p}{1 + 1/p}
    \sum_{k \geq 0}
    \frac{\lambda_{[p^m],p}(p^k)^2}{p^{k s}}.
    \]
    Thus the ``local Lindel\"{o}f bound''
    in the form
    $J_f(s) \ll m p^{-m}$ ($\Re(s) = 1/2$)
    is ``equivalent'' to the estimate
    $
    \sum_{k \geq 0}
    \lambda_{[p^m],p}(p^k)^2/p^{k/2}
    \ll m
    $
    for the sum of the mean squares of the Fourier coefficients
    of $f$ at the cusps of $\Gamma_0(p^{2m})$ with denominator
    $p^m$.
    When the representation $\pi$ of $\PGL_2(\mathbb{Q}_p)$
    generated by $f$ is
    supercuspidal, we note that the identity
    \eqref{eq:5}
    implies the cute formula
    \[
    \frac{
      \sum_{C(\mu)^2 = C(\pi)}
      \ (C(\pi \mu)/C(\pi))^{s-1}
    }
    {
      \sum_{C(\mu)^2 = C(\pi)}
      \ 1
    }
    =
    \sum_{k \geq 0}
    \frac{\lambda_{[p^m],p}(p^k)^2}{p^{k s}}
    \]
    for the ``moments'' of
    $\{\mu : C(\mu)^2 = C(\pi) \} \ni \mu \mapsto C(\pi \mu)$
    (see Section \ref{sec:sketch-proof} for notation).
  \end{remark}

  \subsection{Proof of Theorem \ref{thm:main-refined},
    modulo technicalities}
  \label{sec:outline-proof}
  In this section
  we follow Holowinsky \cite{MR2680498}
  in bounding $D_f(\phi)$
  in terms of shifted convolution sums,
  to which we apply an extension
  (Proposition
  \ref{prop:bounds-shifted-sums-fourier-coefs-generic-cusps})
  of a refinement \cite[Thm 3.10]{PDN-HQUE-LEVEL}
  of his bounds for such sums \cite[Thm 2]{MR2680498}.
  By combining with the bounds obtained in Section
  \ref{sec:bound-d_fphi-terms} and Section \ref{sec:cusps-gamm-four},
  we deduce Theorem \ref{thm:main-refined}.

  Let $Y \geq 1$ be a parameter (to be chosen later),
  and let $h \in C_c^\infty(\mathbb{R}_+^\times)$ be an everywhere
  nonnegative test function with Mellin transform
  $h^\wedge(s) = \int_0^\infty h(y) y^{-s-1} d y$
  such that $h^\wedge(1) = \mu(1)$.
  The proof of
  \cite[Lem 3.4]{PDN-HQUE-LEVEL}
  shows without modification that
  \begin{equation}\label{eq:expand-it-up}
    Y \mu_f(\phi)
    =
    \sum_{\mathfrak{a}_j \in \mathcal{C}} \int_{y_j=0}^{\infty} h(Y w_j y_j)
    \int_{x_j=0}^1 \phi(w_j z_j) |f|^2(z) y^k \, \frac{d x_j \,
      d y_j}{y_j^2}
    + O_\phi(Y^{1/2} \mu_f(1)).
  \end{equation}
  Let \[
  I_\phi(l,n,x)
  = (m n)^{-1/2} \int_{y=0}^\infty
  h(x y) \kappa_\phi(l y) \kappa_f(m y) \kappa_f(n y)
  \frac{d y}{y^2},
  \quad
  m := n + l,
  \]
  where
  $\kappa_\phi$ and $\kappa_f$
  are as in Section \ref{sec:modular-forms-their}.
Write $w_c := [q/c^2, 1]$ for all $c|q$.
By inserting Fourier expansions
and applying
some trivial bounds as in
\cite[Lem 3.8]{PDN-HQUE-LEVEL},
we obtain
\begin{equation}\label{eq:19}
  \begin{split}
    D_f(\phi)
    &=
    \frac{1}{Y \mu_f(1)}
    \sum_{\substack{
        l \in \mathbb{Z}_{\neq 0} \\
        |l| < Y^{1+\eps}
      }
    }
    \frac{\lambda_\phi(l)}{ \sqrt{|l|}}
    \sum_{j}
    \left(
      \sum _{\substack{
          n \in \mathbb{N} \\
          m := n + w_j l \in \mathbb{N}
        }
      }
      \lambda_j(m) \lambda_j(n)
      I_\phi(w_j l,n,Y w_j)
    \right)
    + O_{\phi,\eps}(Y^{-1/2}), \\
    &=
    \frac{1}{Y \mu_f(1)}
    \sum_{\substack{
        l \in \mathbb{Z}_{\neq 0} \\
        |l| < Y^{1+\eps}
      }
    }
    \frac{\lambda_\phi(l)}{ \sqrt{|l|}} \sum_{c|q} I_\phi({w_c} l,n,Y {w_c})  \sum _{\substack{
        n \in \mathbb{N} \\
        m := n + {w_c} l \in \mathbb{N}
      }
    } \left(\sum_{\mathfrak{a}_j \in \mathcal{C}[c]}
      \lambda_j(m) \lambda_j(n) \right)
    + O_{\phi,\eps}(Y^{-1/2}).
  \end{split}
\end{equation}
By Cauchy--Schwarz,
we deduce that
\begin{equation}\label{eq:reduce-to-shifted-sums}
  \begin{split}
    \left\lvert D_f(\phi) \right\rvert
    &\leq
    \frac{1}{Y \mu_f(1)}
    \sum_{\substack{
        l \in \mathbb{Z}_{\neq 0} \\
        |l| < Y^{1+\eps}
      }
    }
    \frac{
      \left\lvert \lambda_\phi(l) \right\rvert
    }{
      \sqrt{|l|}
    }
    \sum_{c|q} \# \mathcal{C}[c]
    \left\lvert
      I_\phi({w_c} l,n,Y {w_c})
    \right\rvert
    \sum _{\substack{
        n \in \mathbb{N} \\
        m := n + {w_c} l \in \mathbb{N}
      }
    }  \lambda_{[c]}(m) \lambda_{[c]}(n)  \\
    &+
    O_{\phi,\eps}(Y^{-1/2}).
  \end{split}
\end{equation}
The weight $I_\phi(w_c l, n, Y w_c)$
essentially restricts the sum to $\max(m,n) \ll Y
k w_c$: indeed, \cite[Lemma 3.12]{PDN-HQUE-LEVEL}
asserts (in slightly different notation) that
\[
I_\phi(l,n,x)
\ll_A
\frac{\Gamma(k-1)}{(4 \pi)^{k-1}}
\cdot \max
\left( 1, \frac{\max(m,n)}{x k} \right) ^{- A}
\]
for every $A > 0$.

In Section \ref{sec:technical-lemmas},
we prove the following:

\begin{proposition}
  \label{prop:bounds-shifted-sums-fourier-coefs-generic-cusps}
  For $l \in \mathbb{Z}_{\neq 0}$, $x \in \mathbb{R}_{\geq 1}$,
  $\eps \in (0,1)$
  and each positive divisor $c$ of $q$, we have
  \[
  \sum _{\substack{
      n \in \mathbb{N} \\
      m := n + l \in \mathbb{N} \\
      \max(m,n) \leq x
    }
  }
  | \lambda_{[c]}(m) \lambda_{[c]}(n) |
  \ll_\eps
  q_\diamond^\eps \log \log(e^e q)^{O(1)}
  \frac{
    x
    \prod_{p
      \leq x}
    ( 1+ 2 |\lambda_f(p)|/p )
  }{ \log(e x)^{2-\eps} }.
  \]
\end{proposition}

Inserting this bound into
\eqref{eq:reduce-to-shifted-sums},
summing dyadically (or by parts)
as in \cite[Proof of Cor 3.14]{PDN-HQUE-LEVEL},
applying the Rankin--Selberg bound
for $\lambda_\phi(l)$
as in \cite[Lem 3.17]{PDN-HQUE-LEVEL},
invoking the Rankin--Selberg formula
\[
\mu_f(1) \asymp
qk  \frac{\Gamma(k-1)}{(4 \pi)^{k-1}}
L(\ad f,1)
\]
for $\mu_f(1)$,
and pulling it all together as in \cite[Section
3.3]{PDN-HQUE-LEVEL},
we obtain
\begin{equation}\label{eq:34}
  D_f(\phi)
  \ll_{\phi,\eps}
  Y^{-1/2}
  +
  \frac{
    Y^{1/2+\eps}
    \log(q k)^\eps  q_\diamond^\eps
  }{
    q
  }
  \sum_{c|q}
  \frac{[q/c^2, 1] \ \varphi((c, q/c))}{\log([q/c^2, 1] k Y)^{2-\eps}}
  \prod_{p \leq [q/c^2, 1] k Y}
  \left( 1 + \frac{2 \left\lvert \lambda _f (p) \right\rvert}{ p} \right).
\end{equation}
To control the sum over $c$ in \eqref{eq:34},
we apply the following lemma,
whose (technical) proof we defer to Section \ref{sec:technical-lemmas}:
\begin{lemma}
  \label{lem:silly-sums-ext}
  Let $x \geq 2$, $\eps \in (0,1)$, and $q \in \mathbb{N}$.
  Then
  \[
  \sum_{c | q}
  \frac{[q/c^2, 1] \ \varphi((c, q/c)) }{\log([q/c^2, 1] x)^{2-\eps} } \ll
  \frac{q \log \log(e^e q)^{O(1)}}{\log(q x)^{2-\eps} }.
  \]
  with absolute implied constants.
\end{lemma}
Applying this lemma to \eqref{eq:34} gives
\begin{equation}\label{eq:35}
  D_f(\phi)
  \ll_{\phi,\eps}
  Y^{-1/2}
  +
  \frac{
    Y^{1/2+\eps}  q_\diamond^\eps
  }{
    \log(q k)^{2-\eps}
  }
  \prod_{p \leq q k Y}
  \left( 1 + \frac{2 \left\lvert \lambda _f (p) \right\rvert}{ p} \right).
\end{equation}
The partial product over $q k < p \leq q k Y$ contributes
negligibly,
so choosing $Y$ suitably as in~\cite{MR2680498} yields
\begin{equation}\label{eq:36}
  D_f(\phi)
  \ll_{\phi,\eps}
  \log(q k)^\eps  q_\diamond^\eps M_f(q k)^{1/2},
\end{equation}
where
\[
M_f(x)
= \frac{\prod_{p \leq x} (1 + 2 |\lambda_f(p)|/p)}{\log(e x)^2
  L(\ad f,1)}.
\]
Feeding \eqref{eq:yay-bound1} and \eqref{eq:36}
into the recipe of \cite[Section 5]{PDN-HQUE-LEVEL}
gives the following result.

\begin{theorem}
  \label{thm:main2}
  Fix a Maass cusp form or incomplete Eisenstein series
  $\phi$
  on $Y_0(1)$.
  Then for a holomorphic newform $f$ of weight $k \in
  2\mathbb{N}$
  on $\Gamma_0(q)$, $q \in \mathbb{N}$, we have
  \[
  D_f(\phi)
  \ll_{\phi,\eps} \log(q k)^\eps
  \min
  \left\{
    \frac{
      (
      q / \sqrt{C}
      )
      ^{- 1 + 2 \theta + \eps }
    }{
      \log(k C)^\delta
      L(\ad f, 1)
    },
    q_\diamond^\eps \log (q k ) ^{1/12}
    L (\ad f, 1) ^{1/4}
  \right\}.
  \]
  Here $\eps > 0$ is arbitrary,
  $\ad f$ is the adjoint lift of $f$,
  $C$ is the (finite) conductor of $\ad f$,
  $\theta \in [0,7/64]$ is a bound towards the Ramanujan
  conjecture for $\phi$ at primes dividing
  $q$ (take $\theta = 0$ if $\phi$ is incomplete Eisenstein),
  and $\delta = 1/2$ or $1$ according as
  $\phi$ is cuspidal or incomplete Eisenstein.
\end{theorem}

When $q$ is squarefree,
one has
$q / \sqrt{C} = 1$,
and Theorem \ref{thm:main2} recovers
a statement appearing on the final page of
\cite{PDN-HQUE-LEVEL} from which the main result of that paper,
the squarefree case of Theorem \ref{thm:main},
is deduced in a straightforward manner.
In general, Proposition \ref{propnN}
implies that $C$ is a square integer
satisfying $C \leq q q_0$,
where $q_0$ is the largest squarefree divisor of $q$.
From this one deduces Theorem \ref{thm:main-refined} by considering
separately the cases that $L(\ad f, 1)$ is large and small,
as in \cite[Section 3]{holowinsky-soundararajan-2008}.

\subsection{Proof that Theorem \ref{thm:main-refined} implies  Theorem~\ref{thm:main}}
\label{sec:proof-thm-1}
We explain briefly how Theorem~\ref{thm:main} follows from
Theorem~\ref{thm:main-refined}.
It's known that the class $C_c(Y_0(1))$ of \emph{compactly supported}
continuous functions on $Y_0(1)$
is contained in the uniform span of the Maass eigencuspforms
and incomplete Eisenstein series
(see~\cite{MR1942691}).
Fix a bounded continuous function $\phi$ on $Y_0(1)$.
Let $\eps > 0$ be arbitrary.
  Choose $T = T(\eps)$ large enough
  that
  the ball
  $B_T := \{z \in \mathbb{H} : -1/2 \le \Re(z) \le 1/2, \
  \Im(z)>T\}$
  has normalized volume $\mu(B_T)/\mu(1) < \eps$.
  Write $\phi = \phi_1 + \phi_2$,
  where $\phi_1 \in C_c(Y_0(1))$
  and $\phi_2$ is supported on $B_T$.
  Because $\phi_1$ can be uniformly approximated by Maass eigencuspforms and incomplete Eisenstein series,
  and because the the collection of maps $D_f(\cdot)$
  is equicontinuous for the uniform topology,
  Theorem~\ref{thm:main-refined} implies that $|D_f(\phi_1)| <
  \eps$ eventually.\footnote{Here and in what follows,
    ``eventually'' means ``provided that $q k$ large enough''.}
  Choose a smooth $[0,1]$-valued
  function $h$ supported on the complement of $B_T$
  in $Y_0(1)$
  that satisfies $\mu(h)/\mu(1) > 1 - 2 \eps$.
  Theorem~\ref{thm:main-refined} implies that
  the positive real number $\mu_f(h)/\mu_f(1)$
  eventually exceeds
  $1 - 3 \eps$.
  By the nonnegativity of $\mu_f$,
  we deduce that $\mu_f(B_T)/\mu_f(1) < 3 \eps$ eventually.
  Let $R$ be the supremum of $|\phi|$.
  Then $|\mu_f(\phi_2)/\mu_f(1)| \le R \mu_f(B_T)/\mu_f(1) \leq
  3 R \eps$
  eventually and $|\mu(\phi_2)/\mu(1)| \le R \eps$, so that
  $|D_f(\phi_2)| \leq
  4 R \eps$ eventually.
  Thus $|D_f(\phi)| < (1 + 4 R) \eps$ eventually.
  Letting $\eps \rightarrow 0$, we obtain
  Theorem \ref{thm:main}.
\subsection{Technical arguments}
\label{sec:technical-lemmas}

\begin{proof}[Proof of Proposition \ref{prop:bounds-shifted-sums-fourier-coefs-generic-cusps}]
  The proof extends that of \cite[Theorem 3.10]{PDN-HQUE-LEVEL},
  which in turn refines \cite[Theorem 2]{MR2680498}.

  We may assume $1 \leq l \leq x$.
  Fix $\alpha \in (0,1/2)$ and set
  $y = x^\alpha$, $s = \alpha \log \log(x)$,
  $z = x^{1/s}$.
  If $x \gg_\alpha 1$ then $10 \leq z \leq y \leq x$,
  as we henceforth assume.
  Define finite sets of primes
  \[
  \mathcal{P} = \{p \leq z, p \nmid q\},
  \quad
  \mathcal{P} ' = \{p \leq z\} \cup \{p \mid q\}.
  \]
  For each set $S$ of primes,
  define the \emph{$S$-part} of a positive integer
  $n$, denoted $n_S$, to be its greatest positive divisor composed
  entirely of primes in $S$.
  We henceforth use the symbol $m$ to denote
  $n + l$.
  By the Cauchy--Schwarz inequality,
  we may bound the contribution to
  the main sum coming from
  those terms for which the $\mathcal{P}'$-part of $m$ or of $n$
  is $>y$
  by
  \[
  \sum_{
    \substack{
      \max(m,n) \leq x \\
      \max(m_{\mathcal{P} '}, n_{\mathcal{P} '}) > y
    }
  }
   \lambda_{[c]}(m) \lambda_{[c]}(n) \leq
  2
  x
  \left( \sum _{m \leq x }
    \frac{\left\lvert \lambda_{[c]} (m) \right\rvert ^2 }{m}
  \right) ^{1/2}
  \left( \sum _{\substack{n \leq x \\ n_{\mathcal{P} '} > y }}
    \frac{\left\lvert \lambda_{[c]} (n) \right\rvert ^2 }{n}
  \right) ^{1/2}.
  \]
 By Proposition~\ref{keypropositionfourier} and Corollary~\ref{corollaryfourierbound}, we have
  \[
  \left( \sum _{m \leq x }
    \frac{\left\lvert \lambda_{[c]} (m) \right\rvert ^2 }{m}
  \right) \le \left(\prod_{p|q}\sum_{k=0}^\infty
  \frac{ \lambda_{[c],p} (p^k)   ^2 }{ p^k} \right)
  \sum_{m \leq x}
  \frac{\left\lvert \lambda (m ) \right\rvert ^2 }{m}
    \ll q_\diamond^\eps \log(x)^3,
  \]
 and
  \begin{equation}\label{eq:28}
   \left( \sum _{\substack{n \leq x \\ n_{\mathcal{P} '} > y }}
    \frac{\left\lvert \lambda_{[c]} (n) \right\rvert ^2 }{n}
  \right) \le \left(\prod_{p|q}\sum_{k=0}^\infty
  \frac{ \lambda_{[c],p} (p^k)   ^2 }{ p^k} \right)
    \sup_{d | q^{\infty}}
    \sum_{
      \substack{
        n \leq x/d \\
        n_{\mathcal{P}} > y/d
      }
    }
    \frac{\left\lvert \lambda (n ) \right\rvert ^2 }{n}
    \ll
     q_\diamond^\eps\sup_{d | q^{\infty}}
    \sum_{
      \substack{
        n \leq x/d \\
        n_{\mathcal{P}} > y/d
      }
    }
    \frac{\left\lvert \lambda (n ) \right\rvert ^2 }{n}.
  \end{equation}
  To bound the RHS of \eqref{eq:28}, we
  consider separately the ranges
  $d > y ^{1/2}$ and $d \leq  y ^{1/2}$.
  If $d > y^{1/2}$, then
  $
  \sum_{
    \substack{
      n \leq x/d \\
      n_{\mathcal{P}} > y/d
    }
  }
  \frac{\left\lvert \lambda (n ) \right\rvert ^2 }{n}
  \ll x^{1-\alpha/4}
  $
  thanks to, say,
  the Deligne bound $|\lambda(n)| \leq \tau(n)$.
  If $d \leq y^{1/2}$, we apply Cauchy--Schwarz,
  the Deligne bound, and the estimate
  $
  \sum_{
    \substack{
      n \leq x \\
      n_{\mathcal{P}} > y^{1/2}
    }
  }
  1
  \ll_{A,\alpha} \frac{x}{ \log(x)^A}
  \quad \text{ for every } A > 0
  $
  which follows from a theorem of Krause \cite{MR1078363}
  (see the discussion in \cite[Proof of Lem 6.3]{PDN-HMQUE})
  to deduce that
  $
  \sum_{
    \substack{
      n \leq x/d \\
      n_{\mathcal{P}} > y/d
    }
  }
  \frac{\left\lvert \lambda (n ) \right\rvert ^2 }{n}
  \ll
  \frac
  {
    x
  }
  {
    \log(x)^{A}
  }
  $ (for a different value of $A$).
  Combining these estimates, we obtain
  \[
  \sum_{\max(m,n) \leq x}
   \lambda_{[c]}(m) \lambda_{[c]}(n)
  \leq
  \sum_{
    \substack{
      \max(m,n) \leq x \\
      \max(m_{\mathcal{P} '}, n_{\mathcal{P} '}) \leq y
    }
  }
   \lambda_{[c]}(m) \lambda_{[c]}(n)
  +  O \left(  \frac{q_\diamond^\eps x}{\log(x)^A} \right).
  \]
  To treat the remaining sum,
  we follow \cite{MR2680498}
  in partitioning it according
  to the values $m_{\mathcal{P} '}$ and $n_{\mathcal{P} '}$.
  Specifically, for $a,b,d \in \mathbb{N}$
  with $(a,b) = 1$ and $d | l$,
  let $\mathbb{N}_{a b d}$
  denote the set of all $n \in \mathbb{N}$
  for which $ad = m_{\mathcal{P} '}$
  and $bd = n_{\mathcal{P} '}$.
  Then $\mathbb{N} = \bigsqcup \mathbb{N}_{a b d}$.
  For $n \in \mathbb{N}_{a b d}$,
  we have
  $\lambda_{[c]}(m) \lambda_{[c]}(n)
  = \left(\prod_{p \in \mathcal{P}'}\lambda_{[c],p}(a d)\right) \left(\prod_{p \in \mathcal{P}'} \lambda_{[c],p}(b d)\right)
  \lambda(m/ad) \lambda(n/bd)$
  because each prime divisor of $q$ is contained
  in $\mathcal{P}'$.
  Recall the notation $\Omega(n) = \sum_{p^\alpha || n} \alpha$
  for the number of prime factors of $n$
  counted with multiplicity.
  Since $|\lambda(n)| \leq \tau(n)$ for all $n \in \mathbb{N}$,
  \[
  \tau \left( \frac{m}{ad} \right)
  = \prod_{p^\alpha || \frac{m}{ad}}
  (\alpha + 1)
  \leq 2^{\Omega(m/ad)},
  \quad \text{ and }
  \Omega(m/ad)
  \leq \frac{\log \left( \frac{m}{a d} \right)}{ \log (z) }
  \leq s,
  \]
  we have $|\lambda(m/ad)| \leq 2^s$.
  Similarly, $|\lambda(n/bd)| \leq 2^s$.  Thus
  \begin{equation}\label{eq:26}
    \sum_{
      \substack{
        \max(m,n) \leq x \\
        \max(m_{\mathcal{P} '}, n_{\mathcal{P} '}) \leq y
      }
    }
    |\lambda_{[c]}(m) \lambda_{[c]}(n)|
    \leq 4^s
    \sum_{d \mid \ell}
      \mathop{\sum_{a \in \mathbb{N}} \sum_{b \in \mathbb{N}}}_{
      \substack{
        (a,b) = 1\\
        \max(a d, b d) \leq y \\
        p | a b d \implies p \in \mathcal{P} ' \\
      }
    }
   \prod_{p \in \mathcal{P}'}\lambda_{[c],p}(a d)\prod_{p \in
     \mathcal{P}'}\lambda_{[c],p}(b d)
    \cdot \#(\mathbb{N}_{a b d} \cap \mathcal{R})
  \end{equation}
  with $\mathcal{R} := [1,x] \cap [1,x-\ell]$.
  The factor
      $4^s$ is negligible
  if $\alpha$ is chosen sufficiently small,
  precisely
  $4^s \ll_\eps \log(x)^\eps$ for  $\alpha \ll_\eps 1$.
  Set $r = a b d ^{-1} l$.
  As in \cite{PDN-HQUE-LEVEL} and \cite{MR2680498},
  the large sieve implies
  \[
  \# (\mathbb{N} _{a b d } \cap \mathcal{R} )
  \ll
  \frac{x/abd + z ^2 }{\sum _{t \leq z} h(t) },
  \]
  where $h(t)$ is supported on squarefree integers $t$,
  multiplicative,
  and given by
  \[
  h(p) = \begin{cases}
    1 &  p \mid r \\
    2 & \text{otherwise}
  \end{cases}
  \]
  on the primes.
  Note that for all $p \leq z$,
  we have
  \[
  h(p) =
  \begin{cases}
    1 &  p \mid r \text{ and } p \leq z \\
    2 & \text{otherwise}
  \end{cases}
  =
  \begin{cases}
    1 &  p \mid r_{\mathcal{P}} \\
    2 & \text{otherwise}
  \end{cases}.
  \]
  It is standard \cite[pp55-59]{MR1836967} that
  \[
  \sum_{t \leq z} h(t)
  \gg
  \frac{\varphi (r_{\mathcal{P}}) }{r_{\mathcal{P}}}
  \log (z) ^2.
  \]
  Since $x + abd z^2 \ll x$,
  $\log (z) \gg \log (x) / \log\log (x) \gg \log (x) ^{1 - \eps }$
  and
  \[
  \frac{\varphi (r_{\mathcal{P}}) }{r_{\mathcal{P}}}
  \gg
  \log\log (x ) ^{-1}
  \log\log (e ^{e} q ) ^{-1} ,
  \]
  we obtain
  \[
  \# (\mathbb{N} _{a b d } \cap \mathcal{R} )
  \ll
  \log\log (e ^{e} q )
  \frac{1}{a b d}
  \frac{
    x
  }
  {
    \log (x)^{2 - \eps }
  }.
  \]
  To complete the proof of the proposition, it now suffices
  to show that
  \begin{equation}\label{eq: lfour}
  \sum_{d \mid \ell}
  \mathop{\sum_{a \in \mathbb{N}} \sum_{b \in \mathbb{N}}}_{
    \substack{
      (a,b) = 1\\
      \max(a d, b d) \leq y \\
      p | a b d \implies p \in \mathcal{P} ' \\
    }
  }
  \frac{
   \prod_{p \in \mathcal{P}'}\lambda_{[c],p}(a d)\prod_{p \in \mathcal{P}'}\lambda_{[c],p}(b d)
  }
  {
    a b d
  }
  \ll
 q_\diamond^\eps
  \log (x) ^{\eps }
  \prod_{p
    \leq x}
  \left(
    1 + \frac{2 |\lambda_f(p)|}{p}
  \right).
  \end{equation}

 Note first that

 $$\mathop{\sum_{a \in \mathbb{N}} \sum_{b \in \mathbb{N}}}_{
    \substack{
      (a,b) = 1\\
      \max(a d, b d) \leq y \\
      p | a b d \implies p \in \mathcal{P} ' \\
    }
  }
  \frac{
   \prod_{p \in \mathcal{P}'}\lambda_{[c],p}(a d)\prod_{p \in \mathcal{P}'}\lambda_{[c],p}(b d)
  }
  {
    a b
  } \le \left(\prod_{\substack{p \le z \\ p\nmid q}} \sum_{k \ge 0} \frac{\lambda(p^{k + v_p(d)})}{p^k} \right)^2 \left(\prod_{\substack{p| q}} \sum_{k \ge 0} \frac{\lambda_{[c],p}(p^{k + v_p(d)})}{p^k} \right)^2
  $$

If $p \nmid q$, then the arguments of~\cite[Proof of Thm. 3.10]{PDN-HQUE-LEVEL} show that
$\sum_{k \ge 0} \frac{\lambda(p^{k + v})}{p^k} \le 3v +3$ if $v \ge 1$ and
$\sum_{k \ge 0} \frac{\lambda(p^{k })}{p^k} \le \left(1+ \frac{\lambda(p)}{p} \right)\left(1+ \frac{20}{p^2}\right).$
If $p|q$ but $c_p^2 \neq q_p$, we have uniformly
$\sum_{k \ge 0} \frac{\lambda_{[c],p}(p^{k + v})}{p^k} \le \frac1{1- p^{-3/2}}. $
Finally, if $p^2 | q$ and $c_p^2 = q_p$, then Corollary~\ref{corollaryfourierbound} shows that $\sum_{k \ge 0} \frac{\lambda_{[c],p}(p^{k + v})}{p^k} \ll p^{\frac{v}{4}}$
 where the implied constant is absolute.
 Putting all this together, and arguing exactly as in~\cite[Proof of Thm. 3.10]{PDN-HQUE-LEVEL}, we see that the LHS of~\eqref{eq: lfour} is bounded by an absolute constant multiple of $$\log(x)^{\eps}\prod_{p
    \leq x}
  \left(
    1 + \frac{2 |\lambda_f(p)|}{p}
  \right) \prod_{p | q_\diamond}O(1)$$

  Since $ \prod_{p | q_\diamond}O(1) \ll q_\diamond^\eps$ this completes the proof.

\end{proof}

\begin{proof}[Proof of Lemma \ref{lem:silly-sums-ext}]
  This lemma generalizes the bound
  \begin{equation}\label{eq:silly-sum-original}
    \sum_{d \mid q} \frac{d}{\log(d x)^{2- \eps}}
    \ll \frac{q \log \log(e^e q)}{\log (q x)^{2-\eps}}
  \end{equation}
  proved in \cite[Lem 3.5]{PDN-HQUE-LEVEL},
  which holds for all squarefree $q$ and all $x \geq 2$, $\eps \in
  (0,1)$,
  with an absolute implied constant.
  The proof of \eqref{eq:silly-sum-original} applies a convexity argument to reduce to the
  case that $q$ is the product of the first $r$ primes,
  partitions the sum according to the number of divisors of $d$,
  and then invokes a weak form of the prime number theorem.  Our strategy here is to reduce the general case to that in which $q$ is squarefree,
  and then  apply the known bound \eqref{eq:silly-sum-original}.

 First, note that
  \[
 \sum_{c| q}
  \frac{[q/c^2, 1] \ \varphi((c, q/c)) }{\log([q/c^2, 1] x)^{2-\eps} }
  =
   \sum_{d\mid q}
  \varphi((d,q/d)) \frac{[d^2/q,1]}{\log([d^2/q,1]k)^{2-\eps}}.
  \]
  Since
  \[
  \varphi((d,q/d)) [d^2/q,1]
  \leq (d,q/d)  [d^2/q,1]
  = d,
  \]
  we see that it suffices to show
  \[
  \sum_{d\mid q}
  \frac{d}{\log([d^2/q,1]k)^{2-\eps}}
  \ll \frac{q \log \log(e^e q)^{O(1)}}{\log(q k)^{2-\eps}}.
  \]

  From here on, the argument is unfortunately a bit technical.
  Let $q_1 < \dotsb < q_r$ be the
  distinct
  prime factors of $q$.
  Define maps $B_i : \{d \in \mathbb{N}  : d \mid q\} \rightarrow \{0,1\}$
  by
  \[
  B_i(d) = \begin{cases}
    0 & (d,q_i^\infty) \mid q_\diamond \\
    1 & \text{otherwise.}
  \end{cases}
  \]
  Thus $B_i(d) = 1$ or $0$ according
  as the valuation of $d$ at $q_i$ does or does not exceed
  half that of $q$.
  Let $B  = \prod B_i : \{d  \in \mathbb{N} : d \mid q \} \rightarrow \{0,1\}^r$
  be the product map that sends
  $d$ to the $r$-tuple $(B_1(d), \dotsc, B_r(d))$.
  For each positive divisor $d = \prod q_i^{\alpha_i}$ of $q$ and each
  $\eta = (\eta_1,\dotsc,\eta_r) \in \{0,1\}^r$,
  write
  $d_\eta = \prod q_i^{\eta_i \alpha_i}$.
  Our reason for introducing this notation is that
  for all $d \in B^{-1}(\eta)$, we have
  $[d^2/q,1] = (d^2/q)_\eta$
  and may write
  $d = {q_\diamond}  (q/{q_\diamond} )_\eta \prod q_i^{-\delta_i}$
  where $\delta_i \geq 0$ for all $i$.
  Thus
  \begin{equation}\label{eq:21}
    \frac{d}{\log(k[d^2/q,1])^{2 - \eps} }
    = {q_\diamond}
    \frac{(q/{q_\diamond} )_\eta }{ \log(k (q/{q_\diamond} )_\eta)^{2-\eps}  }
    \frac{\log(k (q/{q_\diamond} )_\eta)^{2-\eps} }{\log(k
      (d^2/q)_\eta)^{2-\eps} }
    \prod q_i^{-\delta_i}.
  \end{equation}
  Let us now write $q = \prod q_i^{\beta_i}$
  and ${q_\diamond}  = \prod q_i^{\gamma_i}$;
  the definition of ${q_\diamond} $ implies $\gamma_i = \lfloor \beta_i/2 \rfloor$.
  Then
  \[
  \frac{\log(k (q/{q_\diamond} )_\eta)}{\log(k (d^2/q)_\eta)}
  =
  \frac{
    \log(k) +
    \sum \eta _i (\beta _i - \gamma _i ) \log (q _i )
  }{
    \log(k) +
    \sum \eta _i (\beta _i - 2 \delta_i  ) \log (q _i )
  }
  \leq \max_{i:\eta_i=1}
  \frac{\beta_i}{ \beta_i - 2 \delta_i}
  \leq \prod_{i:\eta_i=1} \frac{1}{1 - 2\delta_i/\beta_i}.
  \]
  (In the above, define an empty maximum or an empty product to be $1$.)
  By comparing the sum to an integral, one shows easily that
  \[
  \sum _{0 \leq \delta _i < \beta _i /2}
  \frac{q_i^{-\delta_i}}{ \left( 1 - 2\delta _i /\beta_i
    \right)^{2-\eps}}
  \leq 1 + \frac{9 + O(1/\log q_i)}{q_i} \leq 1 + O(1/q_i)
  \leq (1 + 1/q_i)^{O(1)}.
  \]
  with absolute implied constants.
  Since $\prod_i (1 + 1/q_i) \ll \log \log(e^e q)$, we deduce
  from \eqref{eq:21} that
  \begin{equation*}
    \sum _{d \in B^{-1}(\eta)} \frac{d }{ \log ([d^2/q,1])^{2-\eps}}
    \ll {q_\diamond}  \frac{(q/{q_\diamond} )_\eta }{\log(k (q/{q_\diamond} )_\eta )^{2-\eps}} \log\log(e^e q)^{O(1)}.
  \end{equation*}
  To complete the proof of the lemma, it
  suffices now to establish
  that
  \begin{equation}\label{eq:23}
    \sum _{\eta \in \{0, 1\}^r}
    \frac{(q/{q_\diamond} )_\eta }{ \log(k (q/{q_\diamond} )_\eta )^{2-\eps}}
    \ll
    \frac{q/{q_\diamond} }{\log(k (q/{q_\diamond} ))^{2-\eps}} \log \log(e^e q).
  \end{equation}
  As in \cite[Proof of Lem 3.5]{PDN-HQUE-LEVEL},
  define
  $\beta(x) = x / \log(e^e x k)^{2 - \eps}$.
  Then
  $\beta(x) \asymp x / \log(x k)^{2- \eps}$ for all $x \in
  \mathbb{R}_{\geq 1}$,
  so the desired bound \eqref{eq:23} is
  equivalent to
  \[
  \sum_{\eta \in \{0,1\}^r}
  \frac{\beta((q/{q_\diamond} )_\eta )}{\beta(q/{q_\diamond} )} \ll \log \log(e^e q).
  \]
  Since $\beta$ is increasing on $\mathbb{R}_{\geq 1}$
  and the map $\mathbb{R}_{\geq 0} \ni x \mapsto \log
  \beta(e^x)$ is convex,
  we have (compare with \cite[Proof of Lem 3.5]{PDN-HQUE-LEVEL})
  \begin{align*}
    \frac{
      \beta((q/{q_\diamond} )_\eta )
    }{
      \beta(q/{q_\diamond} )
    }
    &=
    \frac{
      \beta(q_1^{\eta_1(\beta_1-\gamma_1)}
      \dotsb q _r^{\eta_r(\beta_r-\gamma_r)})
    }
    {
      \beta(q_1^{\beta_1-\gamma_1} \dotsb q_r^{\beta_r-\gamma_r})
    }
    \leq
    \frac{
      \beta(
      q_1^{\eta_1}
      q_2^{\eta_2(\beta_2-\gamma_2)}
      \dotsb
      q_r^{\eta_r(\beta_r-\gamma_r)}
      )
    }
    {
      \beta(
      q_1
      q_2^{\beta_2-\gamma_2}
      \dotsb
      q_r^{\beta_r-\gamma_r}
      )
    } \\
    &\leq
    \frac{
      \beta(
      q_1^{\eta_1}
      q_2^{\eta_2}
      q_3^{\eta_3(\beta_3-\gamma_3)}
      \dotsb
      q _r^{\eta_r(\beta_r-\gamma_r)}
      )
    }{
      \beta(
      q_1  q_2  q_3^{\beta_3-\gamma_3} \dotsb
      q_r^{\beta_r-\gamma_r}
      )
    }
    \leq  \dotsb \\
    &\leq
    \frac
    {
      \beta(
      q_1^{\eta_1}
      \dotsb
      q _r ^{\eta_r }
      )
    }
    {
      \beta(q_1 \dotsb q _r)
    }
    =
    \frac{
      \beta(\prod q_i^{\eta_i})
    }
    {
      \beta(\prod q_i)
    }.
  \end{align*}
  But $\prod q_i^{\eta_i}$ is squarefree, so \eqref{eq:silly-sum-original} implies
  \[
  \sum_{\eta \in \{0,1\}^r}
  \frac{\beta(\prod q_i^{\eta_i})}{\beta(\prod q_i)}
  = \sum_{d | \prod q_i}
  \frac{\beta(d)}{\beta(\prod q_i)}
  \ll \log \log(e^e \prod q_i)
  \ll \log \log(e^e q),
  \]
  as desired.
\end{proof}

\bibliography{refs-que}

\end{document}